\documentclass[11pt]{article}
\usepackage{amssymb,amsmath,amsfonts,amsthm,mathtools,bm,enumerate,color}
\usepackage{mathrsfs}
\usepackage[toc,page,title,titletoc,header]{appendix}
\usepackage{graphicx}

\usepackage[colorlinks,citecolor=blue,urlcolor=blue]{hyperref}
\usepackage{paralist}

\usepackage{tikz}
\usetikzlibrary{arrows,positioning,shapes.geometric}

\usepackage{algorithm,algorithmic}

\graphicspath{ {./figures/} }

\usepackage{indentfirst}
\usepackage{multicol}
\usepackage{booktabs}
\usepackage{url}
\usepackage[outdir=./]{epstopdf} 

\usepackage{hyperref}
\hypersetup{
    colorlinks=true, 
    linktoc=all,     
    linkcolor=blue,  
}
\numberwithin{equation}{section}

\newtheorem{Theorem}{Theorem}[section]
\newtheorem{Lemma}[Theorem]{Lemma}
\newtheorem{Proposition}[Theorem]{Proposition}
\newtheorem{Assumption}{H.\!\!}

\theoremstyle{definition}

\newtheorem{Example}{Example}[section]

\theoremstyle{remark}
\newtheorem{Remark}{Remark}[section]

 \def\p{\partial} 
\def\to{\rightarrow}
 \def\ol{\overline}    
\def\Om{\Omega}  \def\om{\omega} 
 
\newcommand{\q}{\quad}

\def\l{\label}    \def\fa{\forall}
\def\b{\beta}  \def\a{\alpha} 
 
\def\eps{\varepsilon}

 \def\t{\times}  
\def\ms{\medskip}

\def \la{\langle} \def\ra{\rangle}

\def\cA{\mathcal{A}}

\def\cF{\mathcal{F}}

\def\cH{\mathcal{H}}
\def\cI{\mathcal{I}}

\def\cP{\mathcal{P}}

\def\cS{\mathcal{S}}

\def\cW{\mathcal{W}}
\def\cX{\mathcal{X}}
\def\cY{\mathcal{Y}}

\def\d{{\mathrm{d}}}

\def\bA{{\textbf{A}}}

\def\sB{\mathbb{B}}

\def\sE{{\mathbb{E}}}
\def\sF{{\mathbb{F}}}
\def\sI{{\mathbb{I}}}

\def\sN{{\mathbb{N}}}
\def\sP{\mathbb{P}}

\def\sR{{\mathbb R}}

\DeclareMathOperator*{\argmin}{arg\,min}
\DeclareMathOperator*{\esssup}{ess\,sup}

\newcommand{\lc}
{\mathrel{\raise2pt\hbox{${\mathop<\limits_{\raise1pt\hbox
{\mbox{$\sim$}}}}$}}}

\newcommand{\gc}
{\mathrel{\raise2pt\hbox{${\mathop>\limits_{\raise1pt\hbox{\mbox{$\sim$}}}}$}}}

\newcommand{\ec}
{\mathrel{\raise2pt\hbox{${\mathop=\limits_{\raise1pt\hbox{\mbox{$\sim$}}}}$}}}

\def\bb{\begin{equation}} \def\ee{\end{equation}}
\def\bbn{\begin{equation*}} \def\een{\end{equation*}}

\def\beqn{\begin{eqnarray}}  \def\eqn{\end{eqnarray}}

\def\beqnx{\begin{eqnarray*}} \def\eqnx{\end{eqnarray*}}

\def\bn{\begin{enumerate}} \def\en{\end{enumerate}}

\def\bd{\begin{description}} \def\ed{\end{description}}



\begin{document}

\title{
Optimal regularity of open-loop mean field controls and their piecewise constant approximation
}

\author{
Christoph Reisinger\thanks{
Mathematical Institute, University of Oxford, Oxford OX2 6GG, UK
 ({\tt christoph.reisinger@maths.ox.ac.uk, 
wolfgang.stockinger@maths.ox.ac.uk})}
\and
Wolfgang Stockinger\footnotemark[1]
\and
Yufei Zhang
\thanks{Department of Statistics, London School of Economics and Political Science, Houghton Street, London, WC2A 2AE, UK
({\tt y.zhang389@lse.ac.uk})}
}
\date{}

\maketitle

\noindent\textbf{Abstract.} 
We consider the control of McKean-Vlasov dynamics whose coefficients have mean field interactions in the state and control. We show that for a class of linear-convex mean field control  problems, the unique optimal open-loop control admits the optimal $1/2$-H\"{o}lder regularity in time. Consequently, we prove that the value function can be approximated by one with piecewise constant controls and discrete-time state processes arising from Euler-Maruyama time stepping, up to an order $1/2$ error, and the optimal control
can be approximated
 up to an order $1/4$ error. These results are novel even for the case without mean field interaction.


\ms


\medskip


\section{Introduction}\l{sec:intro}

In this paper, we  study  
a class of mean field stochastic control problems where the  state
dynamics  and  cost functions
depend upon the joint law of the  state and the control processes.
%
%
%
 Let 
 $T>0$ be a  given  terminal time,
$(\Om,\cF,\sP)$ be a complete  probability space
on which a  $d$-dimensional Brownian motion $(W_t)_{t\in [0,T]}$
is defined,
 $\sF = (\cF_t)_{t\in [0,T]}$ be the natural filtration of $W$
 augmented with an independent $\sigma$-algebra $\cF_0$,
and 
 $\cA$ be the set of 
 square integrable 
$\sF$-progressively measurable processes
$\a=(\a_t)_{ t\in [0,T]}$ taking values in 
a nonempty closed convex set
$\bA\subset\sR^k$.
For any  initial 
state $\xi_0\in L^2(\cF_0;\sR^n)$
and control $\a\in \cA$,
 we consider  the state process governed by
the following controlled  McKean--Vlasov diffusion:
 $X^\a_0=\xi_0$ and 
\bb\l{eq:mfcE_fwd}
\d X^\a_t=
b(t,X^\a_t,\a_t,\sP_{(X^\a_t,\a_t)})\, \d t
+\sigma(t,X^\a_t,\sP_{X^\a_t})\, \d W_t,
\q t\in [0,T],
\ee
where $b$ and $\sigma$ are given (possibly unbounded) Lipschitz continuous functions 
taking values in $\sR^{n}$ and $\sR^{n\t d}$, respectively,
and $\sigma$ is 
possibly degenerate.
The value  function of the optimal control problem is defined by
\bb\l{eq:mfcE}
V(\xi_0)=\inf_{\a\in \cA} J(\a;\xi_0)
\q
\textnormal{with}
\;
J(\a;\xi_0)=\sE\bigg[
\int_0^T f(t,X^\a_t,\a_t,\sP_{(X^\a_t,\a_t)})\, \d t+g(X^\a_T,\sP_{X^\a_T})
\bigg],
\ee
where 
the running cost $f$ and terminal cost $g$
 are given real valued  functions of   at most  quadratic growth.
Above and hereafter,
$\sP_{U}$ stands for  the  law of a given random variable $U$. 

Such mean field control (MFC) problems 
with   interactions 
  through
 the joint distribution of the state and control processes
have  attracted an increasing interest due to the  emergence of the mean field game
theory and their numerous applications in various areas, including economics,
 biology and social interactions (see e.g.~\cite{
chassagneux2014,carmona2015,
bandini2016, pham2016,
carmona2018a,pham2018,
acciaio2019,carmona2019,gobet2019,
burzoni2020,lauriere2020}). 
In particular,
the solution to \eqref{eq:mfcE} 
describes
 large population equilibria   
 of interacting
individuals who  obey a common policy 
controlled by a   central planner.
Equations of the type \eqref{eq:mfcE} are also motivated 
by
control problems whose objective functions are evaluated under convex risk measures,
such as  the mean-variance portfolio selection problem in finance.
In the case that 
 the coefficients of the state  dynamics are  linear in the
state, control and measure variables, and
the cost functions are convex in these variables,
\eqref{eq:mfcE} is called 
a  linear-convex  MFC problem
and is the main focus of this paper.
Moreover, if 
the mean field interactions
enter  
 both 
the controlled dynamics 
and cost functions 
through the marginal law of the state only,
then \eqref{eq:mfcE} reduces to the MFC problems 
{(without control interactions)}
studied in 
\cite{chassagneux2014,carmona2015,carmona2018a,carmona2019}.

As explicit solutions to \eqref{eq:mfcE} are rarely available,
 numerical schemes for solving such control problems become vital.
A common strategy to obtain numerical approximations of   \eqref{eq:mfcE}  
is
to  discretize the control problem
on a given time grid 
 by using piecewise constant policy timestepping.
More precisely,
for any given 
 partition  $\pi=\{0=t_0<\cdots<t_N=T\}$ of $[0,T]$,
we shall approximate the control problem \eqref{eq:mfcE}
 by the following discrete-time control problem:
 \bb\l{eq:mfcE_constant}
V_\pi(\xi_0)=\inf_{\a\in \cA_\pi} J_\pi(\a;\xi_0),
\ee
where  
$\cA_{\pi}$ is the   subset of controls $\cA$ that are constant on each subinterval $[t_i,t_{i+1})$ in $\pi$:
\begin{align}\l{eq:A_pi}
\cA_{\pi}\coloneqq
\Big\{\alpha\in\cA:  
{\; \forall \omega \in \Omega } \;  \exists a_i\in \bA, \, i=0,\ldots, N-1,  \;  \text{ s.t. }  \alpha_s{(\omega)} \equiv \sum^{N-1}_{i=0} a_i \bm{1}_{[t_i,t_{i+1})}(s)
\Big\},
\end{align}
$J_\pi(\a;\xi_0)$ is the discretized cost functional defined by
\bb\l{eq:J_pi}
J_\pi(\a;\xi_0)\coloneqq\sE\bigg[\sum_{i=0}^{N-1}
\int_{t_i}^{t_{i+1}} f(t_i,X^{\a,\pi}_{t_i},\a_{t_i},\sP_{(X^{\a,\pi}_{t_i},\a_{t_i})})\, \d t+g(X^{\a,\pi}_T,\sP_{X^{\a,\pi}_T})
\bigg],
\ee
and $X^{\a,\pi}$ is  the discretized 
 controlled  process 
defined by the  Euler-Maruyama  approximation of  \eqref{eq:mfcE_fwd}:
$X^{\a,\pi}_0=\xi_0$ and  for all $i\in \{0,\ldots,N-1\}$, $X^{\a,\pi}_t=X^{\a,\pi}_{t_i}\bm{1}_{[t_i,t_{i+1})}(t)$ and
\bb\l{eq:mfcE_fwd_euler}
X^{\a,\pi}_{t_{i+1}}=X^{\a,\pi}_{t_{i}}
+
b(t_i,X^{\a,\pi}_{t_i},\a_{t_i},\sP_{(X^{\a,\pi}_{t_i},\a_{t_i})})\, (t_{i+1}-t_i)
+\sigma(t_i,X^{\a,\pi}_{t_i},\sP_{X^{\a,\pi}_{t_i}})\, (W_{t_{i+1}}-W_{t_{i}}).
\ee
The computational advantage of this approach 
comes from the fact that over
the time intervals in which the policy is constant, 
we only need to deal with Gaussian random variables with 
known mean and variance,
{
which provides 
the basis for 
designing efficient numerical methods 
to solve  MFC problems}. 
We refer the reader to 
\cite{reisinger2016,dumitrescu2019,picarelli2020} for 
piecewise constant policy timestepping
for classical control problems,
and to e.g.~\cite{carmona2019,gu2019,achdou2020} 
and  references  therein
for 
numerical methods for (discrete time) mean field control problems
and mean field games.
Note that in practice, one can  estimate the  marginal laws $\sP_{(X^{\a,\pi}_{t_i},\a_{t_i})}$ in \eqref{eq:mfcE_fwd_euler}  by a particle method (see e.g., \cite{carmona2019}).

Motivated by 
the above applications,
in this paper, we aim to address   to what extent
the continuous-time  MFC problem \eqref{eq:mfcE}
can be   approximated 
by discrete-time control problems \eqref{eq:mfcE_constant} arising from
 piecewise constant policy timestepping. 
In particular, we would answer the following two questions:
\begin{description}
\item[Q1]: What is the convergence rate of $|V_\pi(\xi_0)-V(\xi_0)|$ in terms of the stepsize $|\pi|$?
\item[Q2]: How does a (sub-)optimal control $\hat{a}_\pi$ of \eqref{eq:mfcE_constant}
approximate  the optimal control of $\hat{\a}$ of 
\eqref{eq:mfcE}?
\end{description}

The approximation error of value functions
was first  addressed in \cite{krylov1999}
for classical control problems 
(with controlled diffusion coefficients
but
without mean field interaction), 
in which it is shown that
 the value functions of controlled diffusion processes
(whose  coefficients are 
 Lipschitz continuous in space and $1/2$-H\"{o}lder continuous in time)
  can be
approximated with order $1/6$ error by those with controls which are constant on uniform time
intervals.
The convergence rate was then improved to
order  $1/4$ in \cite{jakobsen2019} under the same regularity assumptions. 
The analysis in \cite{krylov1999,jakobsen2019} combines  stochastic and analytic techniques,
which first estimates the local error 
for each subinterval by 
controlling  the generator of the controlled process, 
and then  aggregates the local error over time
by applying  It\^{o}'s lemma and a dynamic programming principle.
No convergence result for optimal controls 
has been provided.

Unfortunately, the arguments in \cite{krylov1999,jakobsen2019} cannot be  adapted to 
study piecewise constant policy approximation of \eqref{eq:mfcE}, mainly due to the following two reasons.
Firstly, 
controlling the  generator of the state process
 usually involves estimating 
 sup-norms of  high-order derivatives of the value functions,
which in turn requires 
the action set $\bA$ to be compact and 
all coefficients of the control problem  to be uniformly bounded  
(see \cite{krylov1999,jakobsen2019}).
Here, we allow 
the action set $\bA$ to be  unbounded,
the coefficients of \eqref{eq:mfcE_fwd} to be of linear growth,
  and 
 the cost functions of  \eqref{eq:mfcE} to be of quadratic growth,
in order to include the most commonly used linear-quadratic
models. 
Secondly, 
since the value function of a control problem is in general  non-differentiable,
\cite{krylov1999,jakobsen2019} first regularizes the (finite-dimensional) value function
and then balances the regularization error and  time discretization error.
However, it is well-known that 
one has to include the marginal distribution
 in the state of the system to  restore a  dynamic programming principle of \eqref{eq:mfcE}
 (see e.g.~\cite{pham2018}).
 This forces us to deal with an infinite-dimensional  generator
 and an infinite-dimensional value function,
 for which  there is no  known regularization technique 
 with quantifiable regularization error.

In fact, to the best of our knowledge, there is no published work on 
the accuracy of piecewise constant policy approximation 
for MFC problems
with general open-loop controls
 (i.e.,  controls that depend on the initial condition and noise
 as those in $\cA$).
A related work is \cite{carmona2019}, which 
restricts the class of admissible controls to be closed-loop  controls 
(i.e.,  controls that are deterministic functions of time and   state processes)
and analyzes the time discretization error for 
{special cases of 
linear-convex MFC problems \eqref{eq:mfcE} 
in which both $b$ and $f$ are independent of the law of controls}.
By assuming that the optimal feedback map is 
 Lipschitz continuous in time and twice-differentiable  in space with Lipschitz continuous derivatives
 (see Assumptions (B1), (C1)-(C3) in \cite{carmona2019}),
the authors show that 
the value functions of the discrete-time control problems converge to that of the original problem with order $1/2$. 
We remark that 
establishing such a  strong regularity of the feedback map is 
  a delicate and technical issue,
which usually requires
to analyze the classical solutions to an infinite-dimensional 
partial differential equation (PDE)
under the assumption that 
the cost functions are three-times differentiable  with bounded  Lipschitz continuous derivatives
(see  e.g.~\cite{chassagneux2014}
for the case where the diffusion coefficient is   constant  and all coefficients are time-independent).

\paragraph{Our work.}
This paper 
studies 
the time discretization error of
 linear-convex  MFC problems   \eqref{eq:mfcE}.
The main contributions  are:
\begin{itemize}
\item
%
We prove under suitable  conditions,
which are verified for
 different classes of
  MFC  problems
(see Examples \ref{example:mfc}, \ref{example:law_c} and \ref{example:lq}),
  that 
  \eqref{eq:mfcE} admits a unique optimal control, which can be characterized by 
  the unique  H\"{o}lder continuous solution to 
 an asociated coupled MV-FBSDE. 
Based on this solution characterization, 
we prove that 
the unique optimal control of   \eqref{eq:mfcE}
has the optimal time regularity,
which is $1/2$-H\"{o}lder continuous in time in the $L^p$-norm
(see Theorem \ref{TH:ControlRegularity}).
We further give conditions under which the optimal control of \eqref{eq:mfcE}
is deterministic.
Such time regularity results for 
optimal  controls
 are novel
even for the  case without mean field interaction.

\item 
We  estimate the  error  introduced by 
 approximating linear-convex 
  MFC 
 problems
with piecewise constant controls
and 
 Euler-Maruyama  discretizations of state processes.
By using the H\"{o}lder regularity of the optimal control,
we prove that 
the value functions of the discrete-time control problems converge to  the original value function with 
an optimal
order $1/2$,
for which 
we merely require the cost functions to be 
 H\"{o}lder continuous in time and 
 Lipschitz continuously differentiable in space
(see Theorems \ref{thm:PCPT_c_value} and \ref{thm:PCPT_discrete_value}).
We further show that 
the  
optimal controls of these discrete-time control problems
converge 
to the optimal control of  \eqref{eq:mfcE}
in the $\cH^2(\sR^k)$-norm
with an order $1/4$
 (see Theorems \ref{thm:control_PCPT_c} and \ref{thm:control_PCPT_discrete}),
which is the first  result on the convergence order of 
approximate   controls,
even for the  case without mean field interaction.

\end{itemize}

\paragraph{Our approaches.} 

{
Due to the non-Markovian nature of the controlled dynamics, 
instead of adapting the dynamic programming approach  in \cite{krylov1999,jakobsen2019},
we approach the control problem 
by directly characterizing the optimal control of \eqref{eq:mfcE}
via the stochastic  Pontryagin maximum principle.
We shall  investigate 
the solution regularity of the Pontryagin system
via Malliavin Calculus, 
and subsequently deduce the time regularity of 
the open-loop optimal control. 
This 
enables us to
  quantify the time discretization error 
for  MFC problems 
with  Lipschitz differentiable cost functions,
and avoids the strong regularity requirements on 
the optimal feedback control as in  \cite{carmona2019}.

Let us briefly comment on 
the main difficulty encountered in  analyzing the well-posedness and 
 regularity
 of optimal controls to \eqref{eq:mfcE}.
As shown in  \cite[Remark 3.4]{acciaio2019},  
an optimal (open-loop) control $\hat{\a}$ of linear-convex  MFC problems
\eqref{eq:mfcE}
can be 
characterized by its first-order 
optimality condition as follows:
\begin{equation}\label{eq:nonMarkovian_optimal}
\hat{\a}_t=\argmin\lbrace \mathbb{E}[H(t,X_t,\beta,\mathbb{P}_{(X_t,\beta)},Y_t)] \mid \ \beta \in L^2(\mathcal{F}_t;\bA) \rbrace,
\q \textnormal{$ \d \sP\otimes \d t$-a.e.},
\end{equation}
where $H$ is an associated Hamiltonian and $Y$ is the associated adjoint process satisfying 
a \textit{non-Markovian} forward-backward  Pontryagin system. 
However,
in contrast with the classical MFC problems
{without control interactions},
in general one can not express $\hat{\alpha}$ in \eqref{eq:nonMarkovian_optimal}
as 
 $\hat{\a}_t=\psi(t,X_t,Y_t,\sP_{(X_t,Y_t)})$,
with $\psi$ being the pointwise minimizer 
 of the Hamiltonian $H$,
  due to the nonlinear
dependence on the law of the control
(see \cite[Remark 4.2]{acciaio2019} for a counterexample 
in the linear-quadratic setting).
This prevents us from simplifying the non-Markovian Pontryagin system by
directly  inserting the formula for the minimizer of the Hamiltonian
into the forward equation as in \cite{carmona2015,carmona2018a}.
}

We shall 
overcome this difficulty
by showing under various structural conditions on the running costs that, 
the  optimality condition \eqref{eq:nonMarkovian_optimal} can still be  achieved by a Lipschitz function
 $\psi$ from the state and adjoint processes to the action set.
The desired deterministic function $\psi$ 
is constructed either from a \textit{modified Hamiltonian}
(see Examples \ref{example:mfc} and \ref{example:law_c})
or by solving the first-order condition explicitly
(see Example   \ref{example:lq}).
This enables us to 
reduce the non-Markovian Pontryagin system
 to a MV-FBSDE
whose forward equation depends on
  the adjoint processes and their marginal distributions.
We then prove the well-posedness  of the  MV-FBSDE
by adapting the continuation method in \cite{bensoussan2015,carmona2015},
and further establish the H\"{o}lder regularity of the solutions 
by extending the path regularity results for decoupled FBSDEs in \cite{zhang2017}
to the present setting.
This subsequently leads to 
 the desired H\"{o}lder continuity of optimal controls of   MFC problems
(see Theorem   \ref{TH:ControlRegularity}).
{
Our proof for   this  time regularity of the optimal control 
exploits the structural properties of the control problem
(e.g.,  the convexity of the loss functional 
and the uncontrolled diffusion coefficients),
which allows us to establish sharper time discretization errors  
  than existing results for general control problems
 in  \cite{krylov1999,jakobsen2019}.}
Note that our argument does not explicitly use any regularity of the value function,
and the optimal feedback control is merely Lipschitz continuous in the state variable. 

This work is organized as follows.
%
Section \ref{sec:mfcE_mvfbsde}
states the main assumptions 
of the   MFC  problem
and derives the corresponding MV-FBSDE from the stochastic maximum principle.
In Section \ref{sec:regularity_mfcE},
we 
analyze the MV-FBSDE 
and then establish the H\"{o}lder regularity of the optimal control of the   MFC  problem.
We prove the  order $1/2$ convergence of 
 the discrete-time approximation of the value function
 in Section  \ref{sec:conv_PCPT}
 and then the   order $1/4$ convergence of optimizers
 for the discrete-time control problems
  in  Section \ref{sec:PCPT_control}.
The Appendix \ref{appendix:hat_a_existence} 
is devoted to the proofs of some technical results.

\paragraph{Notation.}
We end this section by introducing  some  notations  used throughout  this paper.
%
%
%
For any given $n\in \sN$ and $x\in \sR^n$, we  denote by 
  $\sI_n$  the $n\t n$ identity matrix,
  by ${0}_{n}$ the  zero element
  of  $\sR^{n}$
   and
by
 $\bm{\delta}_{x}$
 the Dirac measure supported at $x$.
We shall denote by $\la \cdot,\cdot\ra$
the usual inner product in a given Euclidean space
and by   $|\cdot|$ the norm induced by $\la \cdot,\cdot\ra$,
which in particular satisfies  for all 
$n,m,d\in \sN$
and
$\theta_1=(x_1,y_1,z_1),\theta_2=(x_2,y_2,z_2)\in \sR^n\t \sR^m\t \sR^{m\t d}$
that
$\la z_1,z_2\ra =\textrm{trace}(z^*_1z_2)$
and 
$\la \theta_1,\theta_2\ra =\la x_1,x_2\ra+\la y_1,y_2\ra+\la z_1,z_2\ra$.

We then introduce several spaces:
for each $p \ge 1$, $k \in \sN$, $t\in [0,T]$
and Euclidean space $(E,|\cdot|)$,
$L^p(\Om; E)$ is the space  of 
 $E$-valued
$\cF$-measurable
random variables $X$ satisfying
$\|X\|_{L^p}=\sE[|X|^p]^{1/p}<\infty$,
and 
$L^p(\cF_t; E)$ is the subspace  of $L^p( \Om;E)$
containing all 
$\cF_t$-measurable
random variables;
$\cS^p(t,T;E)$ is the space of 
$\sF$-progressively  measurable 
processes
$Y: \Om\t [t,T]\to E$ 
satisfying $\|Y\|_{\cS^p}=\sE[\esssup_{s\in [t,T]}|Y_s|^p]^{1/p}<\infty$;
 $\cH^p(t,T; E)$ is the space of 
  $\sF$-progressively measurable
 processes 
$Z: \Om\t [t,T]\to E$  
 satisfying $\|Z\|_{\cH^p}=\sE[(\int_t^T|Z_s|^2\,\d s)^{p/2}]^{1/p}<\infty$.
 For notational simplicity, 
 when $t=0$,
 we often
denote 
  $\cS^p=\cS^p(0,T;E)$
  and $\cH^p=\cH^p(0,T;E)$,
if no confusion occurs.

Moreover, for every  
Euclidean space $(E, |\cdot|)$,
we denote by $\cP_2(E)$  the metric space of  probability measures 
$\mu$
on $E$ satisfying $\|\mu\|_2=(\int_E |x|^2\,\d \mu(x))^{1/2}<\infty$,
endowed with the  $2$-Wasserstein metric defined as follows: 
$$
\cW_2(\mu_1,\mu_2)
\coloneqq \inf_{\kappa\in \Pi(\mu_1,\mu_2)} \left(\int_{E\t E}|x-y|^2\d \kappa( x, y)\right)^{1/2},
\q \mu_1,\mu_2\in \cP_2(E),
$$
where $\Pi(\mu_1,\mu_2)$ is the set of all couplings of $\mu_1$ and $\mu_2$, i.e.,
$\kappa\in \Pi(\mu_1,\mu_2)$ is a probability measure on $E\t E$ such that $\kappa(\cdot\t E)=\mu_1$ 
and $\kappa(E\t \cdot)=\mu_2$.
%
For a given function $h:\cP_2(\sR^n\t \sR^k)\to \sR$
and a measure  $\eta\in \cP_2(\sR^n\t \sR^k)$
with marginals $\mu\in \cP_2(\sR^n)$, $\nu\in \cP_2(\sR^k)$,
 we  denote by $\p_\eta h(\eta)(\cdot):\sR^n\t \sR^k\to \sR^n\t \sR^k$
the  L-derivative of $h$ at $\eta$ 
and by 
$(\p_\mu h(\eta),\p_\nu h(\eta))(\cdot):\sR^n\t \sR^k\to \sR^n\t \sR^k$
the partial L-derivatives of 
$h$ with respect to  the marginals;
see e.g. \cite[Section 2.1]{acciaio2019} or \cite[Chapter 5]{carmona2018a}
for detailed definitions.

Finally,
we shall  denote by $C\in [0,\infty)$ a generic constant
throughout this paper, which is independent of the initial condition $\xi_0$,
though it may depend on the constants appearing in the assumptions  
  and may take a different value at each occurrence. 
\section{MV-FBSDEs for mean field control problems}\l{sec:mfcE_mvfbsde}

In this section, we state the main assumptions on the coefficients
of  the   MFC problems \eqref{eq:mfcE},
and then derive a coupled MV-FBSDE based on  the  stochastic  maximum principle,
which plays an essential role for our subsequent 
convergence analysis of 
piecewise constant  policy approximations.

\begin{Assumption}\l{assum:mfcE}
Let $\bA\subset \sR^k$ be a nonempty closed convex set
and
 let $b:[0,T]\t \sR^n\t \sR^k\t  \cP_2(\sR^n\t \sR^k)\to \sR^n$,
$\sigma:[0,T]\t \sR^n\t  \cP_2(\sR^n)\to\sR^{n\t d}$,
$f:[0,T]\t \sR^n\t \sR^k\t  \cP_2(\sR^n\t \sR^k)\to\sR$ 
and $g:\sR^n\t  \cP_2(\sR^n)\to\sR$ 
be measurable functions satisfying the following properties:
\begin{enumerate}[(1)]
\item \l{item:mfcE_lin}
$b$ and  $\sigma$ are affine in $(x,a,\eta)$,
i.e.,
there exist
functions
$b_0\in L^2(0,T;\sR^n)$ 
and
 $( b_1, b_2, b_3,\sigma_0,\sigma_1,\sigma_2)
\in L^\infty(
0,T;  \sR^{n\t n}\t  \sR^{n\t k} \t\sR^{n\t (n+k)} \t \sR^{n\t d}\t \sR^{(n\t d)\t n}\t \sR^{(n\t d)\t n})$ 
such that for all $(t,x,a,\mu,\eta)\in [0,T]\t \sR^n\t \sR^k\t  \cP_2(\sR^n) \t \cP_2(\sR^n \times \sR^k)$,
\begin{align*}
b(t,x,a,\eta)&=b_0(t)+b_1(t)x + b_2(t)a+b_3(t)\bar{\eta},\\
\sigma(t,x,\mu)&=\sigma_0(t)+\sigma_1(t)x+\sigma_2(t)\bar{\mu},
\end{align*}
where  $\bar{\eta}=\int (x,a)\,\d \eta (x,a)$
and $\bar{\mu}=\int x\,\d \mu(x)$ 
denote
the first moments of the measures $\eta$ and $\mu$, respectively. 

\item \l{item:mfcE_growth}
$ f(\cdot,0,0,\bm{\delta}_{0_{n+k}})\in L^\infty(0,T)$,
$f$ and  $g$ are  differentiable with respect to $(x,a,\eta)$ and $(x,\mu)$, respectively,
and all derivatives are of linear growth, i.e., 
there exists a constant $\hat{L}\in [0,\infty)$ such that 
for all $R\ge 0$ and all $(t,x,a,\mu,\eta)$ with 
$|x|, |a|,\|\mu\|_2, \| \eta \|_2\le R$,
we have that
$|\p_xf(t,x,a,\eta)|+|\p_a f(t,x,a,\eta)|+|\p_x g(x,\mu)|\le \hat{L}(1+R)$,
the $L^2(\sR^{n}\t\sR^k,\eta)$-norms
of the maps $(x',a') \mapsto \p_\mu f(t,x,a,\eta)(x',a')$, $(x',a') \mapsto \p_\nu f(t,x,a,\eta)(x',a')$  
are bounded by $\hat{L}(1+R)$,
and the $L^2(\sR^{n},\mu)$-norm of the map
  $x'\mapsto \p_\mu g(x,\mu)(x')$
is bounded by $\hat{L}(1+R)$.
\color{black}

\item\l{item:mfcE_lipschitz} 
There exists a constant $\tilde{L}\in [0,\infty)$ such that 
for all $t\in [0,T]$,
the functions
$\p_{x}f(t,\cdot)
:\sR^n\t \bA\t \cP_2(\sR^n \times \sR^{k}) \to \sR^n$,
$\p_{a}f(t,\cdot)
:\sR^n\t \bA\t \cP_2(\sR^n \times \sR^{k}) \to \sR^k$
and 
$\p_xg(\cdot):\sR^n\t \cP_2(\sR^n)\to \sR^n$
are 
$\tilde{L}$-Lipschitz continuous.
Moreover, for any $(t,x,a,\eta,\mu)\in [0,T]\t \sR^n\t \sR^k\t \cP_2(\sR^n \times \sR^{k})\t \cP_2(\sR^n)$,
there exist versions of 
$\p_\mu f(t,x,a,\eta)(\cdot)$,  $\p_\nu f(t,x,a,\eta)(\cdot)$ and  $\p_\mu g(x,\mu)(\cdot)$
such that 
\begin{align*}
& (x,a,\eta,\mu,x',a')\in \sR^n\t \bA \t \cP_2(\sR^n \times \sR^{k})\t \cP_2(\sR^n)\t \sR^n \t \bA   \\
& \quad \mapsto (\p_\mu f(t,x,a,\eta)(x',a'), \p_\nu f(t,x,a,\eta)(x',a'),\p_\mu g(x,\mu)(x'))\in \sR^n \t \sR^k \t \sR^n,
\end{align*}
is $\tilde{L}$-Lipschitz continuous.

\item\l{item:mfcE_convex}
 $f$ is  convex 
with respect to $(x,a,\eta)$, i.e.,
 there exist constants $\lambda_1,\lambda_2\ge 0$
satisfying  $\lambda_1+\lambda_2>0$ and 
for all $t\in [0,T]$, $(x,a,\eta),(x',a',\eta')\in \sR^n\t \bA \t \cP_2(\sR^n \t \sR^{k})$,
\begin{align*}
&f(t,x',a',\eta')-f(t,x,a,\eta)-
\la \p_{x}f(t,x,a,\eta), x'-x\ra 
-\la \p_{a}f(t,x,a,\eta), a'-a\ra 
\\
&\q 
-\tilde{\sE}[\la\p_\mu f(t,x,a,\eta)(\tilde{X},\tilde{\alpha}),\tilde{X}'-\tilde{X}\ra 
+ \la\p_\nu f(t,x,a,\eta)(\tilde{X},\tilde{\alpha}),\tilde{\alpha}'-\tilde{\alpha}\ra ]
\\
&\q
\ge \lambda_1|a'-a|^2+\lambda_2\tilde{\sE}[|\tilde{\alpha}'-\tilde{\alpha}|^2],
\end{align*}
whenever  $(\tilde{X},\tilde{\a}),(\tilde{X}',\tilde{\a}')\in L^2(\tilde{\Om},\tilde{\cF},\tilde{\sP};\sR^n\t\sR^k)$
with distributions $\eta$ and $\eta'$, respectively.
The function 
$g$ is convex in $(x,\mu)$, i.e.,
we have for all $(x,\mu),(x',\mu')\in \sR^n \t \cP_2(\sR^n)$ that
\begin{align*}
g(x',\mu')-g(x,\mu)-\la \p_{x}g(x,\mu), x'-x \ra 
-\tilde{\sE}[\la\p_\mu g(x,\mu)(\tilde{X}),\tilde{X}'-\tilde{X}\ra ]\ge 0,
\end{align*}
whenever  $\tilde{X},\tilde{X}'\in L^2(\tilde{\Om},\tilde{\cF},\tilde{\sP};\sR^n)$
with distributions $\mu$ and $\mu'$, respectively.
Above and hereafter, we denote by $\tilde{\sE}$  the expectation on  $(\tilde{\Om},\tilde{\cF},\tilde{\sP})$.
\end{enumerate}
\end{Assumption}

\begin{Remark}\l{rmk:mfcE_regularity}
(H.\ref{assum:mfcE})  naturally extends
  Assumption ``Control of MKV Dynamics" in 
\cite{carmona2018a} 
to the present setting with mean field interactions through controls.
In particular, (H.\ref{assum:mfcE})
allows the coefficients $(b,\sigma,f)$ to be merely measurable in time,
and the cost function $f$ to be strongly convex
either in the state or in the law of the controls,
which is important for the applications to control problems whose 
cost function does not explicitly depend on the state of the controls 
(see e.g.~Proposition \ref{prop:control_measure}).
The  assumption that the volatility coefficient is  uncontrolled
enables us to study the regularity of optimal controls 
and subsequently to quantify  
the time discretization error of \eqref{eq:mfcE}
via a probabilistic approach
(see Theorems \ref{TH:ControlRegularity} and \ref{thm:PCPT_discrete_value}).

Note that
the continuous differentiability of $f$ and the linear growth of its derivatives
(see (H.\ref{assum:mfcE}(\ref{item:mfcE_growth})(\ref{item:mfcE_lipschitz})))
show that 
there exists a constant $C$ satisfying 
for all $(x,a,\eta),(x',a',\eta')\in \sR^n\t \bA \t \cP_2(\sR^n \t \sR^{k})$ that 
\begin{align*}
&|f(t,x,a,\eta)-f(t,x',a',\eta')|
\\
&\le C(1+|x|+|x'|+|a|+|a'|+\|\eta\|_2+\|\eta'\|_2)(|(x,a)-(x',a')|+\cW_2(\eta,\eta')),
\end{align*}
which together with the uniform boundedness of $|f(t,0,0,\bm{\delta}_{0_{n+k}})|$ implies that 
the function $f$ is at most of quadratic growth with respect to $(x,a,\eta)$.
Similar arguments show that the function $g$ is 
  locally Lipschitz continuous
and  at most of quadratic growth 
with respect to $(x,\mu)$.

\end{Remark}

It is clear that 
under (H.\ref{assum:mfcE}),
for any given  
initial state $\xi_0\in L^2(\cF_0;\sR^n)$
and
 admissible control $\a\in \cA$,
the controlled state process $X^\a\in \cS^2(\sR^n)$ is well-defined by \eqref{eq:mfcE_fwd}
and the cost functional $J(\a;\xi_0)$ is finite since the functions $f$ and $g$ are
at most
 of quadratic growth
(see Remark \ref{rmk:mfcE_regularity}).
We now apply the stochastic maximum principle
to  \eqref{eq:mfcE}
and 
 characterize the optimal control
by a MV-FBSDE.

Let 
$H:[0,T]\t \sR^n\t \sR^k\t  \cP_2(\sR^n\t \sR^k)\t \sR^n\t \sR^{n\t d}\to \sR$
be the Hamiltonian of \eqref{eq:mfcE} defined as follows:
\bb\l{eq:mfcE_hamiltonian}
H(t,x,a,\eta,y,z)\coloneqq\la b(t,x,a,\eta), y\ra +\la \sigma(t,x,\pi_1\sharp\eta),z\ra+f(t,x,a,\eta),
\ee
where $\pi_1\sharp \eta$ denotes the first marginal of the measure $\eta$.
The linearity of $b,\sigma$ in 
(H.\ref{assum:mfcE}(\ref{item:mfcE_lin}))
and the convexity of  $f,g$ 
in  (H.\ref{assum:mfcE}(\ref{item:mfcE_convex}))
ensure that
the stochastic  maximum principle gives a necessary and sufficient optimality condition
of an optimal control of \eqref{eq:mfcE};
see e.g.~\cite[Theorem 3.5]{acciaio2019}
for the optimality condition with  a bounded function $b_0$ in 
(H.\ref{assum:mfcE}(\ref{item:mfcE_lin})),
which can be easily extended to the present setting.
More precisely, 
 suppose that (H.\ref{assum:mfcE}) holds and let $\xi_0\in L^2(\cF_0;\sR^n)$ be a given initial state. For any given admissible control $\a\in \cA$, 
 let $X^\a$ be the  corresponding controlled state process satisfying \eqref{eq:mfcE_fwd}, and 
 let $(Y^\a,Z^\a)\in \cS^2(\sR^n)\t \cH^2(\sR^{n\t d})$ be 
 an
  adjoint process of $X^\a$ satisfying the following 
 MV-FBSDE:
 for all $t\in [0,T]$,  
 \begin{align}\l{eq:mfcE_bsde_nonMarkov}
\begin{split}
\d Y^{{\a}}_t&=-\big(\p_x H(\theta^\a_t,Y^{{\a}}_t,Z^{{\a}}_t)
+\tilde{\sE}[\p_\mu H(\tilde{\theta}^\a_t,\tilde{Y}^{{\a}}_t,\tilde{Z}^{{\a}}_t)(X^{{\a}}_t,{\a}_t)]\big)\,\d t
+Z^{{\a}}_t\,\d W_t,
\\
Y^{{\a}}_T&=\p_x g(X^{{\a}}_T,\sP_{X^{{\a}}_T})+\tilde{\sE}[\p_\mu g(\tilde{X}^{{\a}}_T,\sP_{X^{{\a}}_t})(X^{{\a}}_T)],
\end{split}
\end{align}
where $\theta^\a_t=(t,X^{{\a}}_t,{\a}_t,\sP_{(X^{{\a}}_t,{\a}_t)})$
and  the tilde notation refers to an independent copy. 
Then the  stochastic  maximum principle asserts that
if the following optimality condition holds:
\begin{align}\label{eq:opti}
\la \p_a H(\theta^\a_t,Y^\a_t,Z^\a_t) + \tilde{\mathbb{E}} [\p_\nu  H(\tilde{\theta}^\a_t,\tilde{Y}^\a_t,\tilde{Z}^\a_t)(X^\a_t,\alpha_t)], \alpha_t -a \ra\leq 0, \quad 
\textnormal{$\forall a \in \bA, \ \mathrm{d}\mathbb{P} \otimes \mathrm{d}t$\ -a.e.,}
\end{align}
then $\a\in \cA$ is an optimal control of \eqref{eq:mfcE}.
Note that under (H.\ref{assum:mfcE}), 
for any given  control $\a\in \cA$, the adjoint process
$(Y^\a,Z^\a)\in \cS^2(\sR^n)\t \cH^2(\sR^{n\t d})$ is uniquely defined
(see \cite{acciaio2019}).

One can clearly observe  that 
 the optimality condition \eqref{eq:opti}
and  the progressively measurable control process $\a$ 
 lead to a non-Markovian coupled forward-backward system
\eqref{eq:mfcE_fwd}, \eqref{eq:mfcE_bsde_nonMarkov}
and \eqref{eq:opti}
with random coefficients.
In the following, 
we shall reduce the problem into a  forward-backward system
 with deterministic coefficients
by assuming the solvability of the optimality condition 
\eqref{eq:opti}.
%
%
%
%
%

We first observe that,
by virtue of  the fact that  
the coefficient $\sigma$ is uncontrolled,
the optimality condition \eqref{eq:opti} can be equivalently 
written as: 
it holds for all $ a \in \bA$ and for $\mathrm{d}\mathbb{P} \otimes \mathrm{d}t$\ -a.e.~that 
\begin{align}\l{eq:opti_re}
\la \p_a H^{\textrm{re}}(t, X^\a_t,\a_t, \sP_{(X^\a_t,\a_t)},Y^\a_t) 
+
\tilde{\mathbb{E}} [\p_\nu  H^{\textrm{re}}(t, \tilde{X}^\a_t,\tilde{\a}_t, {\sP}_{({X}^\a_t,{\a}_t)},\tilde{Y}^\a_t)(X^\a_t,\alpha_t)], \alpha_t -a \ra\leq 0, \quad 
\end{align}
where 
$H^{\textrm{re}}:[0,T]\t \sR^n\t \sR^k\t  \cP_2(\sR^n\t \sR^k)\t \sR^n\to \sR$
is  the reduced Hamiltonian defined by:
\bb\l{eq:hamiltonian_re}
H^{\textrm{re}}(t,x, a,\eta,y)\coloneqq \la b(t,x,a,\eta), y\ra +f(t,x,a,\eta).
\ee
The following assumption then asserts that 
the
 optimality condition \eqref{eq:opti_re} can be achieved by a 
 sufficiently regular
 feedback  map
 from the state and adjoint processes
 to the action set,
 which will be verified 
for several   MFC problems
appearing in practice.

\begin{Assumption}\phantomsection\l{assum:mfcE_hat}
\begin{enumerate}[(1)]
\item \l{item:ex} 
Assume the notation of (H.\ref{assum:mfcE}). 
There exists a  measurable function $\hat{\a}:[0,T]\t \sR^n \t \sR^n \t \cP_2(\sR^{n} \t \sR^{n}) \to \bA$ 
and a constant $L_\a\in [0,\infty)$ such that 
for all $t\in [0,T]$,
$|\hat{\a}(\cdot,0,0,\bm{\delta}_{0_{n + n}})| \leq L_\a$,
the function $\hat{\a}(t,\cdot):\sR^n \t \sR^n \t \cP_2(\sR^{n} \t \sR^{n}) \to \bA$ is 
$L_\a$-Lipschitz continuous,
and the optimality condition \eqref{eq:opti_re} holds,
 i.e., for all $(x,y,\chi,a)\in \sR^n \t \sR^n \t \cP_2(\sR^{n} \t \sR^{n})\t \bA$, 
\begin{align}\l{eq:opti_markov}
\begin{split}
&\la \p_a H^{\textrm{re}}(t,x,
\hat{\a}(t,x,y,\chi),
\phi(t,\chi),y)
\\
&\q
 +\int_{\sR^n\t \sR^n} \p_\nu  H^{\textrm{re}}(t,\tilde{x},\hat{\a}(t,\tilde{x},\tilde{y},\chi) ,\phi(t,\chi), \tilde{y})\big(x,\hat{\a}(t,x,y,\chi) \big)\,\d \chi(\tilde{x},\tilde{y}),
\\
&\q \hat{\a}(t,x,y,\chi) -a \ra\leq 0, 
\end{split}
\end{align}
where
$\phi(t,\chi) \coloneqq \chi \circ \big(\sR^n\t \sR^n\ni (x,y)\mapsto(x,\hat{\a}(t,x,y,\chi))\in \sR^n\t \bA\big)^{-1}$.

\item \l{item:HC} 
The function $\hat{\a}$ is locally H\"{o}lder continuous in time, i.e., it holds for all $t,t'\in [0,T], (x,y,\chi)\in \sR^n \t \sR^n \t \cP_2(\sR^n \times \sR^{n})$ that
$|\hat{\a}(t,x,y,\chi)-\hat{\a}(t',x,y,\chi)|\le L_{\a}(1+|x|+|y|+\|\chi\|_2)|t-t'|^{1/2}$.
\end{enumerate}
\end{Assumption}

Roughly speaking, (H.\ref{assum:mfcE_hat}) ensures that there exists a deterministic function $\hat{\a}$ satisfying 
the optimality condition  \eqref{eq:opti}  pointwise,
which enables us to study controls $\hat{\a}\in \cA$ of the form $\hat{\a}_t=\hat{\a}(t,X^{\hat{\a}}_t,Y^{\hat{\a}}_t,\sP_{(X^{\hat{\a}}_t,Y^{\hat{\a}}_t)})$, $t\in [0,T]$.
In this case, it is not difficult to see that 
 $\phi(t,\sP_{(X^{\hat{\a}}_t,Y^{\hat{\a}}_t)})$ is the joint law of $(X^{\hat{\a}}_t,\hat{\a}_t)$, 
 since  it holds
 for any $t\in [0,T]$, $X,Y \in L^2(\Om;\sR^n)$
and any Borel measurable set $A \subset \sR^n \t \sR^k$ that
\begin{align}\l{eq:PushF}
\begin{split}
\phi(t,\sP_{(X,Y)})(A) &= \sP_{(X,Y)} \left( (\text{id}_{\sR^{n}},\hat{\a}(t,\cdot,\cdot,\sP_{(X,Y)}))^{-1}(A) \right) \\
& = \sP \left( (X,Y) \in (\text{id}_{\sR^{n}},\hat{\a}(t,\cdot,\cdot,\sP_{(X,Y)}))^{-1}(A) \right) \\
& = \sP \left( (X,\hat{\a}(t,X,Y,\sP_{(X,Y)})) \in  A \right) = \mathbb{P}_{(X,\hat{\a}(t,X,Y,\sP_{(X,Y)}))}(A).
\end{split}
\end{align}

Note that similar assumptions have been made in 
\cite[Theorem 3]{gobet2019} and \cite[Assumption (A6)]{lauriere2020}
to study  MFC problems.
Under (H.\ref{assum:mfcE}) and  (H.\ref{assum:mfcE_hat}),
we shall
establish the existence of a H\"{o}lder continuous optimal control for \eqref{eq:mfcE}
in Section \ref{sec:regularity_mfcE},
  and then analyze the  convergence rate of piecewise constant policy approximation
  for \eqref{eq:mfcE}
  in Section \ref{sec:conv_PCPT}.

In the following, we  verify 
 (H.\ref{assum:mfcE_hat})  for  different classes of
   MFC problems
appearing in practice,
which are not covered by results in the existing literature.
In particular, we shall give precise  conditions on the functions $(b,f)$ in \eqref{eq:mfcE}
to ensure the existence and  regularity of the function $\hat{\a}$.
Note that these conditions do not involve high-order derivatives of the cost functions,
which 
enables us to quantify the time discretization error of \eqref{eq:mfcE}
under much
  weaker assumptions than conditions (B1) and (C1)-(C3) in \cite{carmona2019}
  (see the discussions above Theorem \ref{thm:PCPT_discrete_value} for details).
In particular,
we allow
merely measurable  functions $(b_i,\sigma_i)$,
a possibly degenerate state-dependent diffusion coefficient,
and  cost functions $(f,g)$ that are not necessarily twice differentiable.

\begin{Example}\l{example:mfc}
In this example, we show (H.\ref{assum:mfcE_hat}) is satisfied by
 a class of   MFC problems 
with cost function $f$ which does not involve the law of the controls.
This includes the classical MFC problem
as a special case,
for which the  controlled dynamics is also independent of
 the law of the controls
(see \cite{chassagneux2014,carmona2015,carmona2018a,carmona2019}).

The proof of the following proposition is based on 
defining the function $\hat{\a}$ as the minimizer of a modified version 
of the reduced Hamiltonian ${H}^{\textrm{re}}$.
Note that one can adapt 
the arguments  to verify (H.\ref{assum:mfcE_hat})
for more general cost functions which are affine in the law of the controls,
i.e., $f(t,x,a,\eta)=f_1(t,x,a,\pi_1\sharp\eta)+\la f_2(t,x,\pi_1\sharp\eta), \ol{\pi_2\sharp \eta}\ra$,
but for notational simplicity,  we choose to refrain from
providing this level of generality
without the motivation from specific applications.

\begin{Proposition}\l{prop:mfc_hat_a}
Suppose (H.\ref{assum:mfcE}) holds,
and 
for each $(t,x,a)\in [0,T]\t\sR^n\t \sR^k$,
the function
 $\cP_2(\sR^n\t \sR^k)\ni \eta\mapsto 
f(t,x,a,\eta)\in \sR$
 depends only on  
 the first  marginal $\pi_1\sharp \eta$ of the measure $\eta$.
Then there exists a function
$\hat{\a}:[0,T]\t \sR^n\t \sR^n\t \cP_2(\sR^n\t \sR^n)\to \bA$ satisfying 
 (H.\ref{assum:mfcE_hat}(\ref{item:ex})).

Assume further that
there exists a constant $\tilde{K}\in [0,\infty)$ such that 
it holds for all $t,t'\in [0,T]$, $(x,a,\eta)\in \sR^n\t  \bA\t \cP_2(\sR^n\t \sR^n)$
that 
$|b_2(t)-b_2(t')|+|b_3(t)-b_3(t')|\le \tilde{K}|t-t'|^{1/2}$ and $|\p_a f(t,x,a,\eta)-\p_a f(t',x,a,\eta)|\le \tilde{K}(1+|x|+|a|+\|\eta\|_2)|t-t'|^{1/2}$.
Then there exists a function
$\hat{\a}:[0,T]\t \sR^n\t \sR^n\t \cP_2(\sR^n\t \sR^n)\to \bA$ satisfying 
 (H.\ref{assum:mfcE_hat}).

\end{Proposition}

\begin{proof}
Observe that under the assumptions of Proposition \ref{prop:mfc_hat_a},
the reduced Hamiltonian \eqref{eq:hamiltonian_re} can be written as follows:
for all 
$ (t,x, a,\eta,y)\in [0,T]\t\sR^n\t \sR^k\t \cP_2(\sR^n\t \sR^k)\t \sR^n$,
\begin{align*}
{H}^{\textrm{re}}(t,x, a,\eta,y)
&= \la {b}(t,x,a,\eta), y\ra +\tilde{f}(t,x,a,\pi_1\sharp\eta)
\\
&= \psi_1(t,x,a,y)+\psi_2(t,\eta,y) +\tilde{f}(t,x,a,\pi_1\sharp\eta),
\end{align*}
where we have
$\psi_1(t,a,x,y)\coloneqq \la b_0(t)+b_1(t)x + b_2(t)a, y\ra$,
$ \psi_2(t,\eta,y)\coloneqq \la b_3(t)\bar{\eta},y\ra $
and 
 $\tilde{f}(t,x,a,\mu)\coloneqq {f}(t,x,a,\mu\t \bm{\delta}_{0_k})$.
Moreover, we have that
$\p_a {H}^{\textrm{re}}(t,x, a,\eta,y)=b_2^*(t)y+\p_a\tilde{f}(t,x,a,\pi_1\sharp\eta)$
and $\p_\nu {H}^{\textrm{re}}(t,x, a,\eta,y)(\cdot)=\p_\nu\psi_2(t,\eta,y) (\cdot)=\b(t)y$, where 
 $\b(t)\in \sR^{k\t n}$ is the  submatrix  formed by deleting the first $n$ rows of $b^*_3(t)\in \sR^{(n+k)\t n}$.
 
 Let us define the function $G:[0,T]\t \sR^n\t \sR^k\t  \cP_2(\sR^n)\t \cP_2(\sR^n)\t \sR^n\to \sR $
satisfying for all 
$(t,x,a,\mu,\rho,y)\in [0,T]\t \sR^n\t \sR^k\t  \cP_2(\sR^n)\t\cP_2(\sR^k)\t \sR^n $ that 
\begin{align*}
G(t,x,a,\mu,\rho,y)\coloneqq 
 \psi_1(t,x,a,y)+\psi_2(t,\mu\times \bm{\delta}_a,\bar{\rho}) +\tilde{f}(t,x,a,\mu)
\end{align*}
with $\bar{\rho}=\int_{\sR^n}y\,\d \rho(y)$.
 We further  define the map 
$ \hat{\a}:[0,T]\t \sR^n\t  \sR^n\t \cP_2(\sR^n\t \sR^n)\to \bA$
satisfying for all $(t,x,y,\chi)\in [0,T]\t \sR^n\t  \sR^n\t \cP_2(\sR^n\t \sR^n)$ that
\bb\l{eq:mfc_a_hat}
\hat{\a}(t,x,y,\chi)= \argmin_{\a\in \bA}G(t,x,a,\pi_1\sharp \chi,\pi_2\sharp \chi,y).
\ee
Since the function $f$   depends only on  
 the first  marginal $\pi_1\sharp \eta$,
 we see from (H.\ref{assum:mfcE}(\ref{item:mfcE_convex}))
 that
 $\lambda_1>0$ and 
  the map
  $\bA\ni a\mapsto  \tilde{f}(t,x,a,\mu)\in \sR$ is $\lambda_1$-strongly convex,
  which along with the linearity of the map
$\bA\ni a\mapsto  \psi_1(t,x,a,y)+\psi_2(t,\mu\times \bm{\delta}_a,\bar{\rho})\in \sR$
shows $\bA\ni a\mapsto G(t,x,a,\pi_1\sharp \chi,\pi_2\sharp \chi,y)\in \sR$ is $\lambda_1$-strongly convex.
Then by following the same argument as in  \cite[Lemma 3.3]{carmona2018a},
we can show
the above function $\hat{\a}$ is 
well-defined, measurable,
 locally bounded and Lipschitz continuous with respect to $(x,y,\chi)$ uniformly in $t$.

Then it remains to verify  \eqref{eq:opti_markov}
in order to show that  $\hat{\a}$ satisfies  (H.\ref{assum:mfcE_hat}(\ref{item:ex})).
The fact that  $\hat{\a}$ is a minimizer of $G$ over $\bA$
and the definition of $G$ 
imply
for all $(t,x,y,\chi)\in [0,T]\t \sR^n\t  \sR^n\t \cP_2(\sR^n\t \sR^n)$,
$a\in \bA$ that 
\begin{align}\l{eq:hat_a_optimal_G}
\begin{split}
0&
\ge \la \p_a G(t,x,\hat{\a}(t,x,y,\chi),\pi_1\sharp \chi,\pi_2\sharp \chi,y),
\hat{\a}(t,x,y,\chi)-a\ra 
\\
&
= \la b_2^*(t)y +\b(t)\ol{\pi_2\sharp \chi}+\p_a\tilde{f}(t,x,\hat{\a}(t,x,y,\chi),\pi_1\sharp\chi),
\hat{\a}(t,x,y,\chi)-a\ra,
\end{split}
\end{align}
where $\ol{\pi_2\sharp \chi}=\int_{\sR^n} y\,\d {\pi_2\sharp \chi}(y)$.
The fact that  $f(t,x,a,\eta)$ depends only on  the first marginal of the measure $\eta$
gives us that
$\p_a\tilde{f}(t,x,a,\pi_1\sharp\chi)=\p_a\tilde{f}(t,x,a,\pi_1\sharp\phi(t,\chi))$, with the function $\phi$ defined as in \eqref{eq:opti_markov}.
Hence, we can obtain from the expression of $\p_a{H}^{\textrm{re}}$ that 
\begin{align*}
0&
\ge \la b_2^*(t)y +\b(t)\ol{\pi_2\sharp \chi}+\p_a\tilde{f}(t,x,\hat{\a}(t,x,y,\chi),\pi_1\sharp\phi(t,\chi)),
\hat{\a}(t,x,y,\chi)-a\ra
\\
&= \la \p_a {H}^{\textrm{re}}(t,x, \hat{\a}(t,x,y,\chi),\phi(t,\chi),y)
+\b(t)\ol{\pi_2\sharp \chi},
\hat{\a}(t,x,y,\chi)-a\ra,
\end{align*}
which is the optimality condition \eqref{eq:opti_markov} 
since for all $ (t,a,\eta)$,
$$
\int_{\sR^n\t \sR^n}\p_\nu {H}^{\textrm{re}}(t,\tilde{x}, a,\eta,\tilde{y})(\cdot)\,\d \chi(\tilde{x},\tilde{y})=\int_{\sR^n\t \sR^n} \b(t)\tilde{y}\,\d \chi(\tilde{x},\tilde{y})=\b(t)\ol{\pi_2\sharp \chi}.
$$

 We now prove the time regularity of $\hat{\a}$ under the additional assumption 
 on the H\"{o}lder regularity of the functions $b_2, b_3$ and $\p_a f$.
Let $t,t'\in [0,T], (x,y,\chi)\in \sR^n\t \sR^n\t \cP_2(\sR^n\t \sR^n)$,
 $\hat{a}=\hat{\a}(t,x,y,\chi)$ and $\hat{a}'=\hat{\a}(t',x,y,\chi)$.
The optimal condition \eqref{eq:hat_a_optimal_G} gives us that  
$\la  \p_a G(t,x,\hat{a},\pi_1\sharp \chi,\pi_2\sharp \chi,y)
,\hat{a}'-\hat{a}\ra\ge 0
 \ge 
\la  \p_a G(t',x,\hat{a}',\pi_1\sharp \chi,\pi_2\sharp \chi,y),\hat{a}'-\hat{a}\ra
$.
Moreover, the
$\lambda_1$-strong convexity of 
 $\bA\ni a\mapsto G(t,x,a,\pi_1\sharp \chi,\pi_2\sharp \chi,y)\in \sR$ 
 shows that
\begin{align*}
&G(t,x,\hat{a}',\pi_1\sharp \chi,\pi_2\sharp \chi,y)-G(t,x,\hat{a},\pi_1\sharp \chi,\pi_2\sharp \chi,y)
\\
&\q
-\la \p_{a}G(t,x,\hat{a},\pi_1\sharp \chi,\pi_2\sharp \chi,y), \hat{a}'-\hat{a}\ra \ge \lambda_1|\hat{a}'-\hat{a}|^2,
\end{align*}
from which, by exchanging the role of $ \hat{a}'$ and $ \hat{a}$ 
in the above inequality and summing the resulting estimates, we can deduce that  
\begin{align*}
2\lambda_1|\hat{a}'-\hat{a}|^2
&\le
\la\hat{a}'-\hat{a}, \p_{a}G(t,x,\hat{a}',\pi_1\sharp \chi,\pi_2\sharp \chi,y)-\p_{a}G(t,x,\hat{a},\pi_1\sharp \chi,\pi_2\sharp \chi,y)\ra
\\
&\le
\la\hat{a}'-\hat{a}, \p_{a}G(t,x,\hat{a}',\pi_1\sharp \chi,\pi_2\sharp \chi,y)-\p_{a}G(t',x,\hat{a}',\pi_1\sharp \chi,\pi_2\sharp \chi,y)\ra.
\end{align*}
Hence, we can obtain from the expression of $\p_a G$ that
\begin{align*}
|\hat{a}'-\hat{a}|
&\le
C |\p_{a}G(t,x,\hat{a}',\pi_1\sharp \chi,\pi_2\sharp \chi,y)-\p_{a}G(t,x,\hat{a}',\pi_1\sharp \chi,\pi_2\sharp \chi,y)|
\\
&\le C
\Big(|
\p_a\tilde{f}(t,x,\hat{\a}',\pi_1\sharp\chi)-\p_a\tilde{f}(t',x,\hat{\a}',\pi_1\sharp\chi)|
+|b_2^*(t)-b_2^*(t')||y|
\\
&\q +|\b(t)-\b(t')|\|\chi\|_2\Big),
\end{align*}
for a constant $C$ independent of $(t,t',x,y,\chi)$.
Then, by applying the H\"{o}lder regularity assumption of the coefficients, we can obtain that
\begin{align*}
|\hat{a}'-\hat{a}|
&\le
C(1+|x|+\|\chi\|_2+|\hat{a}'|+|y|)|t-t'|^{1/2},
\end{align*}
which, together with the fact that the function $\hat{a}$ is locally bounded and  of linear growth in $(x,y,\chi)$, 
leads to the desired H\"{o}lder continuity of $\hat{\a}$.
\end{proof}

\end{Example}

\begin{Example}\l{example:law_c}
In this example, we verify 
 (H.\ref{assum:mfcE_hat})  for 
  MFC problems where the dependence of the cost function $f$ 
on $(X,\a)$ takes a separable form, and 
the forward dynamics \eqref{eq:mfcE_fwd}
depends on the  control process 
only through its expectation.

Note that  \cite{burzoni2020} studies a class of MFC problems
where the coefficients of the controlled dynamics 
depend on the  state process 
only through its expectation,
and  admissible controls are chosen to be  deterministic functions. 
The following proposition can be viewed as a generalization of such problems 
since it shows that for certain  MFC problems,
the unique optimal control in $\cA$ is in fact deterministic,
even though the coefficients of the forward dynamics can
 depend on the  state of the controlled process explicitly.

The proof of the following result is based on 
defining the function $\hat{\a}$ as the minimizer of the expectation of the reduced Hamiltonian ${H}^{\textrm{re}}$.

\begin{Proposition}\l{prop:control_measure}
Suppose  (H.\ref{assum:mfcE}) holds,
the function $b_2$ 
in (H.\ref{assum:mfcE}(\ref{item:mfcE_lin}))
satisfies $b_2(t)=0$ for all $t \in [0,T]$,
and
the function $f$ 
 is of the form 
$f(t,x,a,\eta)=f_1(t,x,\pi_1\sharp \eta,\pi_2\sharp \eta)
+f_2(t,a,\pi_1\sharp \eta,\pi_2\sharp \eta)$,
where 
 $f_1:[0,T]\t \sR^n\t \cP_2(\sR^n)\t \cP_2(\sR^k)\to \sR$
and $f_2:[0,T]\t \sR^k\t \cP_2(\sR^n)\t \cP_2(\sR^k)\to \sR$
are functions
satisfying (H.\ref{assum:mfcE}(\ref{item:mfcE_growth})(\ref{item:mfcE_lipschitz})),
and $\pi_1\sharp\eta$ (resp.~$\pi_2\sharp \eta$) is 
 the first (resp.~second) marginal of the measure $\eta$.
Then there exists a function $\hat{\a}:[0,T]\t \cP_2(\sR^n\t \sR^n)\to\bA$ satisfying 
 (H.\ref{assum:mfcE_hat}(\ref{item:ex})). 

Assume further that
there exists a constant $\tilde{K}\in [0,\infty)$ such that 
it holds for all $t,t'\in [0,T]$, $(x,a,\mu)\in \sR^n\t  \bA\t \cP_2(\sR^n)$
that 
$|b_3(t)-b_3(t')|\le \tilde{K}|t-t'|^{1/2}$ 
and 
\begin{align*}
&
|\p_\nu f_1(t,x,\mu,\bm{\delta}_{{a}})({a})
-\p_\nu f_1(t',x,\mu,\bm{\delta}_{{a}})({a})|
+
|\p_a f_2(t,a,\mu,\bm{\delta}_{{a}})
-\p_a f_2(t',a,\mu,\bm{\delta}_{{a}})|
\\
&\q+|\p_\nu f_2(t,a,\mu,\bm{\delta}_{{a}})({a})
-
\p_\nu f_2(t',a,\mu,\bm{\delta}_{{a}})({a})|
\le \tilde{K}(1+ |x| + |a| + \|\mu\|_2)|t-t'|^{1/2}.
\end{align*}
Then there exists a function $\hat{\a}:[0,T]\t \cP_2(\sR^n\t \sR^n)\to\bA$ satisfying 
 (H.\ref{assum:mfcE_hat}).
\end{Proposition}

\begin{proof}
We shall consider the function
$\hat{\a}:[0,T]\t  \cP_2(\sR^{n} \t \sR^{n}) \to \bA$ satisfying for all 
$(t,\chi)\in [0,T]\t  \cP_2(\sR^n\t \sR^n)$ that 
$$
\hat{\a}(t,\chi)=\argmin_{a\in \bA} h(t,\chi,a),
\q
\textnormal{
with
$h(t,\chi,a)\coloneqq \tilde{\sE}[H^{\textrm{re}}(t,\tilde{X},a,\tilde{\sP}_{(\tilde{X},a)},\tilde{Y})]$,
 }
$$
where  $(\tilde{X},\tilde{Y})\in L^2(\tilde{\Om},\tilde{\cF},\tilde{\sP};\sR^n\t \sR^n)$
has distribution $\chi$.

We first show the function $\hat{\a}$ is well-defined.
By using the linearity of $b$ and the convexity of $f$ in (H.\ref{assum:mfcE}), we see that 
the map
$\bA\ni a\mapsto h(t,\chi,a)\in \sR$ is strongly convex with factor $\lambda_1+\lambda_2>0$,
which admits a unique minimizer on the nonempty closed convex set $\bA$.
The measurability of $\hat{\a}$ follows from \cite[Lemma 3.3]{carmona2018a}.

Then, we prove that the function $\hat{\a}$ satisfies 
the optimality condition \eqref{eq:opti_markov}.
By using (H.\ref{assum:mfcE}), we have for almost all $(t,\om)\in [0,T]\t \tilde{\Om}$ that 
the mapping $\sR^n\ni a\mapsto  H^{\textrm{re}}(t,\tilde{X}(\om),a,\tilde{\sP}_{(\tilde{X},a)},\tilde{Y}(\om))$
is differentiable with the derivative being   at most of linear growth in $(\tilde{X}(\om),\tilde{Y}(\om))$.
Hence, Lebesgue's differentiation theorem shows that 
$h$ is differentiable  with respect to $a$ with 
the derivative 
\begin{align}\l{eq:p_ah}
\p_a h(t,\chi,a)=
 \tilde{\sE}[\p_aH^{\textrm{re}}(t,\tilde{X},a,\tilde{\sP}_{(\tilde{X},a)},\tilde{Y})]
 + \tilde{\sE}\big[
 \bar{\sE}[
 \p_\nu H^{\textrm{re}}(t,\tilde{X},a,\tilde{\sP}_{(\tilde{X},a)},\tilde{Y})(\bar{X},a)
 ]\big],
\end{align}
where $\bar{X}\in L^2(\bar{\Om},\bar{\cF},\bar{\sP};\sR^n)$ has distribution $\tilde{\sP}_{\tilde{X}}$.

Observe  that  $b_2\equiv 0$ and 
the structural condition of $f$ imply that
the reduced Hamiltonian \eqref{eq:hamiltonian_re}  is given by
$H^{\textrm{re}}(t,x,a,\eta,y)=\la b_0(t)+b_1(t)x +b_3(t)\bar{\eta}, y\ra +f_1(t,x,\pi_1\sharp \eta,\pi_2\sharp \eta)
+f_2(t,a,\pi_1\sharp \eta,\pi_2\sharp \eta)$.
Hence, 
for all $(t,x,y,a)$,
$\p_aH^{\textrm{re}}(t,x,a,\tilde{\sP}_{(\tilde{X},a)},y)
=\p_a f_2(t,a,\tilde{\sP}_{\tilde{X}},\bm{\delta}_a)$
and 
$\p_\nu H^{\textrm{re}}(t,x,a,\eta,y)(\cdot)$ can be chosen as a function
defined only on $\sR^k$ (not on $\sR^n\t \sR^k$ as in the general setting),
which simplifies \eqref{eq:p_ah} into:
\bb\l{eq:p_ah_2}
\p_a h(t,\chi,a)=
\p_aH^{\textrm{re}}(t,x,a,\tilde{\sP}_{(\tilde{X},a)},y)
 + \tilde{\sE}\big[
 \p_\nu H^{\textrm{re}}(t,\tilde{X},a,\tilde{\sP}_{(\tilde{X},a)},\tilde{Y})(a)
\big]
\ee
 for all $(t,\chi,a)\in[0,T]\t  \cP_2(\sR^{n} \t \sR^{n}) \t \bA$. 
Consequently, 
one can conclude from 
the fact that $\hat{\a}(t,\chi)$ is a minimizer
and  the identity that $\tilde{\sP}_{(\tilde{X},\hat{\a}(t,\chi))}=\phi(t,\chi)$
(see \eqref{eq:PushF})
 that  the function $\hat{\a}$ satisfies 
the optimality condition \eqref{eq:opti_markov}:
for all $(t,\chi,a)\in [0,T]\t  \cP_2(\sR^n\t \sR^n)\t \bA$,
\begin{align}\l{eq:mfcE_optimal_a}
\begin{split}
0
&\ge \la \p_a h(t,\chi,\hat{\a}(t,\chi)),\hat{\a}(t,\chi)-a\ra
\\
&=\la 
\p_aH^{\textrm{re}}(t,x,\hat{\a}(t,\chi),\phi(t,\chi),y)
\\
&\q 
+ \tilde{\sE}[\p_\nu H^{\textrm{re}}(t,\tilde{X},\hat{\a}(t,\chi),\phi(t,\chi),\tilde{Y})(\hat{\a}(t,\chi))]
 ,\hat{\a}(t,\chi)-a\ra,
\end{split}
\end{align}
whenever  $(\tilde{X},\tilde{Y})\in L^2(\tilde{\Om},\tilde{\cF},\tilde{\sP};\sR^n\t \sR^n)$
has  distribution $\chi$.
Note that
in the present setting
 \eqref{eq:opti_markov}
is independent of $(x,y)$
since $\p_aH^{\textrm{re}}(t,x,\hat{\a}(t,\chi),\phi(t,\chi),y)
=\p_a f_2(t,a,\tilde{\sP}_{\tilde{X}},\bm{\delta}_{\hat{\a}(t,\chi)})$.

Finally, we establish 
the spatial and time regularity of $\hat{\a}$.
Similar to \cite[Lemma 3.3]{carmona2018a},
by using
 $(\lambda_1+\lambda_2)$-strong convexity of $a\mapsto h(t,\chi,a)$,
we can show for all $t\in[0,T]$ that
$|\hat{\a}(t,\bm{\delta}_{0_{n+n}})-a_0|\le 
(\lambda_1+\lambda_2)^{-1}
|\p_a h(t,\bm{\delta}_{0_{n+n}},a_0)|$,
where $a_0$ an arbitrary element in $\bA$.
Then \eqref{eq:p_ah_2} and (H.\ref{assum:mfcE}(\ref{item:mfcE_growth})) imply
the estimate that 
$\|\hat{\a}(\cdot,\bm{\delta}_{0_{n+n}})\|_{L^\infty(0,T)}<\infty$. 
Now let $(t,\chi),(t',\chi')\in [0,T]\t \cP_2(\sR^n\t \sR^n)$,
$\hat{a}=\hat{\a}(t,\chi)$, $\hat{a}'=\hat{\a}(t',\chi')$
and $(\tilde{X},\tilde{Y}),(\tilde{X}',\tilde{Y}')\in L^2(\tilde{\Om},\tilde{\cF},\tilde{\sP};\sR^n\t \sR^n)$
have  distributions $\chi$ and $\chi'$, respectively.
By following 
a  similar argument as that for 
 Proposition \ref{prop:mfc_hat_a},
one can deduce from 
the  $(\lambda_1+\lambda_2)$-strong convexity of $a\mapsto h(t,\chi,a)$,
the expression of $\p_{a}h$ in \eqref{eq:p_ah_2} 
and the Lipschitz continuity of $(\p_a f_2,\p_\nu f_1,\p_\nu f_2)$
in (H.\ref{assum:mfcE}(\ref{item:mfcE_lipschitz}))
 that 
\begin{align*}
&|\hat{a}'-\hat{a}|
\le
C|\p_{a}h(t,\chi,\hat{a}')-  \p_{a}h(t',\chi',\hat{a}')|
\\
&\le
C\big(|\p_a f_2(t,\hat{a}',\tilde{\sP}_{\tilde{X}},\bm{\delta}_{\hat{a}'})
-\p_a f_2(t',\hat{a}',\tilde{\sP}_{\tilde{X}'},\bm{\delta}_{\hat{a}'})|
\\
&\q
+|\tilde{\sE}\big[
 \p_\nu H^{\textrm{re}}(t,\tilde{X},\hat{a}',\tilde{\sP}_{(\tilde{X},\hat{a}')},\tilde{Y})(\hat{a}')
 \big]
 -\tilde{\sE}\big[
 \p_\nu H^{\textrm{re}}(t',\tilde{X}',\hat{a}',\tilde{\sP}_{(\tilde{X}',\hat{a}')},\tilde{Y}')(\hat{a}')
 \big]|
 \big)
 \\
&\le
C\Big(
\cW_2(\chi,\chi')+
|\p_a f_2(t,\hat{a}',\tilde{\sP}_{\tilde{X}},\bm{\delta}_{\hat{a}'})
-\p_a f_2(t',\hat{a}',\tilde{\sP}_{\tilde{X}},\bm{\delta}_{\hat{a}'})|
+|b_3(t)-b_3(t')|\tilde{\sE}[|\tilde{Y}|]
\\
&\q
\big|
\tilde{\sE}\big[
 \p_\nu f_1(t,\tilde{X},\tilde{\sP}_{\tilde{X}},\bm{\delta}_{\hat{a}'})(\hat{a}')
 \big]
 -\tilde{\sE}\big[
 \p_\nu f_1(t',\tilde{X},\tilde{\sP}_{\tilde{X}},\bm{\delta}_{\hat{a}'})(\hat{a}')
 \big]
 \big|
 \\
&\q +
 \big|
\p_\nu f_2(t,\hat{a}',\tilde{\sP}_{\tilde{X}},\bm{\delta}_{\hat{a}'})(\hat{a}')
-\p_\nu f_2(t',\hat{a}',\tilde{\sP}_{\tilde{X}},\bm{\delta}_{\hat{a}'})(\hat{a}')
\big|
 \Big),
\end{align*}
where the constant $C$ is independent of $t,t',\chi,\chi'$.
Setting $t'=t$ in  the above estimate gives us that 
$|\hat{\a}(t,\chi)-\hat{\a}(t,\chi')|\le C\cW_2(\chi,\chi')$,
which along with 
$\|\hat{\a}(\cdot,\bm{\delta}_{0_{n+n}})\|_{L^\infty(0,T)}<\infty$
 implies 
$|\hat{\a}(t,\chi)|\le C(1+\|\chi\|_2)$.
The desired time regularity of $\hat{\a}$ then follows from the additional assumptions on the time regularity of coefficients.
\end{proof}

\end{Example}

\begin{Example}\l{example:lq}
In this example, we verify  (H.\ref{assum:mfcE_hat})  for 
  MFC problems whose  running costs
 are quadratic in the control variables,
 which extend  the commonly studied linear-quadratic models (see e.g.~\cite{acciaio2019,carmona2019,lauriere2020})
 to cost functions that are convex in the state variables.
 

For notational simplicity, 
 we consider a one-dimensional problem \eqref{eq:mfcE} with  $n=k=d=1$, an action set $\bA=\sR$ 
 and a 
  running cost function of the following form
\begin{align}\l{eq:f_quadratic}
&f(t,x,a,\eta) = \frac{1}{2} \left( f_1(t,x,\pi_1\sharp \eta) +q(t)a^2+
\bar{q}(t)\big(a-r(t)\bar{a}\big)^2
+2 c(t)xa \right),
\end{align}
where
$\pi_1\sharp\eta$ is 
 the first  marginal of $\eta$,
 $\bar{a} = \int a \mathrm{d} \eta(x,a)$, 
$ q,\bar{q}, r, c\in L^\infty(0,T;\sR)$,
 $q\ge \lambda_1>0$, $\bar{q}\geq 0$
and $f_1:[0,T]\t \sR\t \cP_2(\sR)\to \sR$ is a suitable function
such that 
the running cost $f$ satisfies (H.\ref{assum:mfcE}). 
Similar arguments can be adapted to verify (H.\ref{assum:mfcE_hat}) for 
multi-dimensional  running costs with a general quadratic dependence on the control variables.

In the present setting, we see   that 
the drift coefficient of \eqref{eq:mfcE_fwd} reads as
\begin{align*}
b(t,x,a,\eta)&=b_0(t)+b_1(t)x + b_2(t)a +  \beta(t) \bar{x} + \gamma(t) \bar{a},
\end{align*}
where 
$\bar{x} =  \int x \, \mathrm{d}\eta(x,a)$ and
$\beta,\gamma\in L^\infty(0,T;\sR)$ denote the first and second component of the function $b_3$ in  (H.\ref{assum:mfcE}(\ref{item:mfcE_lin})), respectively.
The definition of the reduced Hamiltonian \eqref{eq:hamiltonian_re} and the openness of the set $\bA$ imply that 
it suffices to find
a function $\hat{\a}:[0,T]\t \sR\t \sR\t \cP_2(\sR\t\sR)\to \bA$
such that 
for all $t\in [0,T]$, $X_t,Y_t\in L^2(\Om;\sR)$,
we have that $\a_t=\hat{\a}(t,X_t,Y_t,\sP_{(X_t,Y_t)})$ satisfies 
\begin{align}\label{eq:opti_cond}
b_2(t)Y_t +\gamma(t) \sE[Y_t] + \big(q(t) + \bar{q}(t)\big)\a_t + \bar{q}(t) r(t)(r(t) -2)\mathbb{E}[\a_t] + c(t) X_t = 0.
\end{align}
Taking expectations on both sides of \eqref{eq:opti_cond} gives  us that
\begin{align}\label{eq:ex_alpha}
\sE[\a_t] = \frac{-(b_2(t) + \gamma(t) ) \sE[Y_t] - c(t)\sE[X_t]}{q(t) + \bar{q}(t)\big(r(t)-1\big)^2},
\end{align}
which is well-defined since
 $q(t)\ge \lambda_1>0$
 and $\bar{q}(t)\ge 0$.
Then, by substituting \eqref{eq:ex_alpha}   into  \eqref{eq:opti_cond}, we 
see that it suffices to define 
 $\hat{\a}:[0,T]\t \sR\t \sR\t \cP_2(\sR\t\sR)\to \bA$
 to be the function
satisfying
for all $(t,x,y,\chi)\in[0,T]\t \sR\t \sR\t \cP_2(\sR\t\sR)$ that
\begin{align*}
\hat{\a}(t,x,y,\chi) = \frac{- c(t)x-b_2(t)y   + \psi(t) \int_\sR x\,\d\chi(x,y) +( -\gamma(t)  + \zeta(t)) \int_\sR y\,\d\chi(x,y)  }
{q(t) + \bar{q}(t)},
\end{align*}
with the coefficients 
\begin{align*}
& \psi(t) \coloneqq \frac{c(t) \bar{q}(t) r(t)(r(t) -2)}{q(t) + \bar{q}(t)\big(r(t)-1\big)^2},
\q
 \zeta(t)\coloneqq \frac{(b_2(t) + \gamma(t))\bar{q}(t) r(t)(r(t) -2)}{q(t) + \bar{q}(t)\big(r(t)-1\big)^2}.
\end{align*}
The fact that $q \ge \lambda_1>0, \bar{q}\ge 0$, and the boundedness of coefficients imply that $\hat{\a}$ 
is well-defined and 
satisfies 
(H.\ref{assum:mfcE_hat}(\ref{item:ex})).
By further assuming that 
the functions
$b_2,\gamma,q,\bar{q},r,c$ are $1/2$-H\"{o}lder continuous on $[0,T]$,
we can show  that $\hat{\a}$ satisfies 
(H.\ref{assum:mfcE_hat}(\ref{item:HC})).

Observe that in the present setting, the feedback map $\hat{\a}$ is 
independent of $(x,y)$ if and only if $b_2\equiv c\equiv 0$.
This agrees with the general condition in Proposition \ref{prop:control_measure}
under which
the optimal control of \eqref{eq:mfcE} is  deterministic.

\end{Example}

With (H.\ref{assum:mfcE_hat}(\ref{item:ex})) at  hand, 
we can express the
coupled MV-FBSDE 
\eqref{eq:mfcE_bsde_nonMarkov}
in an equivalent form
that is easier to analyze.
We shall seek
 a tuple of processes $(X^{\hat{\a}},Y^{\hat{\a}},Z^{\hat{\a}},\hat{\a})\in \cS^2(\sR^n)\t\cS^2(\sR^n)\t \cH^2(\sR^{n\t d})\t \cA$ satisfying for all $t\in [0,T]$ that
$\hat{\a}_t= \hat{\a}(t,
X^{\hat{\a}}_t,Y^{\hat{\a}}_t,\sP_{(X^{\hat{\a}}_t,Y^{\hat{\a}}_t)})$
and
\begin{align}\l{eq:mfc_fbsde_hat2}
\begin{split}
\mathrm{d}X^{\hat{\a}}_t&=
b\big(t,X^{\hat{\a}}_t,\hat{\a}_t,\sP_{(X^{\hat{\a}}_t,\hat{\a}_t)}\big)\, \d t
+\sigma(t,X^{\hat{\a}}_t,\sP_{X^{\hat{\a}}_t})\, \d W_t,
\\
\mathrm{d}Y^{\hat{\a}}_t&=-\big(\p_x H(t,X^{\hat{\a}}_t,\hat{\a}_t,\sP_{(X^{\hat{\a}}_t,\hat{\a}_t)},Y^{\hat{\a}}_t,Z^{\hat{\a}}_t)
\\
&\q
+\tilde{\sE}[\p_\mu H(t,\tilde{X}^{\hat{\a}}_t,\tilde{{\hat{\a}}},\sP_{(X^{\hat{\a}}_t,\hat{\a}_t)},\tilde{Y}^{\hat{\a}}_t,\tilde{Z}^{\hat{\a}}_t)(X^{\hat{\a}}_t,\hat{\a}_t)]\big)\,\d t
+Z^{\hat{\a}}_t\,\d W_t,
\\
X^{\hat{\a}}_0&=\xi_0,\q Y^{\hat{\a}}_T=\p_x g(X^{\hat{\a}}_T,\sP_{X^{\hat{\a}}_T})+\tilde{\sE}[\p_\mu g(\tilde{X}^{\hat{\a}}_T,\sP_{X^{\hat{\a}}_t})(X^{\hat{\a}}_T)],
\end{split}
\end{align}
where
$(\tilde{X}^{\hat{\a}},\tilde{Y}^{\hat{\a}}, \tilde{Z}^{\hat{\a}},\tilde{{\hat{\a}}})$
is an  independent copy
 of $({X}^{\hat{\a}},{Y}^{\hat{\a}}, {Z}^{\hat{\a}},{{\hat{\a}}})$ defined on a space 
$L^2(\tilde{\Om},\tilde{\cF},\tilde{\sP})$. 
Note that \eqref{eq:mfc_fbsde_hat2} can be equivalently formulated 
as follows:
for all $t\in [0,T]$,
\begin{subequations}\label{eq:mfc_fbsde2}
\begin{alignat}{2}
\mathrm{d} X_t&=\hat{b}(t,X_t,Y_t,\sP_{(X_t,Y_t)})\,\d t +\sigma (t,X_t, \sP_{X_t})\, \d W_t, 
\q
&& X_0=\xi_0,
\l{eq:mfc_fwd_fb2}
\\
\mathrm{d} Y_t&=-\hat{f}(t,X_t,Y_t,Z_t, \sP_{(X_t,Y_t,Z_t)})\,\d t+Z_t\,\d W_t,
\q && Y_T=\hat{g}(X_T,\sP_{X_T})
\l{eq:mfc_bwd_fb2}
\end{alignat}
\end{subequations}
with  coefficients defined as follows:
for all $(t,x,y,z,a,\mu, \chi, \rho)\in [0,T]\t \sR^n\t \sR^n \t \sR^{n \times d} \t \bA \t \cP_2(\sR^n)\t \cP_2(\sR^n \t \sR^{n}) \t \cP_2(\sR^n \t \sR^{n} \t  \sR^{n \t d})$,
\begin{align}\l{eq:mfc_coefficients2}
\begin{split}
&\hat{b}(t,x,y,\chi)=b(t,x,\hat{\a}(t,x,y,\chi),\phi(t,\chi)),
\\
&\hat{f}(t,x,y,z,\rho)
\\\q&=\p_x H(t,x,\hat{\a}(t,x,y,\pi_{1,2} \sharp \rho),\phi(t,\pi_{1,2} \sharp \rho),y,z)
\\
&\q
 +\int_{E}
\p_\mu H(t,\tilde{x},\hat{\a}(t,\tilde{x},\tilde{y},\pi_{1,2} \sharp \rho),\phi(t,\pi_{1,2} \sharp \rho),\tilde{y},\tilde{z})
(x,\hat{\a}(t,{x},{y},\pi_{1,2} \sharp \rho))\,\d \rho(\tilde{x},\tilde{y},\tilde{z}),
\\
&\hat{g}(x,\mu)=
\p_x g(x,\mu)+\int_{\sR^n}\p_\mu g(\tilde{x},\mu)(x)\,\d \mu(\tilde{x}),
\end{split}
\end{align}
where 
$E\coloneqq \sR^n\t \sR^n\t \sR^{n\t d}$,
$\phi(t,\chi)$ is defined as in (H.\ref{assum:mfcE_hat}(\ref{item:ex})) and
$\pi_{1,2} \sharp\rho\coloneqq\rho(\cdot \t \sR^{n \t d})$ is the marginal of the measure $\rho$ on $\sR^n\t\sR^n$.
In the subsequent analysis,
we shall show that 
\eqref{eq:mfc_fbsde2}
(or equivalently \eqref{eq:mfc_fbsde_hat2})
 admits a unique solution
and then construct an optimal control for \eqref{eq:mfcE}
by using the function $\hat{\a}$ in (H.\ref{assum:mfcE_hat}(\ref{item:ex}));
see Theorem \ref{TH:ControlRegularity} for details.

\section{Regularity of mean field controls}\l{sec:regularity_mfcE}

In this section, we study the regularity of solutions to  
the MV-FBSDE
 \eqref{eq:mfc_fbsde2}.
In particular,
we shall establish that \eqref{eq:mfc_fbsde2} admits a unique 
$1/2$-H\"{o}lder continuous solution in 
$ \cS^2(\sR^n)\t\cS^2(\sR^n)\t \cH^2(\sR^{n\t d})$,
which subsequently enables us to show that the MFC problem admits a unique $1/2$-H\"{o}lder continuous optimal control in $\cA$.

We start by showing that  the coefficients $(\hat{b},\sigma,\hat{f},\hat{g})$ of the MV-FBSDE \eqref{eq:mfc_fbsde2}
are  Lipschitz continuous with respect to the spatial variables (uniformly in the time variable),
and satisfy     a general monotonicity condition. 
{
The detailed steps for the proofs of the following propositions can be found in Appendix \ref{appendix:hat_a_existence}.}

\begin{Proposition}\label{prop:mcfE1}
Suppose (H.\ref{assum:mfcE}) and (H.\ref{assum:mfcE_hat}(\ref{item:ex})) hold,
and let the functions 
$(\hat{b},\hat{f},\hat{g})$ 
be defined as in \eqref{eq:mfc_coefficients2}.
Then there exists a constant $C\ge 0$
satisfying 
for all $t\in [0,T]$ that
 the functions 
$(\hat{b}(t,\cdot),\sigma(t,\cdot),\hat{f}(t,\cdot),\hat{g}(t,\cdot))$ 
are  $C$-Lipschitz continuous in all  variables
and satisfy the estimate
 $
 \|\hat{b}(\cdot,{0},{0},\bm{\delta}_{{0}_{n+n}})\|_{L^2(0,T)}
 +\|\sigma(\cdot,0,\bm{\delta}_{0_{n}})\|_{L^\infty(0,T)}
 +
 \|\hat{f}(\cdot,{0},{0}, {0},\bm{\delta}_{{0}_{n+n+nd}})\|_{L^\infty(0,T)}\le C$.

\end{Proposition}


\begin{Proposition}\label{prop:mcfE2}
Suppose (H.\ref{assum:mfcE}) and (H.\ref{assum:mfcE_hat}(\ref{item:ex})) hold,
and let the functions 
$(\hat{b},\hat{f},\hat{g})$ 
be defined as in \eqref{eq:mfc_coefficients2}.
Then the functions $(\hat{b},\sigma,\hat{f},\hat{g})$ 
satisfy 
for all 
$t\in [0,T]$, $i\in \{1,2\}$, 
 $\Theta_i\coloneqq (X_i,Y_i,Z_i)\in L^2(\Om; \sR^n\t \sR^m\t \sR^{m\t d})$
 that
$ \sE[\la \hat{g}(X_1,\sP_{X_1})-\hat{g}(X_2,\sP_{X_2}), X_1-X_2\ra]
\ge  0$ and
\begin{align}\l{eq:monotonicity}
\begin{split}
&\sE[\la \hat{b}(t,X_1,Y_1,\sP_{(X_1,Y_1)})-\hat{b}(t,X_2,Y_2,\sP_{(X_2,Y_2)}), Y_1-Y_2\ra]
 \\
&\quad 
+\sE[\la \sigma(t,X_1,\sP_{X_1})-\sigma(t,X_2,\sP_{X_2}),  Z_1-Z_2\ra]
\\
&\quad 
+ \sE[\la -\hat{f}(t,\Theta_1,\sP_{\Theta_1})+\hat{f}(t,\Theta_2,\sP_{\Theta_2}), X_1-X_2\ra] 
\\
&
 \le 
 -2(\lambda_1 + \lambda_2)
  \|\hat{\a}(t,X_1,Y_1,\sP_{(X_1,Y_1)})-\hat{\a}(t,X_2,Y_2,\sP_{(X_2,Y_2)})\|^2_{L^2},
\end{split}
\end{align}
with the constants $\lambda_1,\lambda_2$ in  (H.\ref{assum:mfcE}(\ref{item:mfcE_convex})).
\end{Proposition}

{
We then adapt the method of continuation in \cite{bensoussan2015,carmona2015}
 to the present setting,
 and establish the well-posedness and stability of \eqref{eq:mfc_fbsde2}. 
To do so, we first present  a stability result for the following family of MV-FBSDEs:
for $t\in [0,T]$,
\begin{align}\l{eq:moc}
\begin{split}
\d X_t&=(\lambda \hat{b}(t,X_t,Y_t,\sP_{(X_t,Y_t)})+\cI^{\hat{b}}_t)\,\d t +
(\lambda\sigma (t,X_t, \sP_{X_t})+\cI^\sigma_t)\, \d W_t, 
\\
\d Y_t&=-(\lambda \hat{f}(t,X_t,Y_t,Z_t, \sP_{(X_t,Y_t,Z_t)})+\cI^{\hat{f}}_t)\,\d t+Z_t\,\d W_t,
\\
X_0&=\xi,\q Y_T=\lambda \hat{g}(X_T,\sP_{X_T})+\cI^{\hat{g}}_T,
\end{split}
\end{align}
where  $\lambda\in [0,1]$, 
$\xi\in L^2(\cF_0;\sR^n)$,
$(\cI^{\hat{b}},\cI^\sigma,\cI^{\hat{f}})\in \cH^2(\sR^n\t\sR^{n\t d}\t \sR^n)$ and $\cI^{\hat{g}}_T\in L^2(\cF_T;\sR^n)$
are given.
The proof is based on Propositions \ref{prop:mcfE1} and \ref{prop:mcfE2},
whose detail
 is presented in Appendix \ref{appendix:hat_a_existence}.
\begin{Lemma}\l{lemma:mono_stab}
Suppose (H.\ref{assum:mfcE}) and (H.\ref{assum:mfcE_hat}(\ref{item:ex})) hold,
and let the functions 
$(\hat{b},\hat{f},\hat{g})$ 
be defined as in \eqref{eq:mfc_coefficients2}.
 Then, there exists a constant $C>0$ such that,  
 for all $\lambda_0\in [0,1]$,
for every   
${\Theta}\coloneqq(X,Y, Z)\in  \cS^2(\sR^n) \t \cS^2(\sR^n) \t \cH^2(\sR^{n\t d})$
satisfying \eqref{eq:moc}
with 
$\lambda=\lambda_0$,
  functions $(\hat{b},\sigma,\hat{f},\hat{g})$
  and some
$(\cI^{\hat{b}},\cI^\sigma,\cI^{\hat{f}})\in \cH^2(\sR^n\t\sR^{n\t d}\t \sR^n)$,
 $\cI^{\hat{g}}_T\in L^2(\cF_T;\sR^n)$,
 $\xi\in L^2(\cF_0;\sR^n)$,
 and for every 
$ \bar{\Theta}\coloneqq(\bar{X},\bar{Y}, \bar{Z})\in  \cS^2(\sR^n) \t \cS^2(\sR^n) \t \cH^2(\sR^{n\t d})$
satisfying  \eqref{eq:moc}
with 
$\lambda=\lambda_0$,
another 4-tuple of Lipschitz functions $(\bar{b},\bar{\sigma},\bar{f},\bar{g})$ 
 and some 
$(\bar{\cI}^b,\bar{\cI}^\sigma,\bar{\cI}^f)\in \cH^2(\sR^n\t\sR^{n\t d}\t \sR^n)$,
 $\bar{\cI}^g_T\in L^2(\cF_T;\sR^n)$,
 $\bar{\xi}\in L^2(\cF_0;\sR^n)$, we have that 
 \begin{align}\l{eq:mono_stab}
 \begin{split}
 &\|X-\bar{X}\|_{\cS^2}^2+ \|Y-\bar{Y}\|_{\cS^2}^2+ \|Z-\bar{Z}\|_{\cH^2}^2
 \\
& \le C\bigg\{\|\xi-\bar{\xi}\|_{L^2}^2+
\|\lambda_0 (\hat{g}(\bar{X}_T,\sP_{\bar{X}_T})-\bar{g}(\bar{X}_T,\sP_{\bar{X}_T}))+\cI^{\hat{g}}_T-\bar{\cI}^g_T\|_{L^2}^2
\\
&\quad
+\|\lambda_0 (\hat{b}(\cdot,\bar{X}_\cdot,\bar{Y}_\cdot,\sP_{(\bar{X},\bar{Y})_\cdot})-\bar{b}(\cdot,\bar{X}_\cdot,\bar{Y}_\cdot,\sP_{(\bar{X},\bar{Y})_\cdot}))
+\cI^{\hat{b}}-\bar{\cI}^b\|_{\cH^2}^2
\\
&\quad
+\|\lambda_0 (\sigma(\cdot,\bar{X}_\cdot,\sP_{\bar{X}_\cdot})-\bar{\sigma}(\cdot,\bar{X}_\cdot,\sP_{\bar{X}_\cdot}))
+\cI^\sigma-\bar{\cI}^\sigma\|_{\cH^2}^2
\\
&\quad
+\|\lambda_0 (\hat{f}(\cdot,\bar{\Theta}_\cdot,\sP_{\bar{\Theta}_\cdot})-\bar{f}(\cdot,\bar{\Theta}_\cdot,\sP_{\bar{\Theta}_\cdot}))
+\cI^{\hat{f}}-\bar{\cI}^f\|_{\cH^2}^2
\bigg\}.
\end{split}
  \end{align} 
\end{Lemma}
}

Now we are ready to  establish the well-posedness  and stability   of \eqref{eq:mfc_fbsde2}.

\begin{Theorem}\label{TH:fbsde_wellposedness}
{
Suppose (H.\ref{assum:mfcE}) and (H.\ref{assum:mfcE_hat}(\ref{item:ex})) hold. Then, for all  
$t\in [0,T]$ and 
 $\xi\in L^2(\cF_t;\sR^n)$,
there exists a unique triple
$(X^{t,\xi},Y^{t,\xi},Z^{t,\xi}) \in  \cS^2(t,T;\sR^n) \t \cS^2(t,T;\sR^n) \t \cH^2(t,T;\sR^{n \t d})$ 
satisfying \eqref{eq:mfc_fbsde2} on $[t,T]$ with  
the initial condition $X^{t,\xi}_t=\xi$.
Moreover, there exists a constant $C>0$ such that 
it holds for all $t\in [0,T]$  and $\xi,\xi'\in L^2(\cF_t;\sR^n)$ that 
$\|Y^{t,\xi}_t-Y^{t,\xi'}_t\|_{L^2}\le C\|\xi-\xi'\|_{L^2}$,
and
$\|X^{t,\xi}\|_{\cS^2(t,T;\sR^n)}+\|Y^{t,\xi}\|_{ \cS^2(t,T;\sR^n)}+\|Z^{t,\xi}\|_{\cH^2(t,T;\sR^{n\t d})}\le C(1+\|\xi\|_{L^2})$.}
\end{Theorem}

\begin{proof}
{
We shall establish the  well-posedness, stability and \textit{a priori} estimates
 for \eqref{eq:mfc_fbsde2} with
an initial time
 $t=0$ and initial state $\xi_0\in L^2(\cF_0;\sR^n)$
 by applying Lemma \ref{lemma:mono_stab}. Similar arguments apply to a general 
 initial time 
 $t\in [0,T]$ and initial state $\xi\in L^2(\cF_t;\sR^n)$.

Let us start by proving 
the unique solvability of \eqref{eq:mfc_fbsde2} with a given $\xi_0\in L^2(\cF_0;\sR^n)$.
To simplify the notation, 
for every $\lambda_0\in [0,1]$, we say $(\cP_{\lambda_0})$ holds if 
 for any 
 $\xi\in  L^2(\cF_0;\sR^n)$,
$(\cI^{\hat{b}},\cI^\sigma,\cI^{\hat{f}})\in \cH^2(\sR^n\t\sR^{n\t d}\t \sR^n)$ and $\cI^{\hat{g}}_T\in L^2(\cF_T;\sR^n)$,
\eqref{eq:moc} with $\lambda=\lambda_0$ 
admits a unique solution in $\sB\coloneqq \cS^2(\sR^n) \t \cS^2(\sR^n) \t \cH^2(\sR^{n\t d})$.
It is clear that $(\cP_{0})$ holds since \eqref{eq:moc} is decoupled. 
Now we show 
 there exists a constant $\delta>0$, such that 
 if $(\cP_{\lambda_0})$   holds for some $\lambda_0\in [0,1)$,
 then $(\cP_{\lambda'_0})$ also holds for  all $\lambda_0'\in (\lambda_0,\lambda_0+\delta]\cap [0,1]$.
Note that this claim along with  the method of continuation
 implies the desired unique solvability of \eqref{eq:mfc_fbsde2} (i.e., \eqref{eq:moc} with $\lambda=1$,
$(\cI^{\hat{b}},\cI^\sigma,\cI^{\hat{f}},\cI^{\hat{g}}_T)=0$, $\xi=\xi_0$).

To establish the desired claim, let  $\lambda_0\in [0,1)$
be a constant for which $(\cP_{\lambda_0})$ holds,
$\eta\in [0,1]$ and 
$(\tilde{\cI}^{\hat{b}},\tilde{\cI}^\sigma,\tilde{\cI}^{\hat{f}})\in \cH^2(\sR^n\t\sR^{n\t d}\t \sR^n)$, $\tilde{\cI}^{\hat{g}}_T\in L^2(\cF_T;\sR^n)$, 
$\xi\in  L^2(\cF_0;\sR^n)$
be
arbitrarily given coefficients. Then,
we introduce the following mapping $\Xi:\sB\to \sB$
such that 
for all $\Theta=(X,Y,Z)\in \sB$, $\Xi(\Theta)\in \sB$ is the solution to \eqref{eq:moc} 
with $\lambda=\lambda_0$,
$\cI^{\hat{b}}_t=\eta \hat{b}(t,X_t,Y_t\sP_{(X_t,Y_t)})+\tilde{\cI}^{\hat{b}}_t$,
$\cI^\sigma_t=\eta \sigma (t,X_t, \sP_{X_t})+\tilde{\cI}^\sigma_t$,
$\cI^{\hat{f}}_t=\eta \hat{f}(t,\Theta_t, \sP_{\Theta_t})+\tilde{\cI}^{\hat{f}}_t$
and $\cI^{\hat{g}}_T=\eta \hat{g}(X_T,\sP_{X_T})+\tilde{\cI}^{\hat{g}}_T$,
which is well-defined due to the fact that $\lambda_0\in [0,1)$ satisfies the induction hypothesis.
Observe that
by setting $(\bar{b},\bar{\sigma},\bar{f},\bar{g})=(\hat{b},\sigma,\hat{f},\hat{g})$ in Lemma \ref{lemma:mono_stab},
 we see that there exists a constant $C>0$, independent of $\lambda_0$, such that it holds  for all $\Theta,\Theta'\in \sB$ that
 \begin{align*}
 \begin{split}
 &\|\Xi(\Theta)-\Xi({\Theta}')\|_{\sB}^2
 \\
& \le C\bigg\{\|\eta (\hat{g}({X}_T,\sP_{{X}_T})-\hat{g}({X}'_T,\sP_{{X}'_T}))\|_{L^2}^2
+\|\eta (\hat{b}(\cdot,{X}_\cdot,{Y}_\cdot,\sP_{{(X,Y)}_\cdot})-\hat{b}(\cdot,{X}'_\cdot,{Y}'_\cdot,\sP_{{X',Y'}_\cdot}))
\|_{\cH^2}^2
\\
&\quad
+\|\eta (\sigma(\cdot,{X}_\cdot,\sP_{{X}_\cdot})-{\sigma}(\cdot,{X}'_\cdot,\sP_{{X'}_\cdot}))
\|_{\cH^2}^2
+\|\eta (\hat{f}(\cdot,{\Theta}_\cdot,\sP_{{\Theta}_\cdot})-\hat{f}(\cdot,{\Theta}'_\cdot,\sP_{{\Theta}'_\cdot}))
\|_{\cH^2}^2
\bigg\}
\\
&\le C\eta^2\|\Theta-{\Theta}'\|_{\sB}^2,
\end{split}
  \end{align*}
which shows that $\Xi$ is a contraction when $\eta$ is sufficiently small (independent of $\lambda_0$),
and subsequently leads to  the desired claim due to Banach's fixed point theorem.

For any given $\xi, \xi'\in L^2(\cF_0;\sR^n)$,
the desired stochastic stability of \eqref{eq:mfc_fbsde2} follows directly from 
 Lemma \ref{lemma:mono_stab}
 by setting 
$\lambda=1$,
$(\bar{b},\bar{\sigma},\bar{f},\bar{g})=(\hat{b},\sigma,\hat{f},\hat{g})$,
$(\bar{\cI}^b,\bar{\cI}^\sigma,\bar{\cI}^f)=(\cI^{\hat{b}},\cI^\sigma,\cI^{\hat{f}})=0$,
 $\bar{\cI}^g_T=\cI^{\hat{g}}_T=0$ and
 $\bar{\xi}=\xi'$.
Moreover,  
for any given $\xi\in L^2(\cF_0;\sR^n)$,
 by setting 
$\lambda=1$,
$(\bar{b},\bar{\sigma},\bar{f},\bar{g})=0$,
$(\bar{\cI}^b,\bar{\cI}^\sigma,\bar{\cI}^f)=(\cI^{\hat{b}},\cI^\sigma,\cI^{\hat{f}})=0$,
 $\bar{\cI}^g_T=\cI^{\hat{g}}_T=0$,
 $\bar{\xi}=0$
 and $(\bar{X},\bar{Y},\bar{Z})=0$
 in  Lemma \ref{lemma:mono_stab},
we can deduce the 
 estimate that
 \begin{align*}
 \begin{split}
 &\|X\|_{\cS^2}^2+ \|Y\|_{\cS^2}^2+ \|Z\|_{\cH^2}^2
 \\
& \le C\bigg\{\|\xi\|_{L^2}^2+
| \hat{g}(0,\bm{\delta}_{{0}_{n}})|^2
+\| \hat{b}(\cdot, 0,\bm{\delta}_{{0}_{n+n}})\|_{L^2(0,T)}^2
+\| \sigma(\cdot, 0,\bm{\delta}_{{0}_{n}})\|_{L^2(0,T)}^2
\\
&\quad
+\| \hat{f}(\cdot, 0,\bm{\delta}_{{0}_{n+n+nd}})\|_{L^2(0,T)}^2
\bigg\}
\le C(1+\|\xi\|_{L^2}^2),
\end{split}
\end{align*}
which shows the desired moment bound of the processes $(X,Y,Z)$.
}
%
\end{proof}

{
 We now give our result concerning the H\"{o}lder regularity of the solutions to \eqref{eq:mfc_fbsde2}.}
\begin{Theorem}\label{TH:fbsde_Regularity}
{
Suppose (H.\ref{assum:mfcE}) and (H.\ref{assum:mfcE_hat}(\ref{item:ex})) hold,
and let $(X,Y,Z)\in \cS^2(\sR^n)\t\cS^2(\sR^n)\t \cH^2(\sR^{n\t d})$ be the unique solution to \eqref{eq:mfc_fbsde2} with initial condition $X_0 = \xi_0\in L^2(\cF_0;\sR^n)$.
Moreover, 
for all  $p\ge 2$,
there exists
a constant $C>0$, depending only on $p$ and the data in 
(H.\ref{assum:mfcE}) and (H.\ref{assum:mfcE_hat}(\ref{item:ex})),
such that
$
\|X\|_{\cS^p}+\|Y\|_{\cS^p}+ \|Z\|_{\cS^p}
\le 
 C
\big(1+
\|\xi_0\|_{L^{p}}
\big)$
and 
$ \sE\left[\sup_{s\le r\le t}|X_r-X_s|^p\right]^{1/p}
 +\sE\left[\sup_{s\le r\le t}|Y_r-Y_s|^p\right]^{1/p}
 \le
C
(1+\|\xi_0\|_{L^{p}})
|t-s|^{{1}/{2}}
$ 
for all $0\le s\le t\le T$.}
\end{Theorem}
\begin{proof}
Let $\xi_0\in L^2(\cF_0;\sR^n)$ be a given initial condition and 
$(X,Y,Z)\in \cS^2(\sR^n)\t\cS^2(\sR^n)\t \cH^2(\sR^{n\t d})$
be the solution to \eqref{eq:mfc_fbsde2}.
By using the pathwise uniqueness and  the Lipschitz stability  of \eqref{eq:mfc_fbsde2}
in Theorem \ref{TH:fbsde_wellposedness},
we can follow the arguments in \cite[Proposition 5.7]{carmona2015}
and deduce that
there exists  a measurable function $v:[0,T]\t \sR^n\to \sR^n$ (depending on $\xi_0$) 
and a constant $C>0$ (independent of $\xi_0$)
such that 
$\sP(\fa t\in [0,T], Y_t=v(t,X_t))=1$
and it holds for all $t\in [0,T]$ and $x,x'\in \sR^n$ that $|v(t,x)-v(t,x')|\le C|x-x'|$ and $|v(t,0)|\le C(1+\|\xi_0\|_{L^2})$.

By substituting the  relation $Y_t=v(t,X_t)$ into \eqref{eq:mfc_fwd_fb2}, 
we can rewrite \eqref{eq:mfc_fbsde2} into  the following decoupled FBSDE:
\begin{subequations}\label{eq:decoupled_fbsde}
\begin{align}
\d X_t&=\bar{b}(t,X_t)\,\d t +\bar{\sigma} (t,X_t)\, \d W_t, 
\q X_0=\xi_0,
\label{eq:decoupled_sde}
\\
\d Y_t&=-\bar{f}(t,X_t,Y_t,Z_t)\,\d t+Z_t\,\d W_t,
\q
 Y_T=\bar{g}(X_T).
\end{align}
\end{subequations}
with coefficients $\bar{b}, \bar{\sigma}, \bar{f}$ and $\bar{g}$ defined as follows:
\begin{alignat*}{2}
\bar{b}(t,x)&\coloneqq 
\hat{b}(t,x,v(t,x),\sP_{(X_t,Y_t)}),
&&
\q \bar{\sigma}(t,x)\coloneqq \sigma (t,x,\sP_{X_t}),
\\
\bar{f}(t,x,y,z)&\coloneqq
\hat{f}(t,x,y,z, \sP_{(X_t,Y_t,Z_t)}),
&&\q 
\bar{g}(x)\coloneqq \hat{g}(x,\sP_{X_T}).
\end{alignat*}
By Proposition \ref{prop:mcfE1} and Theorem \ref{TH:fbsde_wellposedness},
 these coefficients are $C$-Lipschitz continuous in the state variable
 with a constant $C$ independent of  $\xi_0$,
and satisfy the estimates $\int_0^T|\bar{b}(t,0)|^2\,\d t< \infty$, $\sup_{t\in [0,T]}|\bar{\sigma}(t,0)|< \infty$
 and $\int_0^T|\bar{f}(t,0,0,0)|^2\, \d t<\infty$.
Hence, by applying  \cite[Theorem 5.2.2 (i)]{zhang2017} to \eqref{eq:decoupled_fbsde},
we see there exists a constant $C>0$
such that $|Z_t|\le C|\sigma(t,X_t,\sP_{X_t})|$ $\d \sP\otimes \d t$-a.e.. We remark that in \cite{zhang2017} 
the initial state $\xi_0$ is assumed to be deterministic
and 
the coefficients of the FBSDE are assumed to be H\"{o}lder continuous in time. However, the proof 
relies on expressing the process $Z$ in terms of the Malliavin derivatives  of    $X$ and $Y$,
and hence can be extended to the present setting 
where  $\xi_0$ is $\cF_0$-measurable
and the coefficients are   measurable in time
and satisfy the above estimates.

By  the Lipschitz continuity of $v$, the estimate $\sup_{t\in[0,T]}|v(t,0)|\le C(1+\|\xi_0\|_{L^2})$
and  standard moment estimates of \eqref{eq:decoupled_sde},  $\|X\|_{\cS^p}\le C_{(p)}(1+\|\xi_0\|_{L^p})$,
which
along with the relations $Y_t=v(t,X_t)$ and $|Z_t|\le C|\sigma(t,X_t,\sP_{X_t})|$
leads to $\|Y\|_{\cS^p}+\|Z\|_{\cS^p}\le C_{(p)}(1+\|\xi_0\|_{L^p})$.
Moreover,
by \eqref{eq:mfc_fbsde2}, 
H\"{o}lder's inequality
and the  Burkholder-Davis-Gundy  inequality,
the process $X$ satisfies 
for each $p\ge 2$, $t,s\in [0,T]$,
\begin{align*}
&\sE\left[\sup_{s\le r\le t}|X_r-X_s|^p\right]
\\
&\le \sE\bigg[ \bigg(\int_s^t|\hat{b}(r,X_r,Y_r,\sP_{(X_r,Y_r)})|\, \d r \bigg)^p\bigg] 
+\sE\bigg[\sup_{s\le r\le t}\bigg|\int_s^r \sigma(u,X_u,\sP_{X_u})\,\d W_u\bigg|^p\bigg]
\\
&\le
C_{(p)}\bigg\{
(\|\hat{b}(\cdot,0,0,\bm{\delta}_{{0}_{n+n}})\|^p_{L^2(0,T)}+\|(X,Y)\|^p_{\cS^p})|t-s|^{\frac{p}{2}}
+\sE\bigg[\bigg(\int_s^t |\sigma(r,X_r,\sP_{X_r})|^2\,\d r\bigg)^{\frac{p}{2}}\bigg]
\bigg\}
\\
&\le
C_{(p)}\bigg\{
(\|\hat{b}(\cdot,0,0,\bm{\delta}_{{0}_{n+n}})\|^p_{L^2(0,T)}
+\|\sigma(\cdot,0,\bm{\delta}_{{0}_{n}})\|^p_{L^\infty(0,T)}
+\|(X,Y)\|^p_{\cS^p})|t-s|^{\frac{p}{2}}
\bigg\},
\end{align*}
and the process $Y$ satisfies 
for each $p\ge 2$, $t,s\in [0,T]$,
\begin{align*}
&\sE\left[\sup_{s\le r\le t}|Y_r-Y_s|^p\right]
\\
&\le \sE\bigg[ \bigg(\int_s^t|\hat{f}(r,X_r,Y_r,Z_r,\sP_{(X_r,Y_r,Z_r)})|\, \d r \bigg)^p\bigg] 
+\sE\bigg[\sup_{s\le r\le t}\bigg|\int_s^r Z_u\,\d W_u\bigg|^p\bigg]
\\
&\le
C_{(p)}\bigg\{
(\|\hat{f}(\cdot,0,0,0,\bm{\delta}_{{0}_{n+n+nd}})\|^p_{L^2(0,T)}+\|(X,Y,Z)\|^p_{\cS^p})|t-s|^{\frac{p}{2}}
+\sE\bigg[\bigg(\int_s^t|Z_r|^2\,\d r\bigg)^{\frac{p}{2}}\bigg]
\bigg\},
\end{align*}
which together with  Proposition \ref{prop:mcfE1},
 the inequality that 
$\sE[(\int_s^t|Z_r|^2\,\d r)^{\frac{p}{2}}]
\le \|Z\|^p_{\cS^p} (t-s)^{\frac{p}{2}}$ and
 the estimate of $\|(X,Y,Z)\|_{\cS^p}$
 leads to 
the desired H\"{o}lder continuity of the processes $X$ and $Y$. 
\end{proof}

The following theorem establishes
the $1/2$-H\"{o}lder 
regularity of optimal controls to \eqref{eq:mfcE}
based on the regularity results in Theorem 
\ref{TH:fbsde_Regularity}.

\begin{Theorem}\label{TH:ControlRegularity}
Suppose (H.\ref{assum:mfcE}) and (H.\ref{assum:mfcE_hat}(\ref{item:ex})) hold,
and let   $\xi_0\in L^2(\cF_0;\sR^n)$.
Then
(\ref{eq:mfcE}) admits a unique optimal control $\hat{\a}=(\hat{\a}_t)_{t\in [0,T]}\in \cA$,
which satisfies for all 
$p\ge 2$ that 
$\|\hat{\a}\|_{\cS^p}\le C(1+\|\xi_0\|_{L^p})$.
If we further assume that
(H.\ref{assum:mfcE_hat}(\ref{item:HC})) holds, then the optimal control $\hat{\a}$
satisfies for all 
 $p\ge 2$,
 $0\le s\le t\le T$
that
$ \sE\left[\sup_{s\le r\le t}|\hat{\a}_r-\hat{\a}_s|^p\right]^{1/p}
 \le
C
(1+\|\xi_0\|_{L^{p}})
|t-s|^{{1}/{2}},
$ 
with a constant $C$
depending only on  $p$ and the data in (H.\ref{assum:mfcE}) and (H.\ref{assum:mfcE_hat}).
\end{Theorem}
\begin{proof}
Let 
$\hat{\a}:[0,T]\t \sR^n \t \sR^n \t \cP_2(\sR^{n} \t \sR^{n}) \to \bA$ be
the function in (H.\ref{assum:mfcE_hat}(\ref{item:ex})).
We define
 for each $t\in [0,T]$  that 
$\hat{\a}_t=\hat{\a}(t,X_t,Y_t, \sP_{(X_t,Y_t)})$,
and write $\hat{\a}=(\hat{\a}_t)_{t\in [0,T]}$
with  a slight abuse of notation.
The local boundedness 
and Lipschitz continuity 
of the function $\hat{\a}$ (see (H.\ref{assum:mfcE_hat}(\ref{item:ex}))) show that 
$\|\hat{\a}\|_{\cS^p}\le \|\hat{\a}(\cdot,0,0,\bm{\delta}_{0_{n + n}})\|_{{\cS^p}}+C(\|X\|_{\cS^p}+\|Y\|_{\cS^p})\le C(1+\|\xi_0\|_{L^p})$ for all $p\ge 2$. 
Then, the assumption that  $\xi_0\in L^2(\cF_0;\sR^n)$ and the definition of 
the function
$\hat{\a}$ in (H.\ref{assum:mfcE_hat}(\ref{item:ex}))
imply that the control $\hat{\a}$ is admissible (i.e.,~$\hat{\a}\in \cA$)
and
satisfies 
\eqref{eq:opti_re} (equivalently  \eqref{eq:opti}),
which shows that 
$\hat{\a}$ is an optimal control of \eqref{eq:mfcE}.
The uniqueness of optimal controls of \eqref{eq:mfcE} follows from the strong convexity of 
the cost functional $J:\cA\to \sR$, which will be shown in Lemma \ref{lemma:J}.

Finally, for any given $0\le s\le r\le t\le T$, we obtain from (H.\ref{assum:mfcE_hat}) that
\begin{align*}
&|\hat{\a}_r-\hat{\a}_s|=|\hat{\a}(r,X_r,Y_r,\sP_{(X_r,Y_r)})-\hat{\a}(s,X_s,Y_s,\sP_{(X_s,Y_s)})|
\\
&\le 
C\big\{(1+
|X_r|+|Y_r|+\|\sP_{(X_r,Y_r)}\|_2)|r-s|^{\frac{1}{2}}
+
|X_r-X_s|
\\
&\q +|Y_r-Y_s|+\cW_2(\sP_{(X_r,Y_r)},\sP_{(X_s,Y_s)})
\big\},
\end{align*}
which together with the moment estimates and regularity of the processes $X,Y$ leads to 
\begin{align*}
&\sE\left[\sup_{s\le r\le t}|\hat{\a}_r-\hat{\a}_s|^p\right]^{\frac{1}{p}}
\\
&\le 
C\bigg((1+
\|X\|_{\cS^p}+\|Y\|_{\cS^p})|t-s|^{\frac{1}{2}}
+
\sE\left[\sup_{s\le r\le t}|X_r-X_s|^p\right]^{\frac{1}{p}}
 +\sE\left[\sup_{s\le r\le t}|Y_r-Y_s|^p\right]^{\frac{1}{p}}
\bigg)
\\
&\le 
C(1+\|\xi_0\|_{L^p})|t-s|^{\frac{1}{2}}.
\end{align*}
This completes the proof of Theorem \ref{TH:ControlRegularity}.
\end{proof}

\begin{Remark}\l{rmk:Z_regular}
Note that the 
$1/2$-H\"{o}lder regularity of the optimal open-loop control $\a$ 
in the $\cS^p$-norm
and 
the  dependence on the integrability of the initial condition $\xi_0$
in the estimate are 
optimal, 
since it agrees with the path regularity of Brownian motions.

%
%
%
\end{Remark}

\section{Error estimates of value functions for piecewise constant policy approximations}\l{sec:conv_PCPT}

In this section, 
based on the  regularity results of  optimal controls in Theorem \ref{TH:ControlRegularity},
we  establish   the convergence rate of 
the discrete-time control problem \eqref{eq:mfcE_constant}
in approximating the value function of \eqref{eq:mfcE}.

We start with the 
error introduced by approximating 
the set $\cA$ of admissible controls in \eqref{eq:mfcE} by piecewise constant controls.
More precisely,
let $\pi=\{0=t_0<\cdots<t_N=T\}$ be a 
 partition of $[0,T]$
with  stepsize $|\pi|=\max_{i=0,\ldots, N-1}(t_{i+1}-t_i)$
and let $\cA_{\pi}$ be the   subset of 
piecewise constant 
controls defined as in \eqref{eq:A_pi}. 
For any given initial state $\xi_0\in L^2(\cF_0;\sR^n)$, we consider 
the following minimization problem
\bb\l{eq:mfcE_constant_c}
V_\pi^c(\xi_0)\coloneqq\inf_{\a\in \cA_\pi} J(\a;\xi_0),
\ee
where for each $\a\in \cA_\pi$, 
$J(\a;\xi_0)$ is the cost functional defined as in \eqref{eq:mfcE} with the controlled state process 
$X^\a$ satisfying the MV-SDE \eqref{eq:mfcE_fwd}.

The following theorem 
shows that as the stepsize $|\pi|$ tends to zero,
the value function $V^c_\pi(\xi_0)$ 
converges
 from  above
to the value function $V(\xi_0)$
in \eqref{eq:mfcE}
 with   half-order accuracy. 

\begin{Theorem}\l{thm:PCPT_c_value}
Suppose (H.\ref{assum:mfcE}) and (H.\ref{assum:mfcE_hat}) hold,
let 
the function $V:L^2(\cF_0;\sR^n)\to \sR$ be defined as in \eqref{eq:mfcE},
and for each 
 partition $\pi$ of $[0,T]$
let the function $V_\pi^c:L^2(\cF_0;\sR^n)\to \sR$ be defined as in \eqref{eq:mfcE_constant_c}.
Then there exists a constant $C>0$, such that it holds for all 
$\xi_0\in L^2(\cF_0;\sR^n)$ and for every partition
$\pi$  of $[0,T]$ with  stepsize $|\pi|$ that
\begin{align*}
V(\xi_0)&\le V_\pi^c(\xi_0)\le V(\xi_0)+C(1+\|\xi_0\|^2_{L^{2}})|\pi|^{1/2}.
\end{align*}
\end{Theorem}

\begin{proof}
Throughout this proof,
let $\xi_0\in L^2(\cF_0;\sR^n)$ be a given initial state,
let  $\pi=\{0=t_0<\cdots<t_N=T\}$ be a given
 partition of $[0,T]$
with  stepsize $|\pi|=\max_{i=0,\ldots, N-1}(t_{i+1}-t_i)$,
let $\cA_{\pi}\subset \cA$ be the associated piecewise constant controls
and  $C$ be a generic constant, 
which is independent of the initial state $\xi_0$ and the partition $\pi$, and may take a different value at each occurrence.

It is clear from $\cA_{\pi}\subset \cA$  and
the definitions of $V$ and $V_\pi^c$  that $V_\pi^c(\xi_0)=\inf_{\a\in \cA_\pi} J(\a;\xi_0)\ge \inf_{\a\in \cA} J(\a;\xi_0)=V(\xi_0)$.
We now establish an upper bound of $V_\pi^c(\xi_0)-V(\xi_0)$. Note that
Theorem \ref{TH:ControlRegularity} shows that under 
(H.\ref{assum:mfcE}) and (H.\ref{assum:mfcE_hat}), 
 there exists an admissible control  $\hat{\a}\in \cA$
such that $V(\xi_0)=J(\hat{\a};\xi)$ and 
it holds for all
$0\le s\le t\le T$
that
$ \sE\left[\sup_{s\le r\le t}|\hat{\a}_r-\hat{\a}_s|^2\right]^{1/2}
 \le
C
(1+\|\xi_0\|_{L^{2}})
|t-s|^{{1}/{2}}$.
Let $\hat{\a}^{\pi}$  be a piecewise constant approximation 
of the process $\hat{\a}$ on $\pi$
satisfying for all $t\in [0,T)$ that
$\hat{\a}^\pi_{t}=\sum_{i=0}^{N-1}\hat{\a}_{t_i} \bm{1}_{[t_i,t_{i+1})}(t)$.
Then it is clear that $\hat{\a}^\pi\in \cA_\pi$ and it holds for all $t\in [0,T)$ that 
$t\in [t_i,t_{i+1})$ for some $i\in \{0,\ldots, N-1\}$ and 
$$
\|\hat{\a}^\pi_{t}-\hat{\a}_{t}\|_{L^2}=\|\hat{\a}_{t_i}-\hat{\a}_{t}\|_{L^2}
\le C(1+\|\xi_0\|_{L^{2}})|\pi|^{{1}/{2}}.
$$
Hence we can obtain from standard stability estimates of MV-SDEs  that 
$\|X^{\hat{\a}}-X^{\hat{\a}^\pi}\|^2_{\cS^2}\le C\|\hat{\a}^\pi-\hat{\a}\|^2_{\cH^2}
\le  C(1+\|\xi_0\|^2_{L^{2}})|\pi|$,
which together with Remark \ref{rmk:mfcE_regularity} and $V(\xi_0)=J(\hat{\a};\xi)$,
gives us the estimate that
\begin{align*}
&V_\pi^c(\xi_0)-V(\xi_0)
\le J(\hat{\a}^\pi;\xi_0)-J(\hat{\a};\xi_0)
\\
&\le 
\sE\bigg[
\int_0^T |f(t,X^{\hat{\a}^\pi}_t,{\hat{\a}^\pi}_t,\sP_{(X^{\hat{\a}^\pi}_t,{\hat{\a}^\pi}_t)})-f(t,X^{\hat{\a}}_t,\hat{\a}_t,\sP_{(X^{\hat{\a}}_t,\hat{\a}_t)})|\, \d t
\\
&\q +
|g(X^{\hat{\a}^\pi}_T,\sP_{X^{\hat{\a}^\pi}_T})-g(X^{\hat{\a}}_T,\sP_{X^{\hat{\a}}_T})|
\bigg]
\\
&\le
C\bigg\{\sE
\bigg[
\int_0^T
\big(1+|X^{\hat{\a}}_t|+\|\sP_{(X^{\hat{\a}}_t,{\hat{\a}}_t)}\|_2+|\hat{\a}_t|+|X^{\hat{\a}^\pi}_t|+\|\sP_{(X^{\hat{\a}^\pi}_t,{\hat{\a}^\pi}_t)}\|_2+|{\hat{\a}^\pi}_t|\big)
\\
&\q 
\times
\big(|X^{\hat{\a}}_t-X^{\hat{\a}^\pi}_t|+|\hat{\a}_t-{\hat{\a}^\pi}_t|+\cW_2(\sP_{(X^{\hat{\a}}_t,{\hat{\a}}_t)},\sP_{(X^{\hat{\a}^\pi}_t,{\hat{\a}^\pi}_t)})
\big)
\, \d t
\\
&\q 
+\big(1+|X^{\hat{\a}}_T|+\|\sP_{X^{\hat{\a}}_T}\|_2+|X^{\hat{\a}^\pi}_T|+\|\sP_{X^{\hat{\a}^\pi}_T}\|_2
\big)
\big(|X^{\hat{\a}}_T-X^{\hat{\a}^\pi}_T|+\cW_2(\sP_{X^{\hat{\a}}_T},\sP_{X^{\hat{\a}^\pi}_T})
\big)
\bigg]\bigg\}.
\end{align*}
Then, we can deduce from the above estimate  and the Cauchy-Schwarz inequality that
\begin{align*}
\begin{split}
V_\pi^c(\xi_0)-V(\xi_0)
&\le
C\big\{
\big(1+\|(X^{\hat{\a}},\hat{\a},X^{\hat{\a}^\pi},\hat{\a}^\pi)\|_{\cH^2}\big)
\big(\|X^{\hat{\a}}-X^{\hat{\a}^\pi}\|_{\cH^2}+\|\hat{\a}-\hat{\a}^\pi\|_{\cH^2}
\big)
\\
&\q 
+\big(1+\|X^{\hat{\a}}_T\|_{L^2}+\|X^{\hat{\a}^\pi}_T\|_{L^2}
\big)
\|X^{\hat{\a}}_T-X^{\hat{\a}^\pi}_T\|_{L^2}
\big\}
\\
&\le 
C
(1+\|\xi_0\|^2_{L^{2}})|\pi|^{1/2},
\end{split}
\end{align*}
which completes the desired error estimate.
\end{proof}

In practice, instead of solving  \eqref{eq:mfcE_fwd} with a piecewise constant control,
one can further discretize the controlled dynamics in time
by the  Euler-Maruyama  scheme  (cf.~\eqref{eq:mfcE_constant}),
which allows us to 
 only deal with Gaussian random variables
with known mean and variance.
%
To quantify the  time discretization error of the controlled dynamics and the running cost, 
we assume the following time regularity of the coefficients:
\begin{Assumption}\l{assum:mfc_Holder_t}
Assume  the notation of (H.\ref{assum:mfcE}). 
The functions $b_0,b_1,b_2,b_3,\sigma_0,\sigma_1,\sigma_2$ 
in (H.\ref{assum:mfcE}(\ref{item:mfcE_lin}))
are $1/2$-H\"{o}lder continuous,
and 
there exists a constant $\hat{K}\in [0,\infty)$ satisfying
 for all $t,t'\in [0,T]$, $(x,a,\eta)\in \sR^n\t  \bA\t \cP_2(\sR^n\t \sR^k)$
that 
$|f(t,x,a,\eta)- f(t',x,a,\eta)|\le \hat{K}(1+|x|^2+|a|^2+\|\eta\|^2_2)|t-t'|^{1/2}$.

\end{Assumption}
\begin{Remark}\l{rmk:local_holder_t}
(H.\ref{assum:mfcE}(\ref{item:mfcE_lin})) and  (H.\ref{assum:mfc_Holder_t}) imply 
for all $(x,a,\mu,\eta), (x',a',\mu',\eta')\in \sR^k\t \bA\t \cP_2(\sR^n)\t\cP_2(\sR^n\t \sR^k)$,
\begin{align*}
&|b(r,x,a,\eta)-b(s,x',a',\eta')|
\\
&\le C
\Big((1+|x|+|a|+\|\eta\|_2)|r-s|^{1/2}+|x-x'|+|a-a'|+\cW_2(\eta,\eta')\Big),
\\
&|\sigma(r,x,\mu)-\sigma(s,x',\mu')|
\le C\Big((1+|x|+\|\mu\|_2)|r-s|^{1/2}+|x-x'|+\cW_2(\mu,\mu')\Big).
\end{align*}
\end{Remark}

Under the  H\"{o}lder regularity of the coefficients,
we shall prove that the value function 
 $V_\pi(\xi_0)$ 
converges to the value function $V(\xi_0)$
in \eqref{eq:mfcE}
 with  order $1/2$ as the stepsize $|\pi|$ tends to zero,
 which is  optimal
 for   MFC problems with such irregular running costs $f$.

Note that
a similar convergence rate has been established in 
\cite[Proposition 12]{carmona2019}
for the special case
 where both $b$ and $f$ are independent of the law of controls.
By restricting the analysis to closed-loop (also called Markovian) controls
(i.e., $\a\in \cA$ that are of the form $\a_t=\phi(t,X_t)$ with $\phi\in C^{1,2}_b([0,T]\t\sR^n)$),
and assuming
the decoupling field of \eqref{eq:mfc_fbsde2} and the function $\hat{\a}$ in (H.\ref{assum:mfcE_hat}(\ref{item:ex}))
to be  twice differentiable   with uniformly Lipschitz continuous derivatives in $(t,x,y,\mu)\in [0,T]\t \sR^n\t \sR^n\t \cP_2(\sR^n)$,
the authors establish an order $1/2$ convergence of $(V_\pi(\xi_0))_{\pi}$ in terms of $|\pi|$,
with a constant depending on the sup-norms of the second-order derivatives of the feedback map $\phi$ and coefficients.
These conditions typically require the  cost functions $f$ and $g$ in \eqref{eq:mfcE} 
to be three-times differentiable in $(x,a,\mu)$ with bounded and Lipschitz continuous derivatives.

Here we remove these strong regularity assumptions and establish an order $1/2$ convergence   with general 
open-loop  strategies and 
cost functions
that are merely H\"{o}lder continuous in time and 
  Lipschitz continuously differentiable in space;
see 
Example \ref{example:mfc}
 for precise regularity assumptions
  to ensure  (H.\ref{assum:mfcE_hat})
   in the setting of
 MFC problems.

\begin{Theorem}\l{thm:PCPT_discrete_value}
Suppose (H.\ref{assum:mfcE}), (H.\ref{assum:mfcE_hat})
and (H.\ref{assum:mfc_Holder_t}) hold,
let 
the function $V:L^2(\cF_0;\sR^n)\to \sR$ be defined as in \eqref{eq:mfcE},
and for each 
 partition $\pi$ of $[0,T]$
let the function $V_\pi:L^2(\cF_0;\sR^n)\to \sR$ be defined as in \eqref{eq:mfcE_constant}.
Then there exists a constant $C>0$, such that it holds for all 
$\xi_0\in L^2(\cF_0;\sR^n)$ and for every partition
$\pi$  of $[0,T]$ with  stepsize $|\pi|$ that
$
V_\pi(\xi_0)- V(\xi_0)\le C(1+\|\xi_0\|^2_{L^{2}})|\pi|^{1/2}.
$

If we further assume that $\bA$ is a compact subset of $\sR^k$,
then 
it holds for all 
$\xi_0\in L^2(\cF_0;\sR^n)$ and for every partition
$\pi$  of $[0,T]$ with  stepsize $|\pi|$ that
$
|V_\pi(\xi_0)- V(\xi_0)|\le C(1+\|\xi_0\|^2_{L^{2}})|\pi|^{1/2}.
$
\end{Theorem}

\begin{proof}
Throughout this proof,
let $\xi_0\in L^2(\cF_0;\sR^n)$ be a given initial state,
let  $\pi=\{0=t_0<\cdots<t_N=T\}$ be a given
 partition of $[0,T]$
with  stepsize $|\pi|=\max_{i=0,\ldots, N-1}(t_{i+1}-t_i)$,
let $\cA_{\pi}\subset \cA$ be the associated piecewise constant controls
and  
let $C$ be a generic constant, 
which is independent of the initial state $\xi_0$, the partition $\pi$
and controls $\a\in \cA$, and may take a different value at each occurrence.

\textbf{Step 1: Estimate an upper bound of $V_\pi(\xi_0)-V(\xi_0)$.}
As in the proof of Theorem \ref{thm:PCPT_c_value},
let  $\hat{\a}\in \cA$ be an optimal control of \eqref{eq:mfcE}
satisfying 
that
$ \|\hat{\a}\|_{\cS^2}^2
 \le
C
(1+\|\xi_0\|^2_{L^{2}})$
and 
$ \|\hat{\a}_t-\hat{\a}_s\|_{L^2}^2
 \le
C
(1+\|\xi_0\|^2_{L^{2}})
|t-s|$
for all $ s,t\in [0,T]$,
let
${X}^{\hat{\a}}$ be the solution to \eqref{eq:mfcE_fwd}  with the control ${\hat{\a}}$
satisfying
 $ \|X^{\hat{\a}}\|_{\cS^2}^2
 \le
C
(1+\|\xi_0\|^2_{L^{2}})$
and 
$
\|X^{\hat{\a}}_t-X^{\hat{\a}}_s\|_{L^2}^2
  \le
C
(1+\|\xi_0\|^2_{L^{2}})
|t-s|$
for all
$ s,t\in [0,T]$
(see Theorem \ref{TH:ControlRegularity}),
let  $\hat{\a}^{\pi}$  be a piecewise constant approximation 
of the process $\hat{\a}$ on $\pi$
satisfying for all $t\in [0,T)$ that
$\|\hat{\a}^\pi_{t}-\hat{\a}_{t}\|_{L^2}=\|\hat{\a}_{t_i}-\hat{\a}_{t}\|_{L^2}
\le C(1+\|\xi_0\|_{L^{2}})|\pi|^{{1}/{2}}$,
and 
let
$\hat{X}^\pi$ be the solution to \eqref{eq:mfcE_fwd_euler}  with the control $\hat{\a}^\pi$.
Note that
it is standard to show 
by using the Lipschitz continuity of  $(b,\sigma)$
and Gronwall's inequality that
$$ 
\max_{t_i\in \pi}\|\hat{X}^{\pi}_{t_i}\|_{L^2}^2
 \le
C
\big(1+\|\xi_0\|^2_{L^{2}}+\max_{t_i\in \pi}\|\hat{\a}^\pi_{t_i}\|^2_{L^2}\big)
\le 
C(1+\|\xi_0\|^2_{L^{2}}).
$$

Observe that 
it holds  for each $i\in \{0,\ldots,N-1\}$ that
\begin{align}\l{eq:X^hat_a-X^pi}
\begin{split}
&\sE[|X^{\hat{\a}}_{t_{i+1}}-\hat{X}^{\pi}_{t_{i+1}}|^2]
\\
&
\le C\bigg\{
\sE\bigg[\bigg|\sum_{j=0}^{ i}\int_{t_j}^{t_{j+1}}
\big(b(t,X^{\hat{\a}}_t,\hat{\a}_t,\sP_{(X^{\hat{\a}}_t,\hat{\a}_t)})-
b(t_j,\hat{X}^{\pi}_{t_j},\hat{\a}^\pi_{t_j},\sP_{(\hat{X}^{\pi}_{t_j},\hat{\a}^\pi_{t_j})})\big)
\,\d t\bigg|^2\bigg]
\\
&\q 
+\sE\bigg[\sum_{j=0}^{ i}\int_{t_j}^{t_{j+1}}
|\sigma(t,X^{\hat{\a}}_t,\sP_{X^{\hat{\a}}_t})-
\sigma(t_j,\hat{X}^{\pi}_{t_j},\sP_{\hat{X}^{\pi}_{t_j}})|^2
\,\d t\bigg]
\bigg\}
\\
&\le
 C\bigg\{
R_1+
\sE\bigg[\bigg(
\sum_{j=0}^{ i}\int_{t_j}^{t_{j+1}}
\big|b({t_j},X^{\hat{\a}}_{t_j},\hat{\a}_{t_j},\sP_{(X^{\hat{\a}}_{t_j},\hat{\a}_{t_j})})-
b(t_j,\hat{X}^{\pi}_{t_j},\hat{\a}^\pi_{t_j},\sP_{(\hat{X}^{\pi}_{t_j},\hat{\a}^\pi_{t_j})})\big|
\,\d t\bigg)^2\bigg]
\\
&\q 
+\sE\bigg[\sum_{j=0}^{ i}\int_{t_j}^{t_{j+1}}
|\sigma(t_j,X^{\hat{\a}}_{t_j},\sP_{X^{\hat{\a}}_{t_j}})-
\sigma(t_j,\hat{X}^{\pi}_{t_j},\sP_{\hat{X}^{\pi}_{t_j}})|^2
\,\d t\bigg]
\bigg\},
\end{split}
\end{align}
with the residual term $R_1$ defined as 
\begin{align}\l{eq:R1}
\begin{split}
R_1
&\coloneqq
 \sE\bigg[\bigg(\sum_{i=0}^{N-1}\int_{t_i}^{t_{i+1}}
\big|b(t,X^{\hat{\a}}_t,\hat{\a}_t,\sP_{(X^{\hat{\a}}_t,\hat{\a}_t)})-
b(t_i,X^{\hat{\a}}_{t_i},\hat{\a}_{t_i},\sP_{(X^{\hat{\a}}_{t_i},\hat{\a}_{t_i})})\big|
\,\d t\bigg)^2\bigg]
\\
&\q 
+\sE\bigg[\sum_{i=0}^{ N-1}\int_{t_i}^{t_{i+1}}
|\sigma(t,X^{\hat{\a}}_t,\sP_{X^{\hat{\a}}_t})-
\sigma(t_i,X^{\hat{\a}}_{t_i},\sP_{X^{\hat{\a}}_{t_i}})|^2
\,\d t\bigg]
\\
&\le
 T
 \sE\bigg[\sum_{i=0}^{N-1}\int_{t_i}^{t_{i+1}}
\big|b(t,X^{\hat{\a}}_t,\hat{\a}_t,\sP_{(X^{\hat{\a}}_t,\hat{\a}_t)})-
b(t_i,X^{\hat{\a}}_{t_i},\hat{\a}_{t_i},\sP_{(X^{\hat{\a}}_{t_i},\hat{\a}_{t_i})})\big|^2
\,\d t\bigg]
\\
&\q 
+\sE\bigg[\sum_{i=0}^{ N-1}\int_{t_i}^{t_{i+1}}
|\sigma(t,X^{\hat{\a}}_t,\sP_{X^{\hat{\a}}_t})-
\sigma(t_i,X^{\hat{\a}}_{t_i},\sP_{X^{\hat{\a}}_{t_i}})|^2
\,\d t\bigg],
\end{split}
\end{align}
where we have applied the Cauchy-Schwarz inequality for  the last inequality.
Hence, by applying the Lipschitz continuity of $b$ and $\sigma$ and Gronwall's inequality, we can deduce
the estimate that
\begin{align}\l{eq:X_hat-X_a,pi}
\begin{split}
\max_{t_i\in \pi}\sE[|X^{\hat{\a}}_{t_i}-\hat{X}^{\pi}_{t_i}|^2]
&\le
 C\Big(
R_1+\max_{t_i\in \pi}\|\hat{\a}_{t_i}-\hat{\a}^\pi_{t_i}\|_{L^2}^2
\Big),
\end{split}
\end{align}
which, together with  the definition of $V_\pi(\xi_0)$ and 
the optimality of 
$\hat{\a}$  for \eqref{eq:mfcE},
gives that
\begin{align*}
&V_\pi(\xi_0)-V(\xi_0)
\le J_\pi(\hat{\a}^\pi;\xi_0)-J(\hat{\a};\xi_0)
\\
&\le\sE\bigg[\sum_{i=0}^{N-1}
\int_{t_i}^{t_{i+1}}
\big|
 f(t_i,\hat{X}^{\pi}_{t_i},\hat{\a}^\pi_{t_i},\sP_{(\hat{X}^{\pi}_{t_i},\hat{\a}^\pi_{t_i})})
 -f(t,X^{\hat{\a}}_t,\hat{\a}_t,\sP_{(X^{\hat{\a}}_t,\hat{\a}_t)})
 \big|\,\d t
\\
&\q +
|g(\hat{X}^{\pi}_T,\sP_{\hat{X}^{\pi}_T})-g(X^{\hat{\a}}_T,\sP_{X^{\hat{\a}}_T})|
\bigg]
\\
&\le
C\bigg(R_2+\sE\bigg[\sum_{i=0}^{N-1}
\int_{t_i}^{t_{i+1}}
\big|
f(t_i,\hat{X}^{\pi}_{t_i},\hat{\a}^\pi_{t_i},\sP_{(\hat{X}^{\pi}_{t_i},\hat{\a}^\pi_{t_i})})
-f({t_i},X^{\hat{\a}}_{t_i},\hat{\a}_{t_i},\sP_{(X^{\hat{\a}}_{t_i},\hat{\a}_{t_i})})
 \big|\,\d t
\\
& \q +
  |g(\hat{X}^{\pi}_T,\sP_{\hat{X}^{\pi}_T})-g(X^{\hat{\a}}_T,\sP_{X^{\hat{\a}}_T})|
\bigg]
\bigg)
\end{align*}
with the residual term defined by 
\begin{align}\l{eq:R2}
\begin{split}
R_2
&\coloneqq
\sE\bigg[\sum_{i=0}^{N-1}
\int_{t_i}^{t_{i+1}}
\big|
 f(t,X^{\hat{\a}}_t,\hat{\a}_t,\sP_{(X^{\hat{\a}}_t,\hat{\a}_t)})
 -f({t_i},X^{\hat{\a}}_{t_i},\hat{\a}_{t_i},\sP_{(X^{\hat{\a}}_{t_i},\hat{\a}_{t_i})})
 \big|\,\d t
\bigg].
\end{split}
\end{align}
Then, by using  Remark  \ref{rmk:mfcE_regularity},
 the Cauchy-Schwarz inequality,
\eqref{eq:X_hat-X_a,pi}
and  the fact that $\|\hat{\a}^\pi_{t}-\hat{\a}_{t}\|_{L^2}\le C(1+\|\xi_0\|_{L^{2}})|\pi|^{{1}/{2}}$, we can deduce  that
\begin{align}\l{eq:V_pi-V}
\begin{split}
&V_\pi(\xi_0)-V(\xi_0)
\\
&\le
C\bigg(R_2+
\max_{t_i\in \pi}\big(1+\|(X^{\hat{\a}}_{t_i},\hat{\a}_{t_i},\hat{X}^{\pi}_{t_i},\hat{\a}^\pi_{t_i})\|_{L^2}\big)
\big(\|X^{\hat{\a}}_{t_i}-\hat{X}^{\pi}_{t_i}\|_{L^2}+\|\hat{\a}_{t_i}-\hat{\a}^\pi_{t_i}\|_{L^2}\big)
\bigg)
\\
&\le
C\Big(R_2+
(1+\|\xi_0\|_{L^{2}})
\big(
R^{1/2}_1+(1+\|\xi_0\|_{L^{2}})|\pi|^{{1}/{2}}
\big)
\Big).
\end{split}
\end{align}
Hence it remains to estimate the residual terms $R_1$ and $R_2$ defined as in \eqref{eq:R1} and \eqref{eq:R2}, respectively.
Note that 
Remark \ref{rmk:local_holder_t}
and the H\"{o}lder regularity of $(X^{\hat{\a}},\hat{\a})$
imply that
\begin{align*}
R_1
&\le
C
\bigg((1+\|X^{\hat{\a}}\|^2_{\cS^2}+\|\hat{\a}\|^2_{\cS^2})|\pi|+
\sup_{t_i\in \pi,r\in [t_i,t_{i+1})}
\Big(
\|X^{\hat{\a}}_r-X^{\hat{\a}}_{t_i}\|_{L^2}^2
+\|{\hat{\a}}_r-{\hat{\a}}_{t_i}\|_{L^2}^2
\Big)
\bigg)
\\
&\le
C(1+\|\xi_0\|^2_{L^2})|\pi|,
\end{align*}
while  Remark \ref{rmk:mfcE_regularity}
and 
(H.\ref{assum:mfc_Holder_t}) 
give us that
\begin{align*}
R_2
&\le
C\bigg[
(1+\|X^{\hat{\a}}\|^2_{\cS^2}+\|\hat{\a}\|^2_{\cS^2})|\pi|^{1/2}
\\
&\q 
+
(1+\|X^{\hat{\a}}\|_{\cS^2}+\|\hat{\a}\|_{\cS^2})
\sup_{t_i\in \pi,r\in [t_i,t_{i+1})}
\Big(
\|X^{\hat{\a}}_r-X^{\hat{\a}}_{t_i}\|_{L^2}
+\|{\hat{\a}}_r-{\hat{\a}}_{t_i}\|_{L^2}
\Big)
\bigg]
\\
&\le 
C(1+\|\xi_0\|^2_{L^2})|\pi|^{1/2}.
\end{align*}
These estimates enable us to  conclude from \eqref{eq:V_pi-V} 
the upper bound
that $V_\pi(\xi_0)-V(\xi_0)\le C(1+\|\xi_0\|^2_{L^2})|\pi|^{1/2}$.

\textbf{Step 2: Estimate an upper bound of $V(\xi_0)-V_\pi(\xi_0)$.}
Note that the additional compactness assumption of $\bA$ implies that there exists $C>0$ such that
$\|\a\|_{\cH^2}\le C$ for all $\a\in \cA_\pi$.
Then  standard moment estimates for MV-SDEs (see e.g.~\cite[Theorem 3.3]{reis2019})
shows that there exists $C>0$ such that for all $\a\in \cA_\pi$, the solution to 
\eqref{eq:mfcE_fwd} with the control $\a$ satisfies 
$\|X^\a\|_{\cS^2}\le C(1+\|\xi_0\|_{L^2})$.
Moreover, for any $0\le s\le r\le t\le T$, we can obtain 
from  the  Burkholder-Davis-Gundy  inequality,
H\"{o}lder's inequality
and (H.\ref{assum:mfcE}(\ref{item:mfcE_lin}))
 that
\begin{align}\l{eq:X^a_Holder_discrete}
\begin{split}
&\sE\left[\sup_{s\le r\le t}|X^\a_r-X^\a_s|^2\right]
\\
&\le 2
\sE\bigg[\bigg(
\int_s^t
|b(u,X^\a_u,\a_u,\sP_{(X^\a_u,\a_u)})|^2\, \d u\bigg)(t-s)
+
\int_s^t
|\sigma(u,X^\a_u,\sP_{X^\a_u})|^2\, \d u
\bigg]
\\
&\le C
(\|b_0\|^2_{L^2(0,T)}+\|\sigma_0\|^2_{L^\infty(0,T)}+\|X^\a\|^2_{\cS^2}+\|\a\|^2_{\cH^2})(t-s)
\\
&\le C(1+\|\xi_0\|^2_{L^2})(t-s).
\end{split}
\end{align}
Similarly,
for each $\a\in \cA_\pi$, 
by using the Lipschitz continuity of the coefficients $b,\sigma$
and Gronwall's inequality,
one can show the corresponding solution $X^{\a,\pi}$ to \eqref{eq:mfcE_fwd_euler}
(with control $\a$) satisfies the following moment estimate:
\begin{equation}\label{eq:apriori_1}
\max_{t_i\in \pi}\|{X}^{\a,\pi}_{t_i}\|_{L^2}^2
 \le
C
\big(1+\|\xi_0\|^2_{L^{2}}+\max_{t_i\in \pi}\|{\a}_{t_i}\|^2_{L^2}\big)
\le 
C(1+\|\xi_0\|^2_{L^{2}}).
\end{equation}
Let $\a\in \cA_\pi$ be fixed,
and let $X^\a$ and ${X}^{\a,\pi}$ be the solution to \eqref{eq:mfcE_fwd} and 
\eqref{eq:mfcE_fwd_euler} with the control $\a$, respectively. 
Then by following similar arguments as those for \eqref{eq:X^hat_a-X^pi} and \eqref{eq:R1},
we have
for each
$i\in \{0,\ldots,N-1\}$ that
\begin{align*}
&\sE[|X^{{\a}}_{t_{i+1}}-{X}^{\a,\pi}_{t_{i+1}}|^2]
\\
&\le
 C\bigg\{
R_1^\a+
\sE\bigg[\bigg(
\sum_{j=0}^{ i}\int_{t_j}^{t_{j+1}}
\big|b({t_j},X^{{\a}}_{t_j},{\a}_{t_j},\sP_{(X^{{\a}}_{t_j},{\a}_{t_j})})-
b(t_j,{X}^{\a,\pi}_{t_j},{\a}_{t_j},\sP_{({X}^{\a,\pi}_{t_j},{\a}_{t_j})})\big|
\,\d t\bigg)^2\bigg]
\\
&\q 
+\sE\bigg[\sum_{j=0}^{ i}\int_{t_j}^{t_{j+1}}
|\sigma(t_j,X^{{\a}}_{t_j},\sP_{X^{{\a}}_{t_j}})-
\sigma(t_j,{X}^{\a,\pi}_{t_j},\sP_{{X}^{\a,\pi}_{t_j}})|^2
\,\d t\bigg]
\bigg\},
\end{align*}
with the residual term $R^\a_1$ defined as 
\begin{align}\l{eq:R1_a}
\begin{split}
R^\a_1
&\coloneqq
 \sE\bigg[\sum_{i=0}^{N-1}\int_{t_i}^{t_{i+1}}
\big|b(t,X^{{\a}}_t,{\a}_t,\sP_{(X^{{\a}}_t,{\a}_t)})-
b(t_i,X^{{\a}}_{t_i},{\a}_{t_i},\sP_{(X^{{\a}}_{t_i},{\a}_{t_i})})\big|^2
\,\d t\bigg]
\\
&\q 
+\sE\bigg[\sum_{i=0}^{ N-1}\int_{t_i}^{t_{i+1}}
|\sigma(t,X^{{\a}}_t,\sP_{X^{{\a}}_t})-
\sigma(t_i,X^{{\a}}_{t_i},\sP_{X^{{\a}}_{t_i}})|^2
\,\d t\bigg],
\end{split}
\end{align}
which, along with  the Lipschitz continuity of $b$ and $\sigma$ and Gronwall's inequality, gives that
\begin{align}\l{eq:X_a-X_a,pi_discrete_a}
\begin{split}
\max_{t_i\in \pi}\sE[|X^{{\a}}_{t_i}-{X}^{\a,\pi}_{t_i}|^2]
&\le
 CR^\a_1.
\end{split}
\end{align}
Hence, we can obtain from Remark \ref{rmk:mfcE_regularity},
H\"{o}lder's inequality,
the \textit{a priori} estimate for $\|X^\a\|_{\cS^2}$, and the estimates (\ref{eq:apriori_1}) 
 and \eqref{eq:X_a-X_a,pi_discrete_a} that
 \begin{align}\l{eq:V-V_pi}
 \begin{split}
&
V(\xi_0)-V_\pi(\xi_0)\le \sup_{\a\in \cA_\pi}
\big|J({\a};\xi_0)-J_\pi({\a};\xi_0)\big|
\\
&\le
\sup_{\a\in \cA_\pi}
\sE\bigg[\sum_{i=0}^{N-1}
\int_{t_i}^{t_{i+1}}
\big|
f(t,X^{{\a}}_t,{\a}_t,\sP_{(X^{{\a}}_t,{\a}_t)})
- f(t_i,{X}^{\a,\pi}_{t_i},{\a}_{t_i},\sP_{({X}^{\a,\pi}_{t_i},{\a}_{t_i})})
 \big|\,\d t
\\
&\q +
|g(X^{{\a}}_T,\sP_{X^{{\a}}_T})-g({X}^{\a,\pi}_T,\sP_{{X}^{\a,\pi}_T})|
\bigg]
\\
&\le
C\sup_{\a\in \cA_\pi}\bigg(R^\a_2+\sE\bigg[\sum_{i=0}^{N-1}
\int_{t_i}^{t_{i+1}}
\big|
f({t_i},X^{{\a}}_{t_i},{\a}_{t_i},\sP_{(X^{{\a}}_{t_i},{\a}_{t_i})})
-f(t_i,{X}^{\a,\pi}_{t_i},{\a}_{t_i},\sP_{({X}^{\a,\pi}_{t_i},{\a}_{t_i})})
 \big|\,\d t
\\
& \q +
  |g(X^{{\a}}_T,\sP_{X^{{\a}}_T})-g({X}^{\a,\pi}_T,\sP_{{X}^{\a,\pi}_T})|
\bigg]
\bigg)
\\
&\le
C\sup_{\a\in \cA_\pi}\Big(R^\a_2+ (1+\|\xi_0\|_{L^2})
\max_{t_i\in \pi}\|X^{{\a}}_{t_i}-{X}^{\a,\pi}_{t_i}\|_{L^2}\Big)
\\
&
\le
C\sup_{\a\in \cA_\pi}\Big(R^\a_2+ (1+\|\xi_0\|_{L^2})(R^\a_1)^{1/2}\Big)
 \end{split}
\end{align}
with the residual term $R^\a_1$ defined as in \eqref{eq:R1_a} and 
 the residual term   $R^\a_2$ defined by:
\begin{align}\l{eq:R2_a}
\begin{split}
R_2^\a
&\coloneqq
\sE\bigg[\sum_{i=0}^{N-1}
\int_{t_i}^{t_{i+1}}
\big|
 f(t,X^{{\a}}_t,{\a}_t,\sP_{(X^{{\a}}_t,{\a}_t)})
 -f({t_i},X^{{\a}}_{t_i},{\a}_{t_i},\sP_{(X^{{\a}}_{t_i},{\a}_{t_i})})
 \big|\,\d t
\bigg].
\end{split}
\end{align}
Note that for each $\a\in \cA_\pi$, we have ${\a}_t={\a}_{t_i}$ for all $t\in [t_i,t_{i+1})$, $i\in \{0,\ldots,N-1\}$,
which 
together with 
(H.\ref{assum:mfc_Holder_t}),
Remarks \ref{rmk:mfcE_regularity} and \ref{rmk:local_holder_t},
and the estimate \eqref{eq:X^a_Holder_discrete}
implies that
\begin{align*}
R^\a_1
&\le
C
\Big((1+\|X^{{\a}}\|^2_{\cS^2}+\|{\a}\|^2_{\cH^2})|\pi|+
\sup_{t_i\in \pi,r\in [t_i,t_{i+1})}
\|X^{{\a}}_r-X^{{\a}}_{t_i}\|_{L^2}^2
\Big)
\\
&\le
C(1+\|\xi_0\|^2_{L^2})|\pi|,
\\
R^\a_2
&\le
C\Big(
(1+\|X^{{\a}}\|^2_{\cS^2}+\|{\a}\|^2_{\cH^2})|\pi|^{1/2}
\\
&\q 
+
(1+\|X^{{\a}}\|_{\cS^2}+\|{\a}\|_{\cH^2})
\sup_{t_i\in \pi,r\in [t_i,t_{i+1})}
\|X^{{\a}}_r-X^{{\a}}_{t_i}\|_{L^2}
\Big)
\\
&\le 
C(1+\|\xi_0\|^2_{L^2})|\pi|^{1/2}.
\end{align*}
These estimates lead to the desired upper bound
 $V(\xi_0)-V_\pi(\xi_0)\le C(1+\|\xi_0\|^2_{L^2})|\pi|^{1/2}$.
\end{proof}
\begin{Remark}\l{rmk:non_compact}
The H\"{o}lder regularity of the optimal control  of \eqref{eq:mfcE}
is essential 
for quantifying the time discretization error and obtaining an  upper bound of $V_\pi(\xi_0)-V(\xi_0)$.
For the lower bound of $V_\pi(\xi_0)-V(\xi_0)$,
we use  
the compactness of $\bA$  to establish 
a uniform estimate
for the $\cH^2$-norms of all controls $\a\in \cA_\pi$ with any partition $\pi$, 
which subsequently leads to a uniform H\"{o}lder regularity of
the solution $X^{\a}$ to \eqref{eq:mfcE_fwd} with  control $\a\in \cA_\pi$
and then the desired half-order convergence;
see  \cite[Proposition 3.1]{picarelli2020}
for a similar result with controlled It\^{o} diffusions.

A similar error bound can be established if 
one can obtain a uniform estimate for
the $\cH^2$-norms of  minimizers of 
$V_\pi(\xi_0)$ defined  in \eqref{eq:mfcE_constant}.
For example, for a given initial state $\xi_0\in L^2(\cF_0;\sR^n)$ and  constant $B\in [0,\infty)$, one may consider 
the   MFC problem $V^B(\xi_0)=\inf_{\a\in \cA_B}J(\a;\xi_0)$, with 
the cost functional $J(\a,\xi_0)$ defined as in \eqref{eq:mfcE}
and
a constrained control set 
$\cA_{B}\subset \cA$ consisting of  all admissible controls $\a\in \cA$ satisfying the 
estimate $\sE[\int_0^T|\a_t|^2\, \d t]\le B$
(see e.g.~\cite{lauriere2020}).
It is clear that for a sufficiently large $B$ (depending on the initial condition),
$V^B(\xi_0)=V(\xi_0)$ and 
the minimizer of
\eqref{eq:mfcE} is also a minimizer of $V^B(\xi_0)$. Hence, 
by following the same arguments as in Theorem \ref{thm:PCPT_discrete_value},
we see the value functions $(V^B_\pi(\xi_0))_{\pi}$  with 
the corresponding piecewise constant policies $\cA_{B,\pi}\subset \cA_B$
also
admit a half-order convergence rate to the value function $V^B(\xi_0)$,
with a constant depending on the initial condition $\xi_0$.
\end{Remark}

\section{Error estimates of optimal controls for piecewise constant policy approximations}\l{sec:PCPT_control}

In this section, we proceed to investigate the convergence of minimizers of the approximate control problems 
\eqref{eq:mfcE_constant_c} and \eqref{eq:mfcE_constant}
based on the convergence of their value functions.

Before presenting our convergence analysis, let us point out that the proofs of Theorems 
\ref{thm:PCPT_c_value} and \ref{thm:PCPT_discrete_value} show that 
for every initial state $\xi_0\in L^2(\cF_0;\sR^n)$ and
for every partition $(\pi_i)_{i\in \sN}$ of $[0,T]$ satisfying 
$\lim_{i\to \infty}|\pi_i|=0$,
we can find controls $(\hat{\a}^{\pi_i})_{i\in \sN}$
satisfying 
for all $i\in \sN$ that $\hat{\a}^{\pi_i}\in \cA_{\pi_i}$,
\begin{align*}
V_{\pi_i}^c(\xi_0)&\le J(\hat{\a}^{\pi_i};\xi_0)\le V_{\pi_i}^c(\xi_0)+C(1+\|\xi_0\|^2_{L^2})|\pi_i|^{1/2},
\\
V_{\pi_i}(\xi_0)&\le J_\pi(\hat{\a}^{\pi_i};\xi_0)\le V_{\pi_i}(\xi_0)+C(1+\|\xi_0\|^2_{L^2})|\pi_i|^{1/2}
\end{align*}
with a constant $C$ independent of $\xi_0$ and $\pi$,
and $\lim_{i\to \infty}\|\hat{\a}^{\pi_i}-\hat{\a}\|_{\cH^2(\sR^k)}=0$, where
 $\hat{\a}\in \cA$ the optimal control of \eqref{eq:mfcE}.
In fact, such   controls can be constructed based on  piecewise constant approximations of the optimal control strategy $\hat{\a}\in \cA$  on $\pi_i$.
Since in practice one may not be able to 
exactly compute these control strategies $(\hat{\a}^{\pi_i})_{i\in\sN}$,
 in this section we shall study the convergence of 
\textit{any $\eps$-optimal controls} of \eqref{eq:mfcE_constant_c} and \eqref{eq:mfcE_constant}.
In particular, we shall establish that any  $\eps$-optimal controls of these approximate control problems 
converge {strongly} to the optimal control of \eqref{eq:mfcE} in $\cH^2(\sR^k)$.

We start by showing several important properties of the cost functional $J(\cdot;\xi_0):\cA\to \sR$
defined as in \eqref{eq:mfcE}.

\begin{Lemma}\l{lemma:J}
Suppose (H.\ref{assum:mfcE})
and (H.\ref{assum:mfcE_hat})
 hold, let $\xi_0\in L^2(\cF_0;\sR^n)$ and 
 let $J(\cdot;\xi_0):\cA\to \sR$ be defined as in \eqref{eq:mfcE}.
 Then $J$ is continuous and strongly  convex.
 More specifically, it holds for all $\a,\b\in \cA$, $\tau\in [0,1]$ that
 \begin{align*}
&\tau J(\a;\xi_0)+(1-\tau) J(\b;\xi_0)-J(\tau \a+(1-\tau)\b;\xi_0)\ge 
\tau(1-\tau)(\lambda_1+\lambda_2)
\|\a-\b\|^2_{\cH^2},
\end{align*}
where $\lambda_1, \lambda_2$ are the constants appearing in (H.\ref{assum:mfcE}(\ref{item:mfcE_convex})).
Moreover, we have for all $\a\in \cA$ that 
\bb\l{eq:coercive}
J(\hat{\a};\xi_0)-J(\a;\xi_0)\le -(\lambda_1+\lambda_2)\|\hat{\a}-{\a}\|_{\cH^2}^2,
\ee
where $\hat{\a}$ is the unique minimizer of \eqref{eq:mfcE} defined in  Theorem \ref{TH:ControlRegularity}.
\end{Lemma}
\begin{proof}
The continuity of $J$ follows directly from  stability results of \eqref{eq:mfcE} and the local Lipschitz continuity of functions $(f,g)$ (see (H.\ref{assum:mfcE}(\ref{item:mfcE_lipschitz}))).

We now show the strong convexity of the cost functional $J$.  
Let $\a,\b\in \cA$, $\tau\in [0,1]$, 
and let $X^\a$ (resp.~$X^\b$) be the solution to \eqref{eq:mfcE} with control $\a$ (resp.~$\b$).
Let $\gamma=\tau\a+(1-\tau)\b$
and
let $X\coloneqq\tau X^\a+(1-\tau) X^\b$.
We first show $X=X^\gamma$,
where $X^\gamma$  be the solution to \eqref{eq:mfcE} with control $\gamma$.
 It is clear that $X_0=\tau X^\a_0+(1-\tau) X^\b_0=\xi_0=X^\gamma_0$.
For  each $t\in[0,T]$,
we see 
 that 
 $\sE[(X_t,\gamma_t)]=\tau \sE[(X^\a_t,\a_t)]+(1-\tau)\sE[(X^\b_t,\b_t)]$,
 which together with the linearity of the functions $b,\sigma$  in $(x,a,\eta)$ (see (H.\ref{assum:mfcE}(\ref{item:mfcE_lin})))
gives that
\begin{align*}
 b(t,X_t,\gamma_t,\sP_{(X_t,\gamma_t)})
& =
 \tau b(t,X^\a_t,\a_t,\sP_{(X^\a_t,\a_t)})
+(1-\tau)b(t,X^\b_t,\b_t,\sP_{(X^\b_t,\b_t)}),
\\
 \sigma(t,X_t,\sP_{X_t})&=
 \tau \sigma(t,X^\a_t,\sP_{X^\a_t})
+(1-\tau)\sigma(t,X^\b_t,\sP_{X^\b_t}).
\end{align*}
Hence, we can show by using   It\^{o}'s formula that
$X$ satisfies the same MV-SDE as $X^\gamma$, 
which along with the uniqueness of strong solutions
shows that
$X^\gamma=X=\tau X^\a+(1-\tau) X^\b$. 

Let $\tilde{X}^\a_T$ and  $\tilde{X}^\b_T$ be independent copies of 
${X}^\a_T$ and  ${X}^\b_T$, respectively, defined on $L^2(\tilde{\Om},\tilde{\cF},\tilde{\sP};\sR^n)$.
We see that $\tilde{X}^\gamma_T\coloneqq \tau \tilde{X}^\a_T+(1-\tau) \tilde{X}^\b_T$ is an independent copy of $X^\gamma_T$ with distribution $\sP_{X^\gamma_T}$.
Hence
we can obtain from the convexity of $g$ in  (H.\ref{assum:mfcE}(\ref{item:mfcE_convex}))
and $X^\gamma=\tau X^\a+(1-\tau) X^\b$
 that
\begin{align*}
&g(X^\a_T,\sP_{X^\a_T})-g(X^\gamma_T,\sP_{X^\gamma_T})
\\
&\ge \la \p_{x}g(X^\gamma_T,\sP_{X^\gamma_T}), X^\a_T-X^\gamma_T \ra
+\tilde{\sE}[\la\p_\mu g(X^\gamma_T,\mu)(\tilde{X}^\gamma_T),\tilde{X}^\a_T-\tilde{X}^\gamma_T\ra ]
\\
&=(1-\tau)
\Big(
\la \p_{x}g(X^\gamma_T,\sP_{X^\gamma_T}), X^\a_T-X^\b_T \ra
+\tilde{\sE}[\la\p_\mu g(X^\gamma_T,\mu)(\tilde{X}^\gamma_T),\tilde{X}^\a_T-\tilde{X}^\b_T\ra ]
\Big).
\end{align*}
Similarly, we can show that 
\begin{align*}
&g(X^\b_T,\sP_{X^\b_T})-g(X^\gamma_T,\sP_{X^\gamma_T})
\\
&\ge 
\tau
\Big(
\la \p_{x}g(X^\gamma_T,\sP_{X^\gamma_T}), X^\b_T-X^\a_T \ra
+\tilde{\sE}[\la\p_\mu g(X^\gamma_T,\mu)(\tilde{X}^\gamma_T),\tilde{X}^\b_T-\tilde{X}^\a_T\ra ]
\Big),
\end{align*}
which implies that 
$$
\tau \sE[g(X^\a_T,\sP_{X^\a_T})]+(1-\tau)\sE[g(X^\b_T,\sP_{X^\b_T})]
\ge
\sE[g(X^\gamma_T,\sP_{X^\gamma_T})].
$$
Now 
for each $t\in [0,T]$, let 
$(\tilde{X}^\a_t,\tilde{\a}_t)$ and  $(\tilde{X}^\b_t,\tilde{\b}_t)$ be independent copies of 
$({X}^\a_t,\a_t)$ and  $({X}^\b_t,\b_t)$ defined on $L^2(\tilde{\Om},\tilde{\cF},\tilde{\sP};\sR^n\t \sR^k)$, respectively.
We see that $(\tilde{X}^\gamma_t,\tilde{\gamma}_t)\coloneqq \tau (\tilde{X}^\a_t,\tilde{\a}_t)+(1-\tau) (\tilde{X}^\b_t,\tilde{\b}_t)$ is an independent copy of $(X^\gamma_t,\gamma_t)$ with distribution $\sP_{(X^\gamma_t,\gamma_t)}$.
Then we can obtain from the convexity of $f$ in   (H.\ref{assum:mfcE}(\ref{item:mfcE_convex}))
and $(X^\gamma,\gamma)=\tau( X^\a,\a)+(1-\tau) (X^\b,\b)$ that 
\begin{align*}
&f(t,X^\a_t,\a_t,\sP_{(X^\a_t,\a_t)})-f(t,X^\gamma_t,\gamma_t,\sP_{(X^\gamma_t,\gamma_t)})
\\
&\ge 
(1-\tau)\Big(
\la \p_{(x,a)}f(t,X^\gamma_t,\gamma_t,\sP_{(X^\gamma_t,\gamma_t)}), (X^\a_t-X^\b_t,\a_t-\b_t)\ra 
\\
&\q 
+\tilde{\sE}[\la\p_\mu f(t,X^\gamma_t,\gamma_t,\sP_{(X^\gamma_t,\gamma_t)})(\tilde{X}^\gamma_t,\tilde{\gamma}_t),\tilde{X}^\a_t-\tilde{X}^\b_t\ra] 
\\
&\q
+ \tilde{\sE}[\la\p_\nu f(t,X^\gamma_t,\gamma_t,\sP_{(X^\gamma_t,\gamma_t)}),\tilde{\alpha}_t-\tilde{\b}_t\ra ]
\Big)
+(1-\tau)^2
\Big(\lambda_1|\a_t-\b_t|^2+\lambda_2\tilde{\sE}[|\tilde{\alpha}_t-\tilde{\b}_t|^2]\Big),
\end{align*}
Similarly, we can derive a lower bound of 
$f(t,X^\b_t,\b_t,\sP_{(X^\b_t,\b_t)})-f(t,X^\gamma_t,\gamma_t,\sP_{(X^\gamma_t,\gamma_t)})$,
which subsequently leads to the estimate that
\begin{align*}
&\tau\sE[f(t,X^\a_t,\a_t,\sP_{(X^\a_t,\a_t)})]+(1-\tau)\sE[f(t,X^\b_t,\b_t,\sP_{(X^\b_t,\b_t)})]
-\sE[f(t,X^\gamma_t,\gamma_t,\sP_{(X^\gamma_t,\gamma_t)})]
\\
&\ge 
\Big(\tau(1-\tau)^2+\tau^2(1-\tau)\Big)
\Big(\lambda_1\sE[|\a_t-\b_t|^2]+\lambda_2\tilde{\sE}[|\tilde{\alpha}_t-\tilde{\b}_t|^2]
\Big)
\\
&=\tau(1-\tau)(\lambda_1+\lambda_2)
\sE[|\a_t-\b_t|^2].
\end{align*}
Hence, we can conclude from   \eqref{eq:mfcE} the desired strong convexity estimate. 

We proceed to show the estimate \eqref{eq:coercive}.
The linearity of $(b,\sigma)$
and  the convexity of $f$  
in (H.\ref{assum:mfcE}) imply 
that the Hamiltonian $H$ defined as in \eqref{eq:mfcE_hamiltonian} is convex, i.e., 
for all $(t,y,z)\in [0,T]\in \sR^n\t \sR^{n\t d}$, $(x,\eta,a),(x',\eta',a')\in \sR^n\t \cP_2(\sR^n \t \sR^{k})\t \bA$,
\begin{align}\l{eq:mfcE_H_convex}
\begin{split}
&H(t,x,a,\eta,y,z)-H(t,x',a',\eta',y,z)-\la \p_{(x,\a)}H(t,x,a,\eta,y,z), (x-x',a-a')\ra 
\\
&\q 
-\tilde{\sE}[\la\p_\mu H(t,x,a,\eta,y,z)(\tilde{X},\tilde{\a}),\tilde{X}-\tilde{X}'\ra  + \la\p_\nu H(t,x,a,\eta,y,z)(\tilde{X},\tilde{\a}),\tilde{\a}-\tilde{\a}'\ra ] \\
& \le -\lambda_1 |a'-a|^2 - \lambda_2 \tilde{\mathbb{E}}[|\tilde{\a}'-\tilde{\a}|^2],
\end{split}
\end{align}
whenever  $(\tilde{X},\tilde{\a}),(\tilde{X}',\tilde{\a}')\in L^2(\tilde{\Om},\tilde{\cF},\tilde{\sP};\sR^n\t\sR^k)$
with distributions $\eta$ and $\eta'$, respectively.
Moreover, 
 the same arguments as in 
\cite[Theorem 3.5]{acciaio2019} give us that 
\begin{align*}
J(\hat{\a};\xi_0)-J(\a;\xi_0) 
& \leq
\sE\bigg[\int_0^T
\Big(H(t,X^{\hat{\a}}_t,\hat{\alpha}_t,\sP_{(X^{\hat{\a}}_t,\hat{\alpha}_t)},Y^{\hat{\a}}_t,Z^{\hat{\a}}_t)
-H(\a_t, X^{{\a}}_t,{\alpha}_t,\sP_{(X^{{\a}}_t,{\alpha}_t)},Y^{\hat{\a}}_t,Z^{\hat{\a}}_t)\,\d t
\Big)
\bigg]
\\
&\quad 
-\sE\bigg[\int_0^T
\la
\p_x H(t,X^{\hat{\a}}_t,\hat{\alpha}_t,\sP_{(X^{\hat{\a}}_t,\hat{\alpha}_t)},Y^{\hat{\a}}_t,Z^{\hat{\a}}_t),
X^{\hat{\a}}_t-X^{{\a}}_t
\ra \,\d t
\bigg]
\\
&\quad 
-\sE\bigg[\int_0^T
\tilde{\sE}[\la
\p_\mu H(t,X^{\hat{\a}}_t,\hat{\alpha}_t,\sP_{(X^{\hat{\a}}_t,\hat{\alpha}_t)},Y^{\hat{\a}}_t,Z^{\hat{\a}}_t)
(\tilde{X}^{\hat{\a}}_t,\tilde{\hat{\alpha}}_t),
\tilde{X}^{\hat{\a}}_t-\tilde{X}^{{\a}}_t
\ra] \,\d t
\bigg],
\end{align*} 
which along with  \eqref{eq:mfcE_H_convex}  and 
the fact that $\hat{\alpha}$ satisfies  the optimality condition \eqref{eq:opti}
leads to 
\begin{align*}
J(\hat{\a};\xi_0)-J(\a;\xi_0) 
& \leq
\mathbb{E}
\bigg[
 \int_{0}^{T} \la \p_a H(\theta^{\hat{\a}}_t,Y^{\hat{\a}}_t,Z^{\hat{\a}}_t) + \tilde{\mathbb{E}} [\p_\nu  H(\tilde{\theta}^{\hat{\a}}_t,\tilde{Y}^{\hat{\a}}_t,\tilde{Z}^{\hat{\a}}_t)(X^{\hat{\a}}_t,\hat{\alpha}_t)], \hat{\alpha}_t-a_t \ra \, \d t
 \bigg]
\\
&\quad 
- (\lambda_1+\lambda_2)\|\hat{\a}-{\a}\|_{\cH^2}^2 
\\
& \leq -(\lambda_1+\lambda_2)\|\hat{\a}-{\a}\|_{\cH^2}^2.
\end{align*} 
This finishes the proof of the estimate  \eqref{eq:coercive}.
\end{proof}

\begin{Remark}
It is clear that
 for  each 
 partition $\pi$ of $[0,T]$, $J(\cdot;\xi_0):\cA_\pi\to \sR$ is  continuous and strongly convex,
as $\cA_\pi\subset \cA$ is a convex set. 
Moreover, similar arguments as those in Lemma \ref{lemma:J} show that 
the discrete-time cost functional   $J_\pi(\cdot;\xi_0):\cA_\pi\to \sR$   defined  in \eqref{eq:J_pi}
is continuous and strongly convex. 
Then  the standard theory of strongly convex minimization problems on Hilbert spaces
(see e.g., \cite[Lemma 2.33 (ii)]{bonnans2000}) ensures that 
$J(\cdot;\xi_0):\cA_\pi\to \sR$ and $J_\pi(\cdot;\xi_0):\cA_\pi\to \sR$ admit a unique minimizer.
\end{Remark}

We now show the strong 
convergence of 
 $\eps$-optimal controls of the control problem \eqref{eq:mfcE_constant_c} 
 (with piecewise constant controls but continuous-time  state process).

\begin{Theorem}\l{thm:control_PCPT_c}
Suppose (H.\ref{assum:mfcE}) and (H.\ref{assum:mfcE_hat}) hold,
for every  $\xi_0\in L^2(\cF_0;\sR^n)$,
let $J(\cdot;\xi_0):\cA\to \sR$ be  defined as in \eqref{eq:mfcE}
and
let $\hat{\a}\in\cA$ be the optimal control of \eqref{eq:mfcE}, 
and for each 
 partition $\pi$ of $[0,T]$
 let   $\cA_{\pi}$ be defined as in \eqref{eq:A_pi}
 and
$V_\pi^c(\xi_0)\in \sR$ be defined as in \eqref{eq:mfcE_constant_c}.
Then there exists a constant $C>0$, such that  
for all 
$\xi_0\in L^2(\cF_0;\sR^n)$ and $\eps\ge 0$,  for all partitions
$\pi$  of $[0,T]$ with  stepsize $|\pi|$,
and for all $\a \in \cA_{\pi}$ with 
$J(\a;\xi_0)\le V_{\pi}^c(\xi_0)+\eps$, 
$$
\|\hat{\a}-{\a}\|_{\cH^2}\le C\big((1+\|\xi_0\|_{L^{2}})|\pi|^{1/4} + \sqrt{\eps}\big).
$$
\end{Theorem}

\begin{proof}
Recall that according to Theorem \ref{thm:PCPT_c_value}, we have  $V(\xi_0) - V_{\pi}^c(\xi_0)\le C(1+\|\xi_0\|^2_{L^{2}})|\pi|^{1/2}$,
for a constant $C$ independent of the initial condition and stepsize.
Therefore, by applying the estimate \eqref{eq:coercive}, we have 
\begin{align*}
(\lambda_1+\lambda_2)\|\hat{\a}-{\a}\|_{\cH^2}^2\le J(\a;\xi_0)-J(\hat{\a};\xi_0)
\le V_{\pi}^c(\xi_0)+\eps-J(\hat{\a};\xi_0)\le C(1+\|\xi_0\|^2_{L^{2}})|\pi|^{1/2} + \eps.
\end{align*}
Taking the square root of  both sides of the inequality yields the claim.
\end{proof}

We now establish the  {strong} convergence of 
 $\eps$-optimal controls of the control problem \eqref{eq:mfcE_constant} 
with  piecewise constant controls, state processes and cost functionals.
For simplicity, we only present the result for the case where $\bA$ is a compact subset 
of $\sR^k$, but  refer the reader to Remark \ref{rmk:non_compact} for possible extensions to 
cases with non-compact $\bA$.

\begin{Theorem}\l{thm:control_PCPT_discrete}
Suppose (H.\ref{assum:mfcE}), (H.\ref{assum:mfcE_hat})
and (H.\ref{assum:mfc_Holder_t}) hold,
and $\bA$ is a compact subset of $\sR^k$.
For every $\xi_0\in L^2(\cF_0;\sR^n)$,
 $\hat{\a}\in\cA$ be the optimal control of \eqref{eq:mfcE}, 
and for each 
 partition $\pi$ of $[0,T]$,
 let  $\cA_{\pi}$ be defined as in \eqref{eq:A_pi},
  $J_\pi(\cdot;\xi_0):\cA_\pi\to \sR$ be  defined as in \eqref{eq:J_pi}
and
$V_\pi(\xi_0)\in \sR$ be defined as in \eqref{eq:mfcE_constant}.
Then 
 there exists a constant $C>0$, such that  
for all 
$\xi_0\in L^2(\cF_0;\sR^n)$ and $\eps\ge 0$,  for all partitions
$\pi$  of $[0,T]$ with  stepsize $|\pi|$,
and for all $\a \in \cA_{\pi}$ with 
$J_\pi(\a;\xi_0)\le V_{\pi}(\xi_0)+\eps$, 
$$
\|\hat{\a}-{\a}\|_{\cH^2}\le C\big((1+\|\xi_0\|_{L^{2}})|\pi|^{1/4} + \sqrt{\eps}\big).
$$
\end{Theorem}

\begin{proof}
Recall that 
Step 2 of the proof of Theorem \ref{thm:PCPT_discrete_value} (see \eqref{eq:V-V_pi})
proves that there exists a constant $C>0$, independent of $\xi_0$ and $\pi$, such that for all $\a\in \cA_{\pi}$,
$|J({\a};\xi_0)-J_{\pi}({\a};\xi_0)|\le C(1+\|\xi_0\|^2_{L^2})|\pi|^{1/2}$.
Hence, by  the estimates \eqref{eq:coercive}, for all $\a\in \cA_\pi$ with $J_{\pi}(\a;\xi_0)\le  V_{\pi}(\xi_0)+\eps$,
\begin{align*}
(\lambda_1+\lambda_2)\|\hat{\a}-{\a}\|_{\cH^2}^2 &\le J(\a;\xi_0)-J(\hat{\a};\xi_0)  - J_{\pi}(\a;\xi_0) + J_{\pi}(\a;\xi_0)  \\
& \leq  J(\a;\xi_0)-J(\hat{\a};\xi_0) - J_{\pi}(\a;\xi_0) + V_{\pi}(\xi_0)+\eps  \\
& \leq C(1+\|\xi_0\|^2_{L^{2}})|\pi|^{1/2} + V_{\pi}(\xi_0)-  J(\hat{\a};\xi_0) + \eps\\
& \leq C(1+\|\xi_0\|^2_{L^{2}})|\pi|^{1/2} + \eps,
\end{align*}
where  the last estimate follows from Theorem \ref{thm:PCPT_discrete_value}.
Taking the square root of  both sides of the inequality yields the claim.
\end{proof}

\appendix

\section{Proofs of Propositions \ref{prop:mcfE1} and  \ref{prop:mcfE2} and Lemma \ref{lemma:mono_stab}}
\l{appendix:hat_a_existence}

The following Kantorovich duality theorem,
which plays an important role in the following analysis,
follows as a special case of \cite[Theorem 5.10]{Villani}.

\begin{Lemma}\label{lem:Kant}
Let $(\mathcal{X},\mu)$ and $(\mathcal{Y},\nu)$ be two Polish probability spaces and let $\omega: \mathcal{X} \times \mathcal{Y} \to [0,\infty)$ be a continuous function. Then
we have that
\begin{align*}
\inf_{\kappa \in \Pi(\mu,\nu)} \int_{\mathcal{X} \times \mathcal{Y}} \omega(x,y) \, \mathrm{d} \kappa(x,y) =
 \sup_{
 \substack{
 (\psi,\varphi)\in C_b(\cX)\t C_b(\cY),
 \\\varphi-\psi\le \om}}
   \left( \int_{\mathcal{Y}} \varphi(y)  \, \mathrm{d}\nu(y)  - \int_{\mathcal{X}} \psi(x)  \, \mathrm{d}\mu(x)  \right),
\end{align*}
where $\Pi(\mu,\nu)$ is the set of all couplings of $\mu$ and $\nu$,
and $ C_b(\cX)$ (resp.~$C_b(\cY)$) is the space of bounded continuous functions
$\cX\to \sR$ (resp.~$\cY\to \sR$).
\end{Lemma}

\begin{proof}[Proof of Proposition \ref{prop:mcfE1}]
We first show that the functions 
$(\hat{b},\sigma,\hat{f},\hat{g})$ 
satisfy the Lipschitz continuity.
The claim  obviously holds for the function $\sigma$ due to (H.\ref{assum:mfcE}(\ref{item:mfcE_lin})).
To show the Lipschitz continuity of $\hat{g}$,
 for any $(x,\mu), (x',\mu') \in \sR^n \t \cP_2(\sR^n)$ and any coupling $\kappa$ of $\mu$ and $\mu'$
(i.e., $\kappa\in \Pi(\mu,\mu')$),
we observe from (H.\ref{assum:mfcE}(\ref{item:mfcE_lipschitz}))   that
\begin{align}\l{eq:g_hat_lip}
\begin{split}
&|\hat{g}(x,\mu)-\hat{g}(x',\mu')|^2 
\\
& \leq 2|\p_x g(x,\mu)- \p_x g(x',\mu')|^2  
 +  2 \left| \int_{\sR^n}\p_\mu g(\tilde{x},\mu)(x)\,\d \mu(\tilde{x}) - \int_{\sR^n}\p_\mu g(\tilde{x}',\mu')(x')\,\d \mu'(\tilde{x}')  \right|^2 \\
& \leq 
C (|x-x'|^2
+\cW^2_2(\mu,\mu'))
 + 2 \left( \int_{\sR^n \t \sR^n } |\p_\mu g(\tilde{x},\mu)(x) - \p_\mu g(\tilde{x}',\mu')(x')| \,\d \kappa(\tilde{x},\tilde{x}') \right)^2
\\
& \leq C 
\bigg\{
|x-x'|^2 + \cW^2_2(\mu,\mu')+ \left( \int_{\sR^n \t \sR^n } 
\Big(
|\tilde{x}-\tilde{x}'|+|x-x'|+\cW_2(\mu,\mu')
\Big) \,\d \kappa(\tilde{x},\tilde{x}') \right)^2
\bigg\},
\end{split}
\end{align}
where $C>0$ depends on the Lipschitz constant  in (H.\ref{assum:mfcE}(\ref{item:mfcE_lipschitz})). 
Then, by applying Jensen's inequality to the above estimate
and taking the infimum   over all  $\kappa\in \Pi(\mu,\mu')$,
we can deduce that 
$|\hat{g}(x,\mu)-\hat{g}(x',\mu')| \leq C (|x-x'| + \cW_2(\mu,\mu'))$.

Before proceeding to show the Lipschitz continuity of $\hat{b}$ and $\hat{f}$, we first establish the Lipschitz continuity of $\phi(t,\chi)$ defined as in (H.\ref{assum:mfcE_hat}(\ref{item:ex})).
Let $\chi,\chi' \in \cP_2(\sR^{n} \t \sR^{n})$ be given. By applying Lemma \ref{lem:Kant} with $\mathcal{X}=\cY=\sR^{n} \t \sR^{k}$, $\nu=\phi(t,\chi)$, $\mu=\phi(t,\chi')$ and 
the function
$\omega\big((x',y'),(x,y)\big) \coloneqq |x-x'|^2 + |y-y'|^2$ for any $(x,y)\in \cY, (x',y')\in \cX$, 
we can obtain from the definition of $\phi$ that 
\begin{align*}
& \cW_2^{2}(\phi(t,\chi),\phi(t,\chi'))  \\
& = \sup \left( \int_{\sR^{n} \t \sR^{k}} h_1(x,y) \, \mathrm{d}\phi(t,\chi) (x,y) - \int_{\sR^{n} \t \sR^{k}} h_2(x',y') \, \mathrm{d}\phi(t,\chi') (x',y') \right)  \\
& =  \sup \left( \int_{\sR^{n} \t \sR^{n}} h_1(x,\hat{\a}(t,x,y,\chi)) \, \mathrm{d} \chi(x,y) - \int_{\sR^{n} \t \sR^{n}} h_2(x',\hat{\a}(t,x',y',\chi')) \, \mathrm{d} \chi'(x',y') \right),
\end{align*}
where the supremum is taken over all bounded continuous functions $h_1, h_2: \sR^{n} \t \sR^{k} \to \sR$ satisfying $h_1(x,y)- h_2(x',y') \leq |x-x'|^2 + |y-y'|^2$ for any $(x,y), (x',y') \in \sR^{n} \t \sR^{k}$. Note that for any given 
such functions $h_1,h_2$,  the Lipschitz continuity of $\hat{\a}$ in 
  (H.\ref{assum:mfcE_hat}(\ref{item:ex})) implies that
\begin{align*}
\begin{split}
&h_1(x,\hat{\a}(t,x,y,\chi))  - h_2(x',\hat{\a}(t,x',y',\chi')) 
\\
&\leq  (3L^2_{\a}+1) \left( |x-x'|^2 + |y-y'|^2 + \cW_2^{2}(\chi,\chi') \right)
\coloneqq \omega_2\big((x,y),(x',y')\big).
\end{split}
\end{align*}
Hence, another application of Lemma \ref{lem:Kant} with 
$\mathcal{X}=\cY=\sR^{n} \t \sR^{n}$, $\nu=\chi$, $\mu=\chi'$ and $\om=\om_2$ gives us that
\begin{align*}
\cW_2^{2}(\phi(t,\chi),\phi(t,\chi'))
&\le
\sup \left( \int_{\sR^{n} \t \sR^{n}} \tilde{h}_1(x,y) \, \mathrm{d} \chi(x,y) - \int_{\sR^{n} \t \sR^{n}} \tilde{h}_2(x',y') \, \mathrm{d} \chi'(x',y') \right) \\
& = \inf_{\kappa  \in \Pi(\chi',\chi)}  \int_{(\sR^n \t \sR^{n}) \t (\sR^n \t \sR^{n})}  \omega_2(x,y)
\, \mathrm{d} \kappa(x,y),
\end{align*}
where the supremum is 
 taken over all bounded continuous functions $\tilde{h}_1, \tilde{h}_2: \sR^{n} \t \sR^{n} \to\sR$ satisfying
 $\tilde{h}_1-\tilde{h}_2\le \om_2$. Thus, we readily deduce from the above estimate that
\begin{align}\label{eq:PF}
\cW_2(\phi(t,\chi),\phi(t,\chi')) \leq C \cW_2(\chi,\chi'),
\end{align}
with a constant  $C>0$ depending only on $L_{\a}$. 

Now  for any $(x,y,\chi), (x',y',\chi') \in \sR^n\t \sR^n \t \cP_2(\sR^n \t \sR^{n})$, 
we can obtain from \eqref{eq:mfc_coefficients2},  (H.\ref{assum:mfcE}(\ref{item:mfcE_lin})) and the Lipschitz continuity of $\hat{\a}$ in (H.\ref{assum:mfcE_hat}(\ref{item:ex}))  that
\begin{align*}
 &|\hat{b}(t,x,y,\chi) - \hat{b}(t,x',y',\chi')| 
 \\
 &= |b(t,x,\hat{\a}(t,x,y,\chi),\phi(t,\chi)) - b(t,x',\hat{\a}(t,x',y',\chi'),\phi(t,\chi'))| \\
& \leq C \left( |x-x'| + |y-y'| + \cW_2(\chi,\chi') \right),
\end{align*}
which shows the Lipschitz continuity of $\hat{b}$.
Finally, we shall establish the Lipschitz continuity of $\hat{f}$. 
Observe  that $\p_x H(t,\cdot)$  and $\p_\mu H(t,\cdot)(\cdot,\cdot)$ are Lipschitz continuous (uniformly in $t$), 
which follows from the definition  of the Hamiltonian $H$ in \eqref{eq:mfcE_hamiltonian} and 
(H.\ref{assum:mfcE}(\ref{item:mfcE_lin})(\ref{item:mfcE_lipschitz})). Now, for any $(x,y,z,\rho), (x',y',z',\rho') \in \sR^n \t \sR^n \t \sR^{n \times d} \t \cP_2(\sR^n \t \sR^{n} \t  \sR^{n \t d})$,
let $\chi=\pi_{1,2} \sharp \rho$ (resp.~$\chi'=\pi_{1,2} \sharp \rho'$)
  the marginal of the measure $\rho$ (resp.~$\rho'$) on $\sR^n\t\sR^n$. Then,
  we can obtain from the definition of $\hat{f}$ in \eqref{eq:mfc_coefficients2} that 
\begin{align*}
& |\hat{f}(t,x,y,z,\rho) - \hat{f}(t,x',y',z',\rho')| \\
& \leq  | \p_x H(t,x,\hat{\a}(t,x,y,\chi),\phi(t,\chi),y,z) -  \p_x H(t,x',\hat{\a}(t,x',y',\chi'),\phi(t,\chi'),y',z')|  \\
& \quad +   \Bigg |\int_{\sR^n\t \sR^n\t \sR^{n\t d}}
\p_\mu H(t,\tilde{x},\hat{\a}(t,\tilde{x},\tilde{y},\chi),\phi(t,\chi),\tilde{y},\tilde{z})
(x,\hat{\a}(t,{x},{y},\chi))\,\d \rho(\tilde{x},\tilde{y},\tilde{z}) \\
& \quad - \int_{\sR^n\t \sR^n\t \sR^{n\t d}}
\p_\mu H(t,\tilde{x}',\hat{\a}(t,\tilde{x}',\tilde{y}',\chi'),\phi(t,\chi'),\tilde{y}',\tilde{z}')
(x',\hat{\a}(t,{x}',{y}',\chi'))\,\d \rho'(\tilde{x}',\tilde{y}',\tilde{z}') \Bigg| \\
& \coloneqq\Sigma_1 + \Sigma_2. 
\end{align*}
Note that one can easily deduce from Lemma \ref{lem:Kant}
that 
$\cW_2(\pi_{1,2} \sharp  \rho, \pi_{1,2} \sharp \rho') \leq \cW_2(\rho,\rho')$.
Then, by using 
the uniform Lipschitz continuity of $\p_x H(t,\cdot)$,
 (H.\ref{assum:mfcE_hat}(\ref{item:ex})) and
 (\ref{eq:PF}),
 we have the estimate that 
$\Sigma_1 \leq C \left( |x-x'|+ |y-y'|+ |z-z'| + \cW_2(\rho,\rho') \right)$.
 Furthermore, 
by using the same manipulations as in \eqref{eq:g_hat_lip} with  an arbitrary coupling of $\rho$ and  $\rho'$,
and employing  the Lipschitz continuity of $\p_\mu H(t,\cdot)(\cdot,\cdot)$
 along with (H.\ref{assum:mfcE_hat}(\ref{item:ex})) and (\ref{eq:PF}), 
 we can 
conclude the same upper bound for $\Sigma_2$, 
which leads to the desired Lipschitz continuity of $\hat{f}$.

It remains to show that the functions 
$(\hat{b},\sigma,\hat{f})$ 
satisfy the integrability conditions.
We can clearly see from
(H.\ref{assum:mfcE}(\ref{item:mfcE_lin}))
 that 
$\|\sigma(\cdot,0,\bm{\delta}_{0_{n}})\|_{L^\infty(0,T)}=
\|\sigma_0\|_{L^\infty(0,T)}<\infty$.
Moreover, 
 \eqref{eq:PushF} 
 and  (H.\ref{assum:mfcE_hat}(\ref{item:ex})) 
 imply that 
 $\|\phi(t,\bm{\delta}_{0_{n+n}})\|_2=\|\bm{\delta}_{(0,\hat{\a}(t,0,0,\bm{\delta}_{0_{n+n}}))}\|_2 
\le \|\hat{\a}(\cdot,0,0,\bm{\delta}_{0_{n+n}})\|_{L^\infty(0,T)}<\infty$
for all $t\in [0,T]$.
Hence, we can obtain from 
 \eqref{eq:mfc_coefficients2}  and
 (H.\ref{assum:mfcE}(\ref{item:mfcE_lin})(\ref{item:mfcE_lipschitz}))  
that
$ \|\hat{b}(\cdot,0,0,\bm{\delta}_{0_{n+n}})\|_{L^2(0,T)}+
 \| \hat{f}(\cdot,0,0,0,\bm{\delta}_{0_{n+n+nd}})\|_{L^\infty(0,T)}
 <\infty$,
 which completes the proof.
\end{proof}

\begin{proof}[Proof of Proposition \ref{prop:mcfE2}]
Throughout this proof,
let $t\in [0,T]$, 
for all $i\in \{1,2\}$
let 
$\Theta_i=(X_i,Y_i,Z_i)\in L^2(\Om;\sR^n\t \sR^n\t\sR^{n\t d})$
be a given random variable
and
 $\a_i=\hat{\a}(t,X_i,Y_i,\sP_{(X_i,Y_i)})$.

%
%
Let $(\tilde{X}_i,\tilde{Y}_i, \tilde{Z}_i)_{i=1}^2$ be an independent copy of 
$({X}_i,{Y}_i, {Z}_i)_{i=1}^2$ defined on the space 
$L^2(\tilde{\Om},\tilde{\cF},\tilde{\sP})$.
By applying the  convexity of $g$ in  (H.\ref{assum:mfcE}(\ref{item:mfcE_convex}))
with $(x',\mu')=(X_1(\om),\sP_{X_1})$, $(x,\mu)=(X_2(\om),\sP_{X_2})$ for each $\om$,
taking the expectation with respect to the measure $\sP$
and then exchanging the role of $X_1$ and $X_2$ in the estimates,
we obtain the desired monotonicity property of $\hat{g}$ 
in \eqref{eq:monotonicity} as follows:
\begin{align*}
0&\leq 
\sE[\la \p_x g(X_1,\sP_{X_1})-\p_x g(X_2,\sP_{X_2}),X_1-X_2 \ra]
\\
& \q +
\sE\big[ \tilde{\sE}[
\la \p_\mu g({X}_1,\sP_{X_1})(\tilde{X}_1)-\p_\mu g({X}_2,\sP_{X_1})(\tilde{X}_2),
\tilde{X}_1-\tilde{X}_2 \ra]\big]
\\
&=\sE[\la \hat{g}(X_1,\sP_{X_1})-\hat{g}(X_2,\sP_{{X}_2}), X_1-X_2\ra],
\end{align*}
where for the last equality we have used Fubini's theorem and the fact that 
$\sP_{{X}_i}=\tilde{\sP}_{\tilde{X}_i}$ for $i=1,2$.


To show monotonicity of $\hat{f}$, 
we first deduce from 
the definition of $\hat{b}$ (see \eqref{eq:mfc_coefficients2})
and
the linearity of $H$ in $(y,z)$ (see \eqref{eq:mfcE_hamiltonian})
that
\begin{align*}
\begin{split}
&\la \hat{b}(t,X_1,Y_1,\sP_{(X_1,Y_1)})-\hat{b}(t,X_2,Y_2,\sP_{(X_2,Y_2)}),  Y_1-Y_2\ra
\\
&\q +\la \sigma(t,X_1,\sP_{X_1})-\sigma(t,X_2,\sP_{X_2}),  Z_1-Z_2\ra
\\
&=\la {b}(t,X_1,\a_1,\sP_{(X_1,\a_1)}),  Y_1-Y_2\ra
+\la \sigma(t,X_1,\sP_{X_1}),  Z_1-Z_2\ra
\\
&\q -
\big({b}(t,X_2,\a_2,\sP_{(X_2,\a_2)}),  Y_1-Y_2\ra
+\la \sigma(t,X_2,\sP_{X_2}), Z_1- Z_2\ra
\big)
\\
&=H(t,X_1,\a_1,\sP_{(X_1,\a_1)},Y_1,Z_1)-H(t,X_1,\a_1,\sP_{(X_1,\a_1)},Y_2,Z_2)
\\
&\q
-\big(H(t,X_2,\a_2,\sP_{(X_2,\a_2)},Y_1,Z_1)-H(t,X_2,\a_2,\sP_{(X_2,\a_2)},Y_2,Z_2)\big).
\end{split}
\end{align*}
Moreover, by 
setting 
$\tilde{\a}_i=\hat{\a}(t,\tilde{X}_i,\tilde{Y}_i,\sP_{(X_i,Y_i)})$ for  all $i=1,2$
and
using the definition of $\hat{f}$ 
in \eqref{eq:mfc_coefficients2},
  we can obtain that
\begin{align*}
\begin{split}
&\sE\big[ \la -\hat{f}(t,\Theta_1,\sP_{\Theta_1})+\hat{f}(t,\Theta_2,\sP_{\Theta_2}), X_1-X_2\ra\big]
\\
&=-\sE\big[\la \p_x H(t,X_1,\a_1,\sP_{(X_1,\a_1)},Y_1,Z_1), {X}_1-{X}_2\ra\big]
\\
&\q
-\sE\big[ \tilde{\sE}[\la
\p_\mu H(t,X_1,\a_1,\sP_{(X_1,\a_1)},Y_1,Z_1)(\tilde{X}_1,\tilde{\a}_1), \tilde{X}_1-\tilde{X}_2\ra]
\big]
\\
&\q
+\sE\big[\la \p_x H(t,X_2,\a_2,\sP_{(X_2,\a_2)},Y_2,Z_2), {X}_1-{X}_2\ra\big]
\\
&\q
+\sE\big[\tilde{\sE}[
\la
\p_\mu H(t,X_2,\a_2,\sP_{(X_2,\a_2)},Y_2,Z_2)(\tilde{X}_2,\tilde{\a}_2), \tilde{X}_1-\tilde{X}_2\ra]
\big],
\end{split}
\end{align*}
where
we have also applied
 Fubini's theorem
and the fact
 that
 $\sP_{({X}_i,{Y}_i,{Z}_i,{\a}_i)}=\tilde{\sP}_{(\tilde{X}_i,\tilde{Y}_i,\tilde{Z}_i,\tilde{\a}_i)}$ for all $i=1,2$.

Therefore, we can conclude from \eqref{eq:mfcE_H_convex} that
\begin{align*}
\begin{split}
&\sE\big[\la \hat{b}(t,X_1,Y_1,\sP_{(X_1,Y_1)})-\hat{b}(t,X_2,Y_2,\sP_{(X_2,Y_2)}),  Y_1-Y_2\ra
\\
&\q+\la \sigma(t,X_1,\sP_{X_1})-\sigma(t,X_2,\sP_{X_2}),  Z_1-Z_2\ra
+ \la -\hat{f}(t,\Theta_1,\sP_{\Theta_1})+\hat{f}(t,\Theta_2,\sP_{\Theta_2}), X_1-X_2\ra\big]
\\
&=\sE\Big[H(t,X_1,\a_1,\sP_{(X_1,\a_1)},Y_1,Z_1)-H(t,X_2,\a_2,\sP_{(X_2,\a_2)},Y_1,Z_1)
%
\\
&\quad
-\la \p_x H(t,X_1,\a_1,\sP_{(X_1,\a_1)},Y_1,Z_1), {X}_1-{X}_2\ra
 \\
& \quad - \tilde{\sE}[\la
\p_\mu H(t,X_1,\a_1,\sP_{(X_1,\a_1)},Y_1,Z_1)(\tilde{X}_1,\tilde{\a}_1), \tilde{X}_1-\tilde{X}_2\ra]
\Big]
\\
&\q
-\sE\Big[H(t,X_1,\a_1,\sP_{(X_1,\a_1)},Y_2,Z_2) - H(t,X_2,\a_2,\sP_{(X_2,\a_2)},Y_2,Z_2)
\\
&\quad
-\la \p_x H(t,X_2,\a_2,\sP_{(X_2,\a_2)},Y_2,Z_2), {X}_1-{X}_2\ra
\\
& \quad -\tilde{\sE}[
\la
\p_\mu H(t,X_2,\a_2,\sP_{(X_2,\a_2)},Y_2,Z_2)(\tilde{X}_2,\tilde{\a}_2), \tilde{X}_1-\tilde{X}_2\ra]
\Big]
\\
&\le 
-2(\lambda_1 +\lambda_2) \sE[|\a_1-\a_2|^2], 
\end{split}
\end{align*}
where 
we have applied
Fubini's theorem,
\eqref{eq:opti_markov},
 and the definitions of 
$(\a_1,\a_2)$
  to derive the last estimate.
This shows the desired monotonicity property of $\hat{f}$ 
and completes the proof.
\end{proof}

\begin{proof}[Proof of Lemma \ref{lemma:mono_stab}]
{
For ease of notation, we will write $b,f,g$ instead of $\hat{b}, \hat{f}, \hat{g}$.
Also, throughout this proof, let 
$\delta \xi=\xi-\bar{\xi}$,
$\delta \cI^g_T=\cI^g_T-\bar{\cI}^g_T$,
  $g(X_T)=g(X_T,\sP_{X_T})$,
$g(\bar{X}_T)=g(\bar{X}_T,\sP_{\bar{X}_T})$ 
and $\bar{g}(\bar{X}_T)=\bar{g}(\bar{X}_T,\sP_{\bar{X}_T})$,
for each $t\in [0,T]$
 let 
 $\delta\cI^b_t=\cI^b_t-\bar{\cI}^b_t$,
  $\delta\cI^\sigma_t=\cI^\sigma_t-\bar{\cI}^\sigma_t$,
   $\delta\cI^f_t=\cI^f_t-\bar{\cI}^f_t$,
  $f(\Theta_t)=f(t,X_t,Y_t,\sP_{\Theta_t})$,
$f(\bar{\Theta}_t)=f(t,\bar{X}_t,\bar{Y}_t,\sP_{\bar{\Theta}_t})$
and $\bar{f}(\bar{\Theta}_t)=\bar{f}(t,\bar{X}_t,\bar{Y}_t,\sP_{\bar{\Theta}_t})$.
Similarly, we introduce the notation
$\sigma(X_t), \sigma(\bar{X}_t), \bar{\sigma}(\bar{X}_t)$ and $b(X_t,Y_t), b(\bar{X}_t,\bar{Y}_t), \bar{b}(\bar{X}_t,\bar{Y}_t)$ for $t\in [0,T]$.

By applying It\^{o}'s formula to $\la Y_t-\bar{Y}_t,X_t-\bar{X}_t\ra$,  we obtain that 
\begin{align*}
&\sE[\la \lambda_0(g(X_T)-\bar{g}(\bar{X}_T))+\delta I^g_T,X_T-\bar{X}_T \ra]
-\sE[\la Y_0-\bar{Y}_0, \delta \xi\ra]
\\
&=\sE\bigg[\int_0^T
\la \lambda_0(b(X_t,Y_t)-\bar{b}(\bar{X}_t,\bar{Y}_t))+\delta \cI^b_t,Y_t-\bar{Y}_t\ra 
+
\la \lambda_0(\sigma(X_t)-\bar{\sigma}(\bar{X}_t))+\delta \cI^\sigma_t,Z_t-\bar{Z}_t \ra
\\
&\quad 
+
\la -\big(\lambda_0(f(\Theta_t)-\bar{f}(\bar{\Theta}_t))+\delta \cI^f_t\big),X_t-\bar{X}_t \ra
\, \d t
\bigg].
\end{align*}
Then, 
by adding and subtracting the terms
$g(\bar{X}_T), b(\bar{X}_t,\bar{Y}_t),\sigma(\bar{X}_t),f(\bar{\Theta}_t)$
and applying the monotonicity property established in Proposition \ref{prop:mcfE2}, we can deduce that
\begin{align*}
&
\sE[\la \lambda_0(g(\bar{X}_T)-\bar{g}(\bar{X}_T))+\delta I^g_T,X_T-\bar{X}_T\ra]
-\sE[\la Y_0-\bar{Y}_0,\delta \xi\ra]
\\
&\le 
\sE\bigg[\int_0^T
\la \lambda_0(b(\bar{X}_t,\bar{Y}_t)-\bar{b}(\bar{X}_t,\bar{Y}_t)+\delta \cI^b_t, Y_t-\bar{Y}_t\ra 
+
\la \lambda_0(\sigma(\bar{X}_t)-\bar{\sigma}(\bar{X}_t))+\delta \cI^\sigma_t,Z_t-\bar{Z}_t \ra
\\
&\quad 
+
\la -\big(\lambda_0(f(\bar{\Theta}_t)-\bar{f}(\bar{\Theta}_t))+\delta \cI^f_t\big),X_t-\bar{X}_t\ra
\, \d t
\bigg]-2(\lambda_1 + \lambda_2)\lambda_0\int_0^T \phi_1(t,X_t,Y_t,\bar{X}_t,\bar{Y}_t) \, \d t,
\end{align*}
with $\phi_1(t,X_t,Y_t,\bar{X}_t,\bar{Y}_t) := \| \hat{\alpha}(t,X_t,Y_t,\mathbb{P}_{(X_t,Y_t)}) - \hat{\alpha}(t,\bar{X}_t, \bar{Y}_t,\mathbb{P}_{(\bar{X}_t,\bar{Y}_t)})  \|^2_{L_2}$,
which together with Young's inequality yields for each $\eps>0$ that
\begin{align}\l{eq:mono_stab_decouple}
\begin{split}
&
2(\lambda_1 + \lambda_2)\lambda_0 \int_0^T
 \phi_1(t,X_t,Y_t,\bar{X}_t,\bar{Y}_t) \, \d t
  \\
  &
  \le
\eps( \|X_T-\bar{X}_T\|_{L^2}^2+\|Y_0-\bar{Y}_0\|_{L^2}^2
 +\|\Theta-\bar{\Theta}\|_{\cH^2}^2)
 +C{\eps}^{-1}\textrm{RHS},
 \end{split}
\end{align}
where $\textrm{RHS}$ denotes the right-hand side of \eqref{eq:mono_stab}.

Now, by \eqref{eq:mono_stab_decouple} and the fact that $\lambda_1 + \lambda_2 >0$,
we have  for all $\eps>0$,
\begin{align}\l{eq:mono_stab_decouple_m<n}
\begin{split}
&
\lambda_0\int_0^T
  \phi_1(t,X_t,Y_t,\bar{X}_t,\bar{Y}_t)\, \d t
  \le
\eps( \|X-\bar{X}\|_{\cS^2}^2+\|Y-\bar{Y}\|_{\cS^2}^2
 +\|Z-\bar{Z}\|_{\cH^2}^2)
 +C{\eps}^{-1}\textrm{RHS}.
 \end{split}
\end{align}
Then, by using the Burkholder-Davis-Gundy  inequality, the definition of \eqref{eq:moc} and \eqref{eq:mfc_coefficients2}, Gronwall's inequality
and the fact that $\lambda_0\in [0,1]$, we can 
deduce that 
\begin{align*}
\|X-\bar{X}\|_{\cS^2}^2
&\le 
C\bigg(
\int_0^T \lambda_0\phi_1(t,X_t,Y_t,\bar{X}_t,\bar{Y}_t)\, \d t
+
\|\xi-\bar{\xi}\|_{L^2}^2
\\
&\quad
+\|\lambda_0 (b(\bar{X},\bar{Y})-\bar{b}(\bar{X},\bar{Y})
+\delta \cI^b\|_{\cH^2}^2
+\|\lambda_0 (\sigma(\bar{X})-\bar{\sigma}(\bar{X}))
+\delta \cI^\sigma\|_{\cH^2}^2
\bigg),
\end{align*}
which together with \eqref{eq:mono_stab_decouple_m<n} yields for all small enough $\eps>0$ that 
\begin{align*}
\begin{split}
&
\|X-\bar{X}\|_{\cS^2}^2
  \le
\eps(\|Y-\bar{Y}\|_{\cS^2}^2
 +\|Z-\bar{Z}\|_{\cH^2}^2)
 +C{\eps}^{-1}\textrm{RHS}.
 \end{split}
\end{align*}
Moreover,  by standard estimates for MV-BSDEs, we can obtain that
\begin{align*}
&\|Y-\bar{Y}\|_{\cS^2}^2
 +\|Z-\bar{Z}\|_{\cH^2}^2
 \\
&  \le
  C\bigg(\|X-\bar{X}\|_{\cS^2}^2
  +
  \|\lambda_0 (g(\bar{X}_T)-\bar{g}(\bar{X}_T))+\delta \cI^g_T\|_{L^2}^2
  +\|\lambda_0 (f(\bar{\Theta})-\bar{f}(\bar{\Theta}))
+\delta\cI^f\|_{\cH^2}^2
\bigg),
\end{align*}
which completes the desired estimate \eqref{eq:mono_stab}.}
\end{proof}

\subsection*{Acknowledgements}

Wolfgang Stockinger is supported by a special Upper Austrian  Government grant.


\begin{thebibliography}{1}

\bibitem{acciaio2019}
B.~Acciaio, J.~Backhoff-Veraguas, and R.~Carmona,
\emph{Extended mean field control problems: stochastic maximum
principle and transport perspective}, 
SIAM J. Control Optim., 57 (2019), pp. 3666--3693.


\bibitem{achdou2020}
Y.~Achdou and Z.~Kobeissi, \emph{Mean field games of controls: Finite difference approximations},
Mathematics in Engineering,
3 (2020), pp.~1--35.



\bibitem{bandini2016}
E.~Bandini, A.~Cosso, M.~Fuhrman, and H.~Pham, 
\emph{Randomized filtering and
Bellman equation in Wasserstein space for partial observation control problem}, 
arXiv preprint, arXiv:1609.02697,
2016.






\bibitem{bensoussan2015}
A.~Bensoussan, S.~Yam, and Z.~Zhang, \emph{Well-posedness of mean-field type forward-backward stochastic differential equations}, Stochastic Process. Appl., 125  (2015),
pp. 3327--3354.

\bibitem{bonnans2000}
J. F. Bonnans and A. Shapiro, \emph{Perturbation analysis of optimization problems}, Springer-Verlag, New York, 2000.

\bibitem{burzoni2020}
M.~Burzoni, V.~Ignazio, H.~Soner, and A.~Reppen, \emph{Viscosity solutions for controlled McKean-Vlasov
jump-diffusions},
SIAM J. Control Optim., 
58 (2020),  pp. 1676--1699.

\bibitem{carmona2015}
R.~Carmona and F.~Delarue, \emph{Forward-backward stochastic differential equations and controlled McKean-Vlasov dynamics},
Ann. Probab., 43 (2015), pp.~2647--2700. 
  



\bibitem{carmona2018a} 
R.~Carmona and F.~Delarue, \emph{Probabilistic theory of mean field games with applications I: Mean-field FBSDEs, control, and games}, Springer International Publishing, Switzerland, 2018.
 

\bibitem{carmona2019}
R.~Carmona and M.~Lauri\`{e}re, 
\emph{Convergence analysis of machine learning algorithms for the numerical solution of mean field control and games: II--The finite horizon case},  arXiv preprint, arXiv:1908.01613, 2019.
 
 
 
 \bibitem{chassagneux2014}
J.~F.~Chassagneux, D.~Crisan, and F.~Delarue, 
\emph{A probabilistic approach to
classical solutions of the master equation for large population equilibria}, 
Mem. Amer. Math. Soc.,
(2020),
Available at
 arXiv:1411.3009.
 
 


\bibitem{dumitrescu2019}
R.~Dumitrescu, C.~Reisinger, and Y.~Zhang, 
\emph{
Approximation schemes for mixed optimal stopping and control problems with nonlinear expectations and jumps}, 
Appl. Math. Optim., (2019), to appear.


%

\bibitem{gobet2019}
E.~Gobet and M.~Grangereau,
\emph{Extended McKean-Vlasov optimal stochastic control applied
to smart grid management},
Available at hal-02181227, 2019.


\bibitem{gu2019}
H.~Gu, X.~Guo, X.~Wei, and R.~Xu,
\emph{Dynamic programming principles for learning MFCs},
arXiv preprint,
arXiv:1911.07314, 2019.

%

\bibitem{jakobsen2019}
E.~R.~Jakobsen, A.~Picarelli, and C.~Reisinger,
 \emph{Improved order $1/4$ convergence of Krylov's piecewise constant
policy approximation}, Electron. Comm. Probab., 24 (2019), pp.~1--10.
 

\bibitem{krylov1999} 
N.~V.~Krylov,
\emph{Approximating value functions for controlled degenerate diffusion processes by using piece-wise
constant policies}, Electron. J. Probab., 4 (1999), pp.~1--19.

%
%



\bibitem{lauriere2020}
M.~Lauri\`{e}re and L.~Tangpi,
\emph{Convergence of large population games to mean field games with interaction through
controls}, arXiv preprint, arXiv:2004.08351, 2020.




\bibitem{picarelli2020}
A. Picarelli and C. Reisinger, 
\emph{Probabilistic error analysis for some approximation schemes to optimal control problems}, 
Systems Control Lett., 137 (2020), 104619.


\bibitem{pham2016}
H.~Pham and X.~Wei, \emph{Discrete time McKean-Vlasov control problem: a dynamic programming
approach},  Appl. Math. Optim., 74 (2016), 487--506.

\bibitem{pham2018}
H.~Pham and X.~Wei, \emph{Bellman equation and viscosity solutions for mean-field stochastic
control problem}, ESAIM Control Optim. Calc. Var., 24 (2018), pp. 437--461.





\bibitem{reis2019}
G.~dos Reis, W.~Salkeld, and J.~Tugaut, \emph{Freidlin-Wentzell LDP in path space for McKean-Vlasov equations and the
functional iterated logarithm law}, Ann. Appl. Probab., 29 (2019), pp. 1487--1540.

\bibitem{reisinger2016}
C.~Reisinger and P.A.~Forsyth, 
\emph{Piecewise constant policy approximations to Hamilton--Jacobi--Bellman equations},
 Appl. Numer. Math., 103, (2016), pp.~27--47.



%

\bibitem{Villani} 
C.~Villani, \emph{Optimal Transport: Old and New}, Springer--Verlag, Berlin, 2009.







\bibitem{zhang2017}
J.~Zhang, \emph{Backward Stochastic Differential Equations: From Linear to Fully Nonlinear
Theory}, vol. 86., Springer, New York, 2017.

\end{thebibliography}

\end{document}